\pdfoutput=1
\documentclass{amsart} 
\newif\ifpersonal
\personalfalse
\usepackage{amsmath,amscd,amsthm,amssymb,mathrsfs,mathtools,bm} 
\usepackage{microtype,lmodern,textcomp} 
\usepackage{enumitem,comment,braket,xspace,tikz-cd} 
\usepackage[utf8]{inputenc} 
\usepackage[T1]{fontenc} 
\usepackage[centering,vscale=0.7,hscale=0.7]{geometry}
\usepackage[hidelinks]{hyperref}
\usepackage[capitalize]{cleveref}
\usepackage{xcolor}
\usepackage{stmaryrd}
\usepackage{graphicx}
\usepackage{accents}
\usepackage{enumitem}
\usepackage{subcaption}
\usepackage[all]{xy}
\usepackage{arydshln}
\numberwithin{equation}{section}
\theoremstyle{plain}
\newtheorem{theorem}[equation]{Theorem}
\newtheorem*{theorem*}{Theorem}
\newtheorem{lemma}[equation]{Lemma}
\newtheorem*{lemma*}{Lemma}

\newtheorem*{claim*}{Claim}
\newtheorem{conjecture}[equation]{Conjecture}
\newtheorem{proposition}[equation]{Proposition}
\newtheorem*{proposition*}{Proposition}
\newtheorem{corollary}[equation]{Corollary}
\newtheorem*{corollary*}{Corollary}
\theoremstyle{definition}
\newtheorem{definition}[equation]{Definition}
\newtheorem*{definition*}{Definition}
\newtheorem{definition-theorem}[equation]{Definition-Theorem}
\newtheorem{definition-lemma}[equation]{Definition-Lemma}

\newtheorem{assumption}[equation]{Assumption}

\newtheorem{example}[equation]{Example}
\newtheorem{remark}[equation]{Remark}
\newtheorem*{remark*}{Remark}

\numberwithin{equation}{section}

\ifpersonal
\newcommand{\personal}[1]{\textcolor[rgb]{0,0,1}{(Personal: #1)}}

\newcommand{\todo}[1]{\textcolor{red}{(Todo: #1)}}
\else
\newcommand{\personal}[1]{\ignorespaces}
\newcommand{\discussion}[1]{\ignorespaces}
\newcommand{\todo}[1]{\ignorespaces}
\fi

\providecommand{\abs}[1]{\lvert#1\rvert}

\newcommand{\bbC}{\mathbb C}

\newcommand{\bbF}{\mathbb F}

\newcommand{\bbN}{\mathbb N}

\newcommand{\bbP}{\mathbb P}
\newcommand{\bbQ}{\mathbb Q}
\newcommand{\bbR}{\mathbb R}

\newcommand{\bbZ}{\mathbb Z}

\newcommand{\cF}{\mathcal F}
\newcommand{\cH}{\mathcal H}

\newcommand{\cI}{\mathcal I}
\newcommand{\cJ}{\mathcal J}

\newcommand{\cL}{\mathcal L}
\newcommand{\cM}{\mathcal M}

\newcommand{\cO}{\mathcal O}

\newcommand{\cQ}{\mathcal Q}

\newcommand{\cS}{\mathcal S}

\makeatletter
\let\save@mathaccent\mathaccent
\newcommand*\if@single[3]{%
	\setbox0\hbox{${\mathaccent"0362{#1}}^H$}%
	\setbox2\hbox{${\mathaccent"0362{\kern0pt#1}}^H$}%
	\ifdim\ht0=\ht2 #3\else #2\fi
}
\newcommand*\rel@kern[1]{\kern#1\dimexpr\macc@kerna}
\newcommand*\widebar[1]{\@ifnextchar^{{\wide@bar{#1}{0}}}{\wide@bar{#1}{1}}}
\newcommand*\wide@bar[2]{\if@single{#1}{\wide@bar@{#1}{#2}{1}}{\wide@bar@{#1}{#2}{2}}}
\newcommand*\wide@bar@[3]{%
	\begingroup
	\def\mathaccent##1##2{%
		\let\mathaccent\save@mathaccent
		\if#32 \let\macc@nucleus\first@char \fi
		\setbox\z@\hbox{$\macc@style{\macc@nucleus}_{}$}%
		\setbox\tw@\hbox{$\macc@style{\macc@nucleus}{}_{}$}%
		\dimen@\wd\tw@
		\advance\dimen@-\wd\z@
		\divide\dimen@ 3
		\@tempdima\wd\tw@
		\advance\@tempdima-\scriptspace
		\divide\@tempdima 10
		\advance\dimen@-\@tempdima
		\ifdim\dimen@>\z@ \dimen@0pt\fi
		\rel@kern{0.6}\kern-\dimen@
		\if#31
		\overline{\rel@kern{-0.6}\kern\dimen@\macc@nucleus\rel@kern{0.4}\kern\dimen@}%
		\advance\dimen@0.4\dimexpr\macc@kerna
		\let\final@kern#2%
		\ifdim\dimen@<\z@ \let\final@kern1\fi
		\if\final@kern1 \kern-\dimen@\fi
		\else
		\overline{\rel@kern{-0.6}\kern\dimen@#1}%
		\fi
	}%
	\macc@depth\@ne
	\let\math@bgroup\@empty \let\math@egroup\macc@set@skewchar
	\mathsurround\z@ \frozen@everymath{\mathgroup\macc@group\relax}%
	\macc@set@skewchar\relax
	\let\mathaccentV\macc@nested@a
	\if#31
	\macc@nested@a\relax111{#1}%
	\else
	\def\gobble@till@marker##1\endmarker{}%
	\futurelet\first@char\gobble@till@marker#1\endmarker
	\ifcat\noexpand\first@char A\else
	\def\first@char{}%
	\fi
	\macc@nested@a\relax111{\first@char}%
	\fi
	\endgroup
}
\makeatother

\newcommand{\hS}{\widehat S}

\newcommand{\tR}{\widetilde R}

\newcommand{\tX}{\widetilde X}

\newcommand{\tbeta}{\widetilde\beta}

\newcommand{\dbp}[1]{(\!(#1)\!)}

\newcommand{\gitquot}{\mathbin{\!/\mkern-5mu/\!}}

\newcommand{\CF}{\mathsf{CF}}

\newcommand{\pt}{\mathrm{pt}}

\newcommand{\td}{\tilde d}

\newcommand{\fm}{\mathfrak{m}}
\newcommand{\fp}{\mathfrak{p}}

\newcommand{\Id}{\mathrm{Id}}

\newcommand{\longto}{\longrightarrow}

\newcommand{\te}{\tilde{e}}

\DeclareMathOperator{\End}{End}
\DeclareMathOperator{\Hom}{Hom}

\DeclareMathOperator{\NE}{NE}

\DeclareMathOperator{\Sym}{Sym}

\DeclareMathOperator{\Spf}{Spf}

\DeclareMathOperator{\Spec}{Spec}
\DeclareMathOperator{\Proj}{Proj}
\DeclareMathOperator{\Nef}{Nef}
\DeclareMathOperator{\Bl}{Bl}
\DeclareMathOperator{\NS}{NS}

\DeclareMathOperator{\QH}{QH} 
\DeclareMathOperator{\QDM}{QDM} 
\newcommand{\La}{\mathrm{La}}
\newcommand{\rat}{\mathrm{rat}}
\newcommand{\ext}{\mathrm{ext}}
\newcommand{\red}{\mathrm{red}}
\newcommand{\loc}{\mathrm{loc}}
\newcommand{\tw}{\mathrm{tw}}

\newcommand{\ttau}{\widetilde{\tau}}

\newcommand{\tzeta}{\widetilde{\zeta}}

\def\blue{\textcolor{blue}}
\def\orange{\textcolor{orange}}

\newcommand{\bff}{\mathbf{f}}
\newcommand{\bfs}{\mathbf{s}}
\newcommand{\bft}{\mathbf{t}}
\newcommand{\sfF}{\mathsf{F}}
\newcommand{\sF}{\mathscr{F}}
\newcommand{\FT}{\mathrm{FT}}
\newcommand{\LT}{\mathrm{LT}}
\renewcommand{\Im}{\mathrm{Im}}
\newcommand{\Ker}{\mathrm{Ker}}

\newcommand{\hcS}{\widehat{\cS}}

\DeclareMathOperator{\rank}{rank}

\newcommand{\formal}[1]{\llbracket #1 \rrbracket}
\newcommand{\bigformal}[1]{\big\llbracket #1 \big\rrbracket}
\newcommand{\laurent}[1]{(\!( #1 )\!)}
\newcommand{\biglaurent}[1]{\big(\!\big( #1 \big)\!\big)}
\newcommand{\Biglaurent}[1]{\Big(\!\Big( #1 \Big)\!\Big)}

\renewenvironment{abstract}{%
  \quotation
  \small
  \textbf{\textit{\abstractname.}} 
}{\endquotation}

\begin{document}
\title{Quantum cohomology of variations of GIT quotients and flips}

\author{Zhaoxing Gu}
\address{Zhaoxing Gu, Department of Mathematics, M/C 253-37, Caltech, 1200 E.\ California Blvd., Pasadena, CA 91125, USA}
\email{zgu2@caltech.edu}

\author{Song Yu}
\address{Song Yu, Yau Mathematical Sciences Center, Tsinghua University, Haidian District, Beijing 100084, China}
\email{song-yu@tsinghua.edu.cn}

\author{Tony Yue YU}
\address{Tony Yue YU, Department of Mathematics, M/C 253-37, Caltech, 1200 E.\ California Blvd., Pasadena, CA 91125, USA}
\email{yuyuetony@gmail.com}

\date{August 21, 2025}
\subjclass[2020]{Primary 14N35; Secondary 14E30}

\maketitle

\begin{abstract}
We prove a decomposition theorem for the quantum cohomology of variations of GIT quotients.
More precisely, for any reductive group $G$ and a simple $G$-VGIT wall-crossing $X_- \dashrightarrow X_+$ with a wall $S$, we show that the quantum $D$-module of $X_-$ can be decomposed into a direct sum of that of $X_+$ and copies of that of $S$.
As an application, we obtain a decomposition theorem for the quantum cohomology of local models of standard flips in birational geometry.
\end{abstract}

\personal{Personal comments are shown!}

\tableofcontents

\addtocontents{toc}{\protect\setcounter{tocdepth}{1}}

\section{Introduction}\label{sect:Intro}

In this paper, we prove a decomposition theorem for the quantum cohomology of variations of geometric invariant theory (VGIT) quotients.
As an application, we deduce a decomposition theorem for the quantum cohomology of local models of standard flips in birational geometry.

In the introduction, we begin by reviewing the decomposition of (classical) cohomology of variations of GIT quotients.
Then we state our main theorem for the decomposition of quantum cohomology and discuss its implications.
After that, we discuss an application towards the decomposition of quantum cohomology of standard flips.
We outline the proof strategy of our main theorem by the end of the introduction.

\subsection{VGIT wall-crossings and decomposition of cohomology}
Let $W$ be a quasi-projective variety $W$ with an action of a reductive group $G$. For any ample $G$-linearization $L$, the GIT quotient $W\gitquot_{L}G$ is defined as the good quotient of the semistable locus of $L$ (see e.g. \cite{GIT_quotients}).
The quotient $W\gitquot_L G$ depends on the linearization $L$.
Nonempty GIT quotients of $W$ given by different linearizations are birational.

The cone $E^G(W)$ of all ample $G$-linearizations with nonempty semistable loci has a wall-and-chamber structure \cite{dolgachev1998variation,thaddeus1996geometric}.
Assume that $L_+$ and $L_-$ are two linearizations lying in two adjacent chambers of $E^G(W)$.
The two quotients $X_\pm \coloneqq W\gitquot_{L_{\pm}}G$ are related by a birational transformation $f\colon X_-\dashrightarrow X_+$, referred to as a VGIT wall-crossing. 
In this paper, we restrict to the case where $X_\pm$ are smooth projective and $f$ is a \emph{simple} VGIT wall-crossing with wall $S$ (see Definition~\ref{def:SimpleWall} for the precise definition).
In particular, such $f$ is a standard flip of type $(r_+-1,r_--1)$, where the exceptional loci are $\bbP^{r_{\pm}-1}$-bundles over $S$ respectively (see Definition~\ref{def:Flip}).
In this context, we have the following decomposition of cohomology, as a consequence of a decomposition of Chow motives \cite{lee2010flops,jiang2023chow}.

\begin{proposition}[Corollary~\ref{cor:deRhamIsoVGIT}]\label{prop:deRhamIsoVGIT}
For a simple $G$-VGIT wall-crossing $f\colon X_-\dashrightarrow X_+$ with wall $S$, assuming $r_+ \le r_-$ without loss of generality, we have an isomorphism of singular cohomology groups
$$
    H^*(X_-)\xrightarrow{\ \sim\ } H^*(X_+) \oplus \bigoplus\limits_{j=0}^{r_- - r_+-1}H^{*-2j-2r_+}(S).
$$
\end{proposition}

\subsection{Decomposition of quantum cohomology}

The main result of this paper is to extend the decomposition in Proposition~\ref{prop:deRhamIsoVGIT} from classical cohomology to quantum cohomology and quantum $D$-modules, encoding the enumerative geometry of curves in the varieties.

Recall that for a smooth projective variety $X$, its (big) quantum cohomology $\mathrm{QH}^*(X)$ is a deformation of $H^*(X)$ over $\bbC\formal{Q_X,\tau_X}$ given by the quantum product $\star_{\tau_X}$, where $Q_X$ denotes the Novikov variable, and $\tau_X$ is a formal parameter on $H^*(X)$.
The (big) quantum $D$-module $\QDM(X)$ of $X$ is the trivial $H^*(X)$-bundle over $\bbC[z]\formal{Q_X,\tau_X}$ with a flat connection $\nabla$ (called the quantum connection given by $\star_{\tau_X}$) and a $z$-sesquilinear pairing $P_X$ induced by the Poincar\'e pairing. See Section~\ref{subsec:quantum cohomology and qdm} for details.

The quantum $D$-modules $\QDM(X_{\pm})$ and $\QDM(S)$ are defined over distinct rings $\bbC[z]\formal{Q_{X_{\pm}},\tau_{X_{\pm}}}$ and $\bbC[z]\formal{Q_{S},\tau_{S}}$ respectively.
In order to compare them, we need to pull them back to a common base ring. In the case $r_+ < r_-$, we use $\bbC[z]\Biglaurent{Q_{X_-}^{\frac{a}{2(r_--r_+)}}}\formal{Q_{X_-},\tau_{X_-}}$ where $Q_{X_-}^a$ is the Novikov variable for the fiber curve class in the exceptional locus of $f\colon X_-\dashrightarrow X_+$ in $X_-$, which is a projective bundle over $S$. In the case $r_+ = r_-$, we use a more complicated ring $R^0$ which is a Laurent extension of smaller base rings $\bbC[z]\otimes_{\bbC}\bbC\formal{Q_{X_\pm}}^h\formal{\tau_{X_\pm}}$ over which the quantum $D$-modules of $X_{\pm}$ are well-defined (see Section~\ref{sect:GeneralVGITFlops} and \eqref{eqn:R0} for details). We will use the superscripts ``$\mathrm{La,red}$'' and ``$h,\La$'' to denote the base changes in the two cases respectively.

\begin{theorem}[Theorems~\ref{thm:GeneralQDMDecomposition}, \ref{thm:QDMDecompCrep}]\label{thm:IntroGeneralQDMDecomp}

When $r_+<r_-$, there exist a formal change of variables
\begin{align*}
H^*(X_-) &\longto H^*(X_+)\ \oplus\ H^*(S)^{\oplus (r_--r_+)},\\
\tau_{X_-} &\longmapsto \big(\tau_{X_+}^{\red}(\tau_{X_-}),\ \{\zeta_{j}^{\red}(\tau_{X_-})\}_{0\le j\le r_--r_+-1}\big)
\end{align*}
defined over $\bbC\Biglaurent{Q_{X_-}^{\frac{a}{2(r_--r_+)}}}\formal{Q_{X_-}}$, and an isomorphism of  $\bbC[z]\Biglaurent{Q_{X_-}^{\frac{a}{2(r_--r_+)}}}\formal{Q_{X_-},\tau_{X_-}}$-modules
$$
    \Psi\colon \QDM(X_-)^{\mathrm{La,red}} \longrightarrow (\tau_{X_+}^{\red})^*\QDM(X_+)^{\mathrm{La,red}}\oplus\bigoplus_{j=0}^{r_--r_+-1}(\zeta_j^{\red})^*\QDM(S)^{\mathrm{La,red}}
$$ 
that are compatible with the quantum connections and the pairings.

When $r_+=r_-$, 
there exists a formal change of variables $H^*(X_-)\to H^*(X_+)$, $\tau_{X_-} \mapsto \tau_{X_+}(\tau_{X_-})$, and an $R^0$-module isomorphism 
$$\Psi^{0}\colon\QDM(X_-)^{h,\La}\longrightarrow \tau_{X_+}^*\QDM(X_+)^{h,\La}$$
that are compatible with the quantum connections and the pairings. 
\end{theorem}

\begin{remark}\label{rem:remformaintheorem}
\begin{enumerate}[wide]
\item\label{rem:flopcase} When $r_+ = r_-$, the simple $G$-VGIT wall-crossing forms an ordinary flop (see Definition~\ref{def:Flip}).
In this case, Lee, Lin, Qu, and Wang  \cite{lee2010flops,lee2016ordinaryflops1,lee2016ordinaryflops2,lee2016ordinaryflops3} showed that the quantum cohomology is invariant after analytic continuation, providing an important case of the \emph{Crepant Transformation Conjecture} \cite{ruan2002stringy, ruan2006crepant, bryan2009crepant}. Previous works on the conjecture in the case of flops also include \cite{coates2018crepant, gonzalez2024vgit}.
The second part of Theorem~\ref{thm:IntroGeneralQDMDecomp} is a reformulation of this result in terms of quantum $D$-modules.
In this paper, we will mostly focus on the case $r_+ < r_-$. 

\item The formal change of variables $\tau_{X_+}^{\red}$, $\zeta_{j}^{\red}$ and the module map $\Psi$ can be uniquely and explicitly reconstructed from initial conditions using the extension of framing of $F$-bundles or Birkhoff factorization, see \cite{HYZZ_Decomposition,HLYZZ_Unfolding,iritani2023quantum}.

\item Using quantum $D$-modules and motivic Galois theory, the notion of atoms is introduced in \cite{KKPY_Birational} as elementary components of smooth projective varieties for the study of birational equivalences.
In particular, for any smooth projective variety $X$, we have its atomic decomposition, or chemical formula $\CF(X)$, which is a multiset of atoms contained in $X$.
In the context of our Theorem~\ref{thm:IntroGeneralQDMDecomp}, we can deduce that
\[
\CF(X_-)=\CF(X_+)+(r_- - r_+)\CF(S).
\]
\end{enumerate}
\end{remark}

\begin{corollary}\label{cor:decom of qh}
    In the context of Theorem~\ref{thm:IntroGeneralQDMDecomp}, when $r_+ < r_-$, the isomorphism of quantum $D$-modules induces an isomorphism of quantum cohomology rings
    \[\big(H^*(X_-),\star_{\tau_{X_-}}\big)\cong \big(H^*(X_+),\star_{\tau_{X_+}^{\red}(\tau_{X_-})}\big)\oplus\bigoplus_{j=0}^{r_--r_+-1}\big(H^*(S), \star_{\zeta_j^{\red}(\tau_{X_-})}\big)\]
    over $\bbC\Biglaurent{Q_{X_-}^{\frac{a}{2(r_--r_+)}}}\formal{Q_{X_-},\tau_{X_-}}$.
    It also preserves the Euler vector fields.
\end{corollary}

\begin{remark}\label{rem:Eigenvalues}
Note that the decomposition of quantum cohomology rings in Corollary~\ref{cor:decom of qh} induces a decomposition of the eigenvalues of $E_{X_-}\star_{\tau_{X_-}}$ along the restriction $Q_{X_-}^{\beta}=\tau_{X_-}=0$ for all $\beta\notin \frac{1}{2(r_--r_+)}\bbZ a $ into the eigenvalues of $E_{X_+}\star_{\tau_{X_+}^{\red}(\tau_{X_-})}$ and those of $E_{S}\star_{\zeta_j^{\red}(\tau_{X_+})}$. 
In Proposition~\ref{prop:eigenvalue of Euler vector field}, we compute that along the locus 
$$
    Q_{X_-}^{\beta} = 0 \text{ for all } \beta\notin \frac{1}{2(r_--r_+)}\bbZ a,\quad Q_{X_-}^{-\frac{a}{2(r_--r_+)}}=q\in\bbC, \quad \tau_{X_-}=0,
$$
all eigenvalues of $E_{X_+}\star_{\tau_{X_+}^{\red}(\tau_{X_-})}$ are $0$ and those of $E_{S}\star_{\zeta_j^{\red}(\tau_{X_-})}$ are $(r_--r_+)q^2e^{\frac{\pi\sqrt{-1}(r_+-2j)}{r_--r_+}}$.
Therefore, along this locus, the eigenvalues of $E_{X_-}\star_{\tau_{X_-}}$ consist of $0$ with multiplicity $\mathrm{dim}_{\bbC}(H^*(X_+))$ and $(r_--r_+)q^2e^{\frac{\pi\sqrt{-1}(r_+-2j)}{r_--r_+}}$ with multiplicity $\mathrm{dim}_{\bbC}(H^*(S))$ for $0\le j\le r_--r_+-1$.
Shen-Shoemaker \cite{shen2025quantum} provides a similar eigenvalue decomposition of the quantum multiplication by Euler vector fields in the case of a standard flip, formulated in terms of extremal quantum cohomology rings. 
\end{remark}
\medskip
Our Theorem~\ref{thm:IntroGeneralQDMDecomp} includes, as examples, the following previously studied cases.

\begin{example}[Blowups]\label{ex:Blowup}
Let $X$ be a smooth complex projective variety, $Z \subset X$ a smooth closed subvariety of codimension $r \ge 2$, and $\tX = \Bl_Z X$ the blowup of $X$ along $Z$. The decomposition of the quantum $D$-module of $\tX$ in terms of those of $X$ and $Z$ was proved by Iritani \cite{iritani2023quantum}, based on the following realization of $X$ and $\tX$ as a variation of GIT quotients. 
Let $W = \Bl_{Z \times \{0\}} X \times \bbP^1$ which is a smooth projective variety with a $\bbC^*$-action induced by the action on the $\bbP^1$-factor. The $\bbC^*$-fixed locus has three connected components $W^{\bbC^*} = X \sqcup Z \sqcup \tX$.
The divisors $X$ and $\tX$ are the only two nonempty GIT quotients of $W$, connected by a simple VGIT wall-crossing of type $(0,r-1)$ with wall $Z$.
Our Theorem~\ref{thm:IntroGeneralQDMDecomp} is a generalization of \cite[Theorem 1.1]{iritani2023quantum}.
\end{example}

\begin{example}[Toric flips]\label{ex:Toric}
Let $T$ be an algebraic torus acting on $W = \bbC^N$ linearly. Notice that all $T$-linearizations on $W$ are induced by characters of $T$. 
For a simple VGIT wall-crossing of smooth projective toric varieties $W\gitquot_{\chi_-} T\dashrightarrow W\gitquot_{\chi_{+}}T$ for characters $\chi_+$ and $\chi_-$ of $T$, their geometric relation can be described combinatorially. Theorem~\ref{thm:IntroGeneralQDMDecomp} provides a decomposition of quantum $D$-modules for $X_{\pm}=W\gitquot_{\chi_\pm} T$. More generally, the decompositions of quantum $D$-modules under crepant and discrepant transformations of semi-projective toric Deligne-Mumford stacks have been studied by Coates-Iritani-Jiang \cite{coates2018crepant} and Iritani \cite{iritani2020discrepant}. We note that \cite[Theorem 1.1]{iritani2020discrepant} obtains an equivariant version of our Theorem~\ref{thm:IntroGeneralQDMDecomp} with respect to the dense torus $(\bbC^*)^N \gitquot T$ of $X_\pm$. On the other hand, in our case, our Theorem~\ref{thm:IntroGeneralQDMDecomp} provides a geometric description of the residue piece in the decomposition, while in \cite{iritani2020discrepant} the residue piece is obtained through the decomposition of Brieskorn modules. Previous work on the toric case also includes \cite{gonzalez2019toric, acosta2018quantum, acosta2020GW}.
\end{example}

\subsection{Application: decomposition of quantum cohomology for general flips}
As an application of Theorem~\ref{thm:IntroGeneralQDMDecomp}, we study the decomposition of quantum $D$-modules of general standard flips (see Definition~\ref{def:Flip}).
First we state the following conjecture.


\begin{conjecture}\label{conj:IntroGeneralFlip}
For a standard flip $X_- \dashrightarrow X_+$ of type $(r_+-1, r_--1)$ with wall $S$ and $r_+<r_-$, there exist formal changes of variables $\tau_{X_+}$ and $\zeta_j$, and a $\bbC[z]\Biglaurent{Q_{X_-}^{\frac{a}{2(r_--r_+)}}}\formal{Q_{X_-},\tau_{X_-}}$-module isomorphism 
$$
    \Psi \colon \QDM(X_-)^{\mathrm{La,red}}\longrightarrow(\tau_{X_+})^*\QDM(X_+)^{\mathrm{La,red}}\oplus\bigoplus_{j=0}^{r_--r_+-1}(\zeta_j)^*\QDM(S)^{\mathrm{La,red}}
$$ 
that are compatible with the quantum connections and the pairings.
\end{conjecture}

In the case of ordinary flops, as mentioned in Remark~\ref{rem:remformaintheorem}(\ref{rem:flopcase}),
the invariance of quantum cohomology up to analytic continuation is proved in \cite{lee2010flops, lee2016ordinaryflops1, lee2016ordinaryflops2, lee2016ordinaryflops3}; for simple flops ($S = \pt$), the result is extended to all genera in \cite{iwao2012simpleflop}.


The main approach of the aforementioned works on ordinary flops is to consider the degeneration to the normal cone of the exceptional locus of $X_- \dashrightarrow X_+$. Recall that the exceptional loci in $X_\pm$ are projective bundles $\psi_\pm \colon \bbP(V_{\pm}) \to S$ for vector bundles $V_\pm$ over $S$ of ranks $r_\pm$ respectively. Via the degeneration formula in Gromov-Witten theory \cite{Ruan2001Symplectic, Li2002Degeneration}, the problem is reduced to the birational transformation $X_{-,\loc} \dashrightarrow X_{+,\loc}$ between the (projective) local models, where
$$
    X_{\pm,\loc} = \bbP_{\bbP(V_{\pm})}(N_{\bbP(V_{\pm})/X_{\pm}} \oplus \cO) = \bbP_{\bbP(V_{\pm})}(\psi_\pm^*(V_\mp) (-1) \oplus \cO) 
$$
are double projective bundles over $S$. It is then shown that $X_{\pm, \loc}$ have isomorphic quantum cohomology rings up to analytic continuation.



In the case of flips, this approach has been applied by \cite{lee2021flips} to simple $(2,1)$-flips and adopted by \cite{shen2025quantum} to obtain a decomposition of the quantum spectrum. 
It is expected that this approach can be applied to general standard flips and lead to a proof of Conjecture~\ref{conj:IntroGeneralFlip}. Here, following \cite[Section 3.1]{shen2025quantum}, we observe that
the local models $X_{\pm,\loc}$ arise from a simple VGIT wall-crossing (see Section~\ref{sect:FlipLocalModel} for details). Therefore, Theorem~\ref{thm:IntroGeneralQDMDecomp} has in addition the following corollary.


\begin{corollary}[Corollary~\ref{cor:qdmdec for local model}]\label{cor:intro qdmthm for local model}
For a standard flip $X_- \dashrightarrow X_+$ of type $(r_+-1, r_--1)$ with $r_+ < r_-$, Conjecture~\ref{conj:IntroGeneralFlip} holds for the local model $X_{-,\loc} \dashrightarrow X_{+,\loc}$.
\end{corollary}

The decomposition for local models of simple flips ($S = \pt$) is proved by \cite{lee2021flips}.



An alternative approach to Conjecture~\ref{conj:IntroGeneralFlip} is to directly realize the standard flip $X_-\dashrightarrow X_+$ itself as a simple VGIT wall-crossing and apply Theorem~\ref{thm:IntroGeneralQDMDecomp}. We attempt this approach in Section~\ref{sect:GeneralFlips} and give a sufficient condition in Corollary~\ref{cor:FlipQDMDecomposition}, by examining a construction of Reid and Thaddeus \cite{thaddeus1996geometric}.

\subsection{Proof strategy of the main result}\label{sect:ProofStrategy}
In this subsection, we describe the proof strategy of Theorem~\ref{thm:IntroGeneralQDMDecomp}.

\subsubsection{Reduction to 3-component $\bbC^*$-VGIT wall-crossings}
A key step in proving Theorem~\ref{thm:IntroGeneralQDMDecomp} is the reduction from a general $G$-VGIT wall-crossing to a 3-component $\bbC^*$-VGIT wall-crossing. For the latter we mean the case where $W$ is projective, $G = \bbC^*$ acts with weights $\pm 1$ (see Definition~\ref{def:WtPm1}), and the $\bbC^*$-fixed locus only has three connected components $W^{\bbC^*}=F_+\sqcup F_0\sqcup F_-$ where $F_{\pm}$ are the highest/lowest components respectively and $F_0$ is the middle component; see Assumption~\ref{3-component assumption} for details and Section~\ref{sect:C*FixedLocus} for the partial order on the components. In this case, as discussed in Section~\ref{sect:3Component}, $W$ admits a unique simple $\bbC^*$-VGIT wall-crossing $X_-\dashrightarrow X_+$ of type $(r_{F_0,+}-1,r_{F_0,-}-1)$ with wall $F_0$, where $X_{\pm}$ are projective bundles over $F_{\pm}$ respectively, and $r_{F_0,\pm}$ is the rank of the eigenbundle of the normal bundle $N_{F_0/W}$ with $\bbC^*$-weight $\pm 1$. We denote $c_{F_0}=r_{F_0,+}-r_{F_0,-}$.

The reduction is made through the master space construction of Thaddeus \cite{thaddeus1996geometric} (see also \cite[Section 3]{liu2025invariance}). Specifically, for a smooth quasi-projective $G$-variety $W$ with a simple $G$-VGIT wall-crossing $X_-\dashrightarrow X_+$, we can construct a new smooth projective $\bbC^*$-variety $M$ with
three fixed component as above, where $X_-\dashrightarrow X_+$ also arises as the unique $\bbC^*$-VGIT wall-crossing for $M$. The simpleness condition (Definition~\ref{def:SimpleWall}) on the $G$-VGIT wall-crossing ensures that $M$ is smooth and the $\bbC^*$-action on $M$ has weights $\pm 1$. See Section~\ref{sect:MasterSpace} for details of the construction.


  
Therefore, we reduce Theorem~\ref{thm:IntroGeneralQDMDecomp} to the following.

\begin{theorem}[Theorem~\ref{thm:QDMDecompRed}]\label{thm:intro 3compqdmdecom}
Let $X_-\dashrightarrow X_+$ be a 3-component $\bbC^*$-VGIT wall-crossing (in the sense of Assumption~\ref{3-component assumption}). If $c_{F_0} < 0$, there exist formal changes of coordinates $\tau_{X_+}^{\red}$, $\zeta_{j,F_0}^{\red}$, $0\le j\le \abs{c_{F_0}}-1$, and a $\bbC[z]\Biglaurent{Q_{X_-}^{-\frac{a}{2c_{F_0}}}}\formal{Q_{X_-},\tau_{X_-}}$-module isomorphism 
$$
    \Phi^{\red}\colon\QDM(X_-)^{\red, \La}\longrightarrow(\tau_{X_+}^{\red})^*\QDM(X_+)^{\red, \La}\oplus \bigoplus_{j=0}^{\abs{c_{F_0}}-1}(\zeta_{j,F_0}^{\red})^*\QDM(F_0)^{\red, \La}
$$
that are compatible with the quantum connections and the pairings.  
\end{theorem}

\subsubsection{Fourier analysis on equivariant quantum cohomology}
The main technique used to prove Theorem~\ref{thm:intro 3compqdmdecom} is Fourier analysis on equivariant quantum cohomology developed by Iritani \cite{iritani2023quantum, iritani2025fourier} and Iritani-Koto \cite{iritani-koto2023quantum}, which is inspired by Teleman \cite{teleman2014gauge}. In this paper, we only consider the $\bbC^*$-equivariant setting, although the technique is developed for tori of higher ranks. For a smooth projective variety $W$ with a $\bbC^*$-action, Fourier transformations relate the equivariant Gromov-Witten theory of $W$ to the non-equivariant Gromov-Witten theories of its $\bbC^*$-fixed components and $\bbC^*$-GIT quotients. There are two types of Fourier transformations, described in the Givental formalism of genus-zero Gromov-Witten theory (see Section~\ref{subsec:Givental formalism} for preliminaries) as follows:
\begin{enumerate}[wide]
    \item The continuous Fourier transformations $\sF_{F,j}$, which map the $\bbC^*$-equivariant Givental cone $\cL_W$ to the (nonequivariant) Givental space $\cH_F$ of a connected component $F \subset W^{\bbC^*}$ up to certain base changes. Here, the index $j$ varies from $0$ to $\abs{c_F}-1$, and the quantity $c_F$ depends on the weight decomposition of the $\bbC^*$-action on $N_{F/W}$ (see Definition~\ref{def:ContinuousFT} for details).

    \item The discrete Fourier transformation $\sfF_{X}$ (see Definition~\ref{def:DiscreteFT}), which maps $\cL_W$ to the an extension $\cH_X^{\ext}$ of the (nonequivariant) Givental space of a smooth GIT quotient $X = W \gitquot \bbC^*$. 
  
\end{enumerate}

It is known that continuous Fourier transformations $\sF_{F,j}$ in fact map $\cL_W$ to the Givental cone $\cL_F$ (see \cite[Corollary 4.9]{iritani2023quantum}). On the other hand, for the discrete Fourier transformation $\sfF_{X}$, it is not clear from the definition that it is defined on all of $\cL_W$, nor that it preserves the Givental cones (up to the same extension of $\cH_X^{\ext}$). Iritani conjectures that both properties should hold.


\begin{conjecture}[{\cite[Conjecture 1.8]{iritani2023quantum}}]\label{conj:introReduction}
Let $J_W(\theta)$ be the $\bbC^*$-equivariant $J$-function of $W$ and $X$ be a smooth GIT quotient of $W$. Then $z\sfF_X(J_W(\theta))$ is defined and lies on the Givental cone $\cL_X$ of $X$.
\end{conjecture}



In the case where $X$ is a $\bbC^*$-fixed divisor of $W$ whose normal bundle has weight $\pm 1$, such as for the divisors $X$ and $\tX$ in the blowup example (see Example~\ref{ex:Blowup}), Conjecture~\ref{conj:introReduction} was proved in \cite[Theorem 1.13]{iritani2023quantum}. In this case, the proof is based on an identification of the discrete Fourier transformation $\sfF_X$ with the continuous Fourier transformation $\sF_{X, 0}$ (up to a certain factor) which preserves the Givental cone (see \cite[Section 4.3, Appendix A]{iritani2023quantum}). We prove Conjecture~\ref{conj:introReduction} for the more general case where $X$ is the \emph{highest} (resp.\ \emph{lowest}) GIT quotient of $W$ (see Definition~\ref{def:highest/lowest GIT}) where $X$ is a projective bundle over the highest (resp.\ lowest) $\bbC^*$-fixed component of $W$ whose normal bundle has $\bbC^*$-weight $-1$ (resp.\ $1$).



\begin{theorem}[Theorem~\ref{thm:Reduction}] \label{thm:IntroReduction}
Conjecture~\ref{conj:introReduction} holds for the highest/lowest GIT quotients of $W$.
\end{theorem}

As a consequence, for the highest/lowest GIT quotient $X$, the discrete Fourier transformation $\sfF_{X}$ acting on the Givental cone induces a transformation $\FT_{X}$ on the level of quantum $D$-modules (see Proposition~\ref{prop:DiscreteFTDmodule}). Similarly, for each fixed component $F\subset W^{\bbC^*}$, the continuous Fourier transformations $\sF_{F,j}$ on the cone level induce transformations $\FT_{F,j}$ on the level of quantum $D$-modules (see Proposition~\ref{prop:ContFTDmodule}). 

Now, in the 3-component $\bbC^*$-VGIT context of Theorem~\ref{thm:intro 3compqdmdecom}, we assemble the Fourier transformations to construct the desired isomorphism $\Phi^{\red}$ by passing through the equivariant quantum $D$-module of $W$. Specifically, our analysis in Section~\ref{sect:Decomposition} gives the following relationships among the extended quantum $D$-modules $\QDM(X_{\pm})^{\La}$, $\QDM(F_0)^{\La}$, and the completed equivariant quantum $D$-module $\QDM_{\bbC^*}(W)_{X_-}^{\wedge,\La}$ of $W$ (see Definition~\ref{def:completion of qdm}).

\begin{theorem}[Theorem~\ref{thm:QDMDecompExt}] \label{thm:introQDMDECompEXt}
In the context of Theorem~\ref{thm:intro 3compqdmdecom}, the Fourier transformations induce $\bbC[z]\Biglaurent{S_{F_0}^{-\frac{1}{2c_{F_0}}}}\formal{Q_W,\theta}$-module isomorphisms
$$
    \FT_{X_-}\colon\QDM_{\bbC^*}(W)^{\wedge,\La}_{X_-}\longrightarrow \tau_{X_-}^*\QDM(X_-)^{\La}
$$ 
and
$$
    \FT_{X_+}\oplus \bigoplus_{j=0}^{\abs{c_{F_0}}-1}\FT_{F_0,j}\colon\QDM_{\bbC^*}(W)^{\wedge,\La}_{X_-}\longrightarrow \tau_{X_+}^*\QDM(X_+)^{\La}\oplus \bigoplus_{j=0}^{\abs{c_{F_0}}-1}\zeta_j^*\QDM(F_0)^{\La}.
$$
Moreover, the isomorphism $\Phi$ defined by the composition
$$
    \left( \FT_{X_+} \oplus\bigoplus_{j=0}^{\abs{c_{F_0}}-1}\FT_{F_0,j} \right)\circ \FT_{X_-}^{-1}\colon\tau_{X_-}^*\QDM(X_-)^{\La}\to\tau_{X_+}^*\QDM(X_+)^{\La}\oplus \bigoplus_{j=0}^{\abs{c_{F_0}}-1} \zeta_j^*\QDM(F_0)^{\La}
$$
is compatible with the quantum connections and the pairings.
\end{theorem}

The isomorphism $\Phi$ in Theorem~\ref{thm:introQDMDECompEXt} is defined over the extended base ring $\bbC[z]\Biglaurent{S_{F_0}^{-\frac{1}{2c_{F_0}}}}\formal{Q_W,\theta}$ which involves both the Novikov variables and coordinates of $W$. In particular, $S_{F_0}$ is the Novikov variable for an equivariant curve class in $W$ that is associated to $F_0$ and is a lift of the curve class $a$. We show that the extra Novikov variables and coordinates can in fact be eliminated so that the isomorphism can be established over the reduced base $\bbC[z]\Biglaurent{Q_{X_-}^{-\frac{a}{2c_{F_0}}}}\formal{Q_{X_-},\tau_{X_-}}$ in Theorem~\ref{thm:intro 3compqdmdecom}.

\subsection{Outline for this paper}
In Section~\ref{sect:Prelim}, we review the preliminaries of quantum cohomology and quantum $D$-modules.
In Section~\ref{sect:GIT}, we discuss variations of GIT quotients and focus on the geometry in the case of $\bbC^*$-actions. In Section~\ref{sect:Fourier}, we discuss the technique of Fourier analysis on the levels of Givental cones and quantum $D$-modules, and prove Theorem~\ref{thm:IntroReduction}.
In Section~\ref{sect:Decomposition}, we prove the decomposition of quantum $D$-modules for 3-component $\bbC^*$-VGIT wall-crossings (Theorems~\ref{thm:intro 3compqdmdecom},~\ref{thm:introQDMDECompEXt}). Finally, in Section~\ref{sect:General}, we prove the decomposition for general $G$-VGIT wall-crossings (Theorem~\ref{thm:IntroGeneralQDMDecomp}) based on the master space reduction, and discuss the application to the decomposition for general standard flips.

\subsection{Acknowledgments}
We would like to thank Yassine El Maazouz, Shuai Guo, Thorgal Hinault, Hiroshi Iritani, Ludmil Katzarkov, Maxim Kontsevich, Y.P.\ Lee, Henry Liu, Tony Pantev, Mark Shoemaker, Yefeng Shen, Weihong Xu, Haosen Wu, Chi Zhang, and Shaowu Zhang for the support and valuable discussions.
Z.\ Gu and T.Y.\ Yu were partially supported by NSF grants DMS-2302095 and DMS-2245099.

\section{Preliminaries}\label{sect:Prelim}

In this section, we introduce the preliminaries for this paper and set up notation.
We work over $\bbC$.
The discussion in this section includes both the nonequivariant and the $\bbC^*$-equivariant settings. To distinguish the two settings notationally, we use $X$ to denote a smooth projective variety in the nonequivariant case, and $W$ to denote a smooth projective variety with a $\bbC^*$-action in the equivariant case.

\subsection{Formal series rings and graded completions}\label{sect:GradedCompletion}
In this subsection, we introduce the conventions for formal series rings, following \cite[Section 2.2]{iritani2023quantum}.
Let $M=\bigoplus_{n\in\bbZ}M_n$ be a graded $\bbZ$-module with a descending chain of graded submodules $N_k=\bigoplus_{n\in\bbZ}N_{k,n}$. We denote the graded completion by $\widehat{M}=\bigoplus_{n\in \bbZ}\widehat{M}_n$ , where $\widehat{M}_n=\varprojlim_kM_n/N_{k,n}$. Throughout this paper, completions are always taken in the graded sense. 

For a graded $\bbZ$-module $K$ and a family of formal variables with assigned gradings $x=(x_1,x_2,\cdots)$, we define $K\formal{x}$ as the completion of $K[x]$ with respect to the descending chain of graded submodules $(x_1^n,\cdots,x_n^n,x_{n+1},\cdots)K[x]$. If $x$ is a single graded variable, $K\formal{x}$ denotes the formal completion of $K[x]$ with respect to the descending chain $x^nK[x]$. We also use the formal Laurent series ring $K\dbp{x}$, defined as the completion of $K[x,x^{-1}]$ with respect to the descending chain $x^nK[x]$. A degree-$d$ homogeneous element of $K\dbp{x}$ has form $\sum\limits_{m=n}^{\infty}a_mx^m$, where $m\deg(x)+\deg(a_m)=d$ for all $m$. 

Additionally, given a monoid $C_{\bbN}$ in a lattice $N$ such that $C=C_{\bbN}\otimes \bbR$ is a strict convex cone in $N\otimes \bbR$, we choose an interior lattice point $\omega\in C^{\vee}$ in the dual cone and define $\bbC\formal{C_{\bbN}}$ as the completion of $\bbC[C_{\bbN}]$ with respect to the descending chain of ideals 
$$
    I_k=\langle\beta\in  C_{\bbN}\ \big|\ (\beta, \omega) > k \rangle\subset \bbC[C_{\bbN}].
$$
The completion is independent of the choice of $\omega$.

\subsection{Cohomology groups}
We use $H^*(-)$ (resp. $H^*_{\bbC^*}(-)$) to denote the singular cohomology group (resp. $\bbC^*$-equivariant singular cohomology group). When the coefficients are not specified, we take $\bbC$-coefficients.

We denote $n_X\coloneqq \dim_{\bbC}H^*(X)$. Fix a homogeneous $\bbC$-basis $\{\phi_0, \cdots,\phi_{n_X-1}\}$ of $H^*(X)$ with $\phi_0=1$ and let $\{\phi^0, \cdots, \phi^{n_X - 1}\}$ be the dual basis. We also write $\phi^i = \tau_X^i$ and refer to them as \emph{coordinates} hereafter. We assign the degree $(2-\deg(\phi_i))$ to $\tau_X^i$. We will always distinguish these coordinates $\tau_{X}^i$ for different algebraic varieties $X$ via their subscripts. We use $\int_X$ to denote the Poincar\'e pairing.

Now consider the $\bbC^*$-variety $W$. 
Let $\lambda \in H^2_{\bbC^*}(\pt; \bbZ) \cong \Hom(\bbC^*, \bbC^*)$ be the generator corresponding to the character $\Id$. We have $H^*_{\bbC^*}(\pt) \cong \bbC[\lambda]$. We fix a homogeneous $H_{\bbC^*}^*(\pt)$-basis $\{\varphi_0, \cdots,\varphi_{n_W-1}\}$ of $H_{\bbC^*}^*(W)$ with $\varphi_0=1$ and let $\{\varphi^0, \cdots, \varphi^{n_W-1}\}$ be the dual basis. Then $\{\varphi_i\lambda^n\}_{0\le i\le n_W-1,n\in\bbN}$ is a $\bbC$-basis of $H_{\bbC^*}^*(W)$. The $\bbC$-dual basis is denoted by $\{\theta^{i,n}\}_{0\leq i\leq n_W-1,n\in\bbN}$, where each $\theta^{i,n}$ is assigned the degree $(2-\deg(\varphi_i)-2n)$. We will consistently use $\theta^{i,n}$ to represent the equivariant coordinates of $W$ without emphasizing it via a subscript. We use $\int_W$ to denote the $\bbC^*$-equivariant Poincar\'e pairing.

\subsection{Curve classes and Novikov rings}
Let $\NS(X)$ be the N\'eron-Severi group of line bundles on $X$ modulo algebraic equivalence. We define $N^1(X)\subset H^2(X,\bbZ)$ as the image of $\NS(X)$ under the first Chern class homomorphism $c_1\colon \NS(X)\to H^2(X,\bbZ)$.

Dually, let $N_1(X)\subset H_2(X,\bbZ)$ be the group generated by the homology classes of algebraic curves in $X$. We use $(-, -)$ to denote the pairing between $H_2(X)$ and $H^2(X)$ as well as the induced pairing between $N_1(X)$ and $N^1(X)$. For $\bbF=\bbQ,\bbR,\bbC$, we denote the $\bbF$-vector spaces 
$$
    N_1(X)_{\bbF} = N_1(X)\otimes_{\bbZ} \bbF, \quad N^1(X)_{\bbF} = N^1(X) \otimes_{\bbZ} \bbF.
$$
The induced pairing between $N_1(X)_{\bbF}$ and $N^1(X)_{\bbF}$ is perfect over ${\bbF}$.

Let $\NE_{\bbN}(X) \subset N_1(X)$ be the monoid of effective curve classes in $X$. The \emph{Novikov ring} $\bbC\formal{\NE_{\bbN}(X)}$ of $X$ is the completed monoid ring of $\NE_{\bbN}(X)$ (see Section~\ref{sect:GradedCompletion}). We denote it by $\bbC\formal{Q_X}$, where the formal variable $Q_X^\beta$ is the Novikov variable for $\beta \in \NE_{\bbN}(X)$ and has degree $2(\beta, c_1(X))$. 



Let $\NE(X) \subset N_1(X)_{\bbR}$ be the cone generated by $\NE_{\bbN}(X)$ over $\bbR_{\ge 0}$. 
The cone $\NE(X)$ is dual to the nef cone $\Nef(X) \subset N^1(X)_{\bbR}$ of $X$, which is the closure of the cone generated by ample classes. In particular, the pairing $(\beta, \alpha)$ between a curve class $\beta \in \NE_{\bbN}(X)$ and a nef (or ample) divisor class $\alpha \in \Nef(X)$ is nonnegative.

\subsection{Quantum cohomology and quantum $D$-modules}\label{subsec:quantum cohomology and qdm}
For $n, k_1,\dots,k_n \in\bbN$, $\beta\in \NE_{\bbN}(X)$, and $\alpha_1,\cdots,\alpha_n\in H^*(X)$, let 
$$
    \langle\alpha_1\psi^{k_1},\cdots,\alpha_n\psi^{k_n}\rangle^X_{0,n,\beta}
$$
denote the corresponding genus-zero, degree-$\beta$, $n$-pointed (descendant) \emph{Gromov-Witten invariant} of $X$, where $\psi$ denotes the $\psi$-class on the moduli space of stable maps. The \emph{quantum cohomology ring} $\QH^*(X)$ of $X$ is the graded module $H^*(X)\formal{Q_X,\tau_X}$ equipped with the quantum product $\star_{\tau_X}$, where $H^*(X)\formal{Q_X,\tau_X}=H^*(X)\formal{Q_X}\formal{\tau_X^0,\tau_X^1,\cdots,\tau_X^{n_X-1}}$, $\tau_X=\sum\limits_{i=0}^{n_X-1}\tau^i_X\phi_i$ is a general point of $H^*(X)$, and $\star_{\tau_X}$ is a $\bbC\formal{Q_X,\tau_X}$-bilinear, associative, and super-commutative product on $\QH^*(X)$ defined by
$$
    \int_X (\phi_i \star_{\tau_X} \phi_j) \cup \phi_k  = \sum_{\substack{\beta \in \NE_{\bbN}(X) \\ n \geq 0}} 
    \big{\langle} \phi_i, \phi_j, \phi_k, \tau_X, \ldots, \tau_X \big{\rangle}^{X}_{0,n+3,\beta} 
    \frac{Q_X^\beta}{n!}.
$$

Let $z$ be a formal variable with degree $2$. The \emph{quantum $D$-module} $\QDM(X)$ of $X$ is the graded module $H^*(X)[z]\formal{Q_X,\tau_X}$, which is a trivial $H^*(X)$-bundle over the $(z,\tau_X,Q_X)$-space, equipped with a flat quantum connection $\nabla$ and a $z$-sesquilinear pairing $P_X$. The quantum connection $\nabla\colon\QDM(X)\to z^{-1}\QDM(X)$ is the flat connection given by the following operators:
\begin{align*}
&\nabla_{\tau_X^i} = \partial_{\tau_X^i} + z^{-1} \left( \phi_i \star_{\tau_X} \right),\\
&\nabla_{z \partial_z} = z \partial_z - z^{-1} \left( E_X \star_{\tau_X} \right) + \mu_X ,\\
&\nabla_{\xi Q \partial_Q} = \xi Q \partial_Q + z^{-1} \left( \xi \star_{\tau_X} \right).
\end{align*}
Here, $E_X$ is the Euler vector field and $\mu_X$ is the grading operator, defined respectively by 
\begin{equation}\label{Euler vector field}
E_X=c_1(X)+\sum\limits_{i}\bigg(1-\frac{\deg\phi_i}{2}\bigg)\tau_X^i\phi_i,\text{ } \mu_X(\phi_i)=\frac{\deg\phi_i-\dim_{\bbC}X}{2}\phi_i.
\end{equation}
Moreover, $\xi Q\partial_Q\cdot Q_X^\beta = (\beta, \xi)Q_X^\beta$ is the derivative in the direction of $\xi\in H^2(X)$. For $f,g\in \QDM(X)$, the $z$-sesquilinear pairing $P_X$ is defined by the Poincar\'e pairing as
$$
    P_X(f,g) = \int_Xf(-z)\cup g(z)
$$
and satisfies $dP_X(f,g)=P_X(f,\nabla g)+P_X(\nabla f,g)$.

The flat sections of the quantum connection can be given by \emph{fundamental solutions} \cite{givental1996equivariant}. Specifically, the operator $M_X(\tau_X)\in \End(H^*(X))[z^{-1}]\formal{Q_X,\tau_X}$ given by
$$
    \int_X M_X(\tau_X)\phi_i \cup \phi_j = \int_X \phi_i \cup \phi_j + \sum_{\substack{\beta \in \NE_{\bbN}(X),\, n \geq 0 \\ (n,\beta) \ne (0,0)}} 
    \Big{\langle} \phi_i, \tau_X, \ldots, \tau_X, \frac{\phi_j}{z-\psi} \Big{\rangle}^{X}_{0, n+2, \beta} \frac{Q_X^\beta}{n!}
$$
satisfies that 
$$
\begin{aligned}
M_X(\tau_X) \circ \nabla_{\tau_X^i} &= \partial_{\tau_X^i} \circ M_X(\tau_X), \\
M_X(\tau_X) \circ \nabla_{z \partial_z} &= \left(z \partial_z - z^{-1} c_1(X) + \mu_X \right) \circ M_X(\tau_X), \\
M_X(\tau_X) \circ \nabla_{\xi Q \partial_Q} &= \left(\xi Q \partial_Q + z^{-1} \xi \right) \circ M_X(\tau_X), \\
P_X(M_X(\tau_X)f,M_X(\tau_X)g)&= P_X(f,g).
\end{aligned}
$$
By the string equation, the \emph{$J$-function} of $X$ defined by  $J_X(\tau_X)=M_X(\tau_X)1$ expands as 
$$
    J_X(\tau_X) = 1+\frac{\tau_X}{z}+\sum_{\substack{\beta \in \NE_{\bbN}(X),\, n \geq 0 \\ (n,\beta) \ne (0,0),(1,0)\\ i=0,1,\ldots,n_X-1}} 
    \phi^i\Big{\langle} \tau_X, \ldots, \tau_X, \frac{\phi_i}{z(z-\psi)} \Big{\rangle}^{X}_{0, n+1, \beta} \frac{Q_X^\beta}{n!}.
$$

\subsection{Equivariant quantum cohomology and equivariant quantum $D$-modules}
For the $\bbC^*$-variety, the $\bbC^*$-equivariant quantum cohomology $\QH_{\bbC^*}(W)$ and the $\bbC^*$-equivariant quantum $D$-module $\QDM_{\bbC^*}(W)$ of $W$ may be defined in a similar way. 
More precisely, for $\alpha_1,\cdots,\alpha_n\in H^*_{\bbC^*}(W)$, let 
$$
    \langle \alpha_1\psi^{k_1},\cdots,\alpha_n\psi^{k_n} \rangle^{W,\bbC^*}_{0,n,\beta}
$$
denote the corresponding genus-zero, degree-$\beta$, $n$-pointed $\bbC^*$-equivariant (descendant) Gromov-Witten invariant of $W$. 
We represent a general point in $H_{\bbC^*}^*(W)$ by $\theta=\sum\limits_{i,n}\theta^{i,n}\varphi_i\lambda^n$. Then $\QH_{\bbC^*}(W)$ is the graded module $H_{\bbC^*}^*(W)\formal{Q_W,\theta} = H_{\bbC^*}^*(W)\formal{Q_W}\formal{\theta^{i,n}}$ equipped with the $\bbC[\lambda]\formal{Q_W,\theta}$-bilinear, associative, and super-commutative quantum product $\star_{\theta}$ defined by 
$$
    \int_W (\varphi_i \star_\theta \varphi_j) \cup \varphi_k  = \sum_{\substack{\beta \in \NE_{\bbN}(W) \\ n \geq 0}} 
    \big{\langle} \varphi_i, \varphi_j, \varphi_k, \theta, \ldots, \theta \big{\rangle}^{W,\bbC^*}_{0,n+3,\beta} 
\frac{Q_W^\beta}{n!}.
$$

Moreover, $\QDM_{\bbC^*}(W)$ is the graded module $H_{\bbC^*}^*(W)[z]\formal{Q_W,\theta}$ which is a trivial $H_{\bbC^*}^*(W)$ bundle over the $(z,\theta,Q_W)$-space. The quantum connection $\nabla\colon\QDM_{\bbC^*}(W)\to z^{-1}\QDM_{\bbC^*}(W)$ is the flat connection given by the following operators:
\begin{align*}
    &\nabla_{\theta^{i,n}} = \partial_{\theta^{i,n}} + z^{-1} \left( \varphi_i\lambda^n \star_{\theta} \right),\\
    &\nabla_{z \partial_z} = z \partial_z - z^{-1} \big( E_W^{\bbC^*} \star_{\theta} \big) + \mu_W^{\bbC^*} ,\\
    &\nabla_{\xi Q \partial_Q} = \xi Q \partial_Q + z^{-1} ( \xi \star_{\theta} ).
\end{align*}
Here, $E_W^{\bbC^*}$ is the equivariant Euler vector field and $\mu_W^{\bbC^*}$ is the equivariant grading operator, defined respectively by 
$$
    E_W^{\bbC^*}=c_1^{\bbC^*}(W)+\sum_{i,n}\bigg(1-\frac{\deg\varphi_i}{2}-n\bigg)\theta^{i,n}\varphi_i\lambda^n, \quad \mu_W^{\bbC^*}(\varphi_i\lambda^n)=\frac{\deg\varphi_i+2n-\dim_{\bbC}W}{2}\varphi_i\lambda^n.
$$
Moreover, we have $\xi Q\partial_Q\cdot Q_W^\beta = (\beta, \xi)Q_W^\beta$ for any $\xi \in H^2_{\bbC^*}(W)$. For $f,g\in \QDM_{\bbC^*}(W)$, the $z$-sesquilinear pairing $P_W$ is defined by 
$$
    P_W(f,g) = \int_W f(-z)\cup g(z)
$$
and satisfies $dP_W(f,g)= P_W(f,\nabla g) + P_W(\nabla f,g)$. 

The equivariant fundamental solution $M_W(\theta)\in \End_{\bbC[\lambda]}(H_{\bbC^*}^*(W))\formal{z^{-1}}\formal{Q_W,\theta}$ is defined by 
$$
    \int_W M_W(\theta)\varphi_i \cup \varphi_j = (\varphi_i, \varphi_j)_W^{\bbC^*} + \sum_{\substack{\beta \in \NE_{\bbN}(W),\, n \geq 0 \\ (n,\beta) \ne (0,0)}} 
    \Big{\langle} \varphi_i, \theta, \ldots, \theta, \frac{\varphi_j}{z-\psi} \Big{\rangle}^{W,\bbC^*}_{0, n+2, \beta} \frac{Q_W^\beta}{n!}
$$
and satisfies that 
\begin{align*}
    M_W(\theta) \circ \nabla_{\theta^{i,n}} &= \partial_{\theta^{i,n}} \circ M_W(\theta), \\
    M_W(\theta) \circ \nabla_{z \partial_z} &= \left(z \partial_z - z^{-1} c_1^{\bbC^*}(W) + \mu_W^{\bbC^*} \right) \circ M_W(\theta), \\
    M_W(\theta) \circ \nabla_{\xi Q \partial_Q} &= \left(\xi Q \partial_Q + z^{-1} \xi \right) \circ M_W(\theta), \\
    P_W^{\bbC^*}(M_W(\theta)f,M_W(\theta)g)&= P_W^{\bbC^*}(f,g).
\end{align*}
In particular, the $z$-direction is formal because the $\psi$-class may carry a nontrivial $\bbC^*$-weight and may not be nilpotent. However, through virtual localization \cite{graber1999localization}, we obtain $M_W(\theta)\in \End_{\bbC[\lambda]}(H_{\bbC^*}^*(W))\otimes_{\bbC[\lambda]}\bbC(\lambda,z)_{\mathrm{hom}}\formal{Q_W,\theta}$ where $\bbC(\lambda,z)_{\mathrm{hom}}=\bbC(\lambda/z)[z,z^{-1}]$ is the localization of $\bbC[\lambda,z]$ with respect to the nonzero homogeneous elements. 

The equivariant $J$-function $J_W(\theta)=M_W(\theta)1$ expands as 
$$
    J_W(\theta)=1+\frac{\theta}{z}+\sum_{\substack{\beta \in \NE_{\bbN}(W),\, n \geq 0 \\ (n,\beta) \ne (0,0),(1,0)\\ i=0,1,\ldots,n_W-1}} 
    \varphi^i\Big{\langle} \theta, \ldots, \theta, \frac{\varphi_i}{z(z-\psi)}  \Big{\rangle}^{W,\bbC^*}_{0, n+1, \beta} \frac{Q_W^\beta}{n!}.
$$

Beyond these structures, the equivariant quantum $D$-module $\QDM_{\bbC^*}(W)$ also carries the action of the \emph{shift operators} parameterized by equivariant curve classes, which we will define below.

\subsection{Notation on $\bbC^*$-fixed loci}\label{sect:C*FixedLocus}
The $\bbC^*$-fixed locus of $W$ is nonempty and can be decomposed into connected components as
$$
    W^{\bbC^*} = F_1 \sqcup \cdots \sqcup F_r.
$$
For each $i = 1, \dots, r$, the component $F_i$ is smooth and projective, and the $\bbC^*$-action on $W$ induces a globally diagonalizable $\bbC^*$-action on the normal bundle $N_{F_i/W}$ of $F_i$ in $W$ \cite{borel1966linear}. We have a $\bbC^*$-weight decomposition 
\begin{equation}\label{eqn:NormalWtDecomp}
    N_{F_i/W} = \bigoplus_{c \in \bbZ}  N_{F_i, c}
\end{equation}
where $N_{F_i, c}$ is the eigenbundle with weight $c \in \bbZ$.

Let $\omega$ be a K\"ahler form on $W$ that is invariant under the action of the maximal compact subgroup $S^1$ of $\bbC^*$. By Frankel's theorem \cite{frankel1958fixed}, the $S^1$-action on the K\"ahler manifold $(W, \omega)$ is Hamiltonian and induces a moment map
$$
    \mu: W \to \bbR.
$$
The map $\mu$ is a Morse-Bott function with critical locus $W^{\bbC^*}$ \cite{frankel1958fixed,carrell1979some}. Its value on each component $F_i$ is constant, and thus induces a partial order among the components $F_1, \dots, F_r$. In particular, there is a unique \emph{highest} (resp.\ \emph{lowest}) component on which $\mu$ achieves maximum (resp.\ minimum), and we assume that this component is $F_1$ (resp.\ $F_r$).

\subsection{Equivariant curve classes}\label{sect:EquivCurveClass}
Let $\NS^{\bbC^*}(W)$ be the $\bbC^*$-equivariant N\'eron-Severi group of $\bbC^*$-linearized line bundles on $W$ modulo algebraic equivalence, and define $N^1_{\bbC^*}(W)\subset H^2_{\bbC^*}(W,\bbZ)$ as the image of $\NS^{\bbC^*}(W)$ under the $\bbC^*$-equivariant first Chern class $c_1^{\bbC^*}\colon \NS^{\bbC^*}(W)\to H^2_{\bbC^*}(W,\bbZ)$.

The dual picture is introduced in \cite[Section 2.5]{iritani2023quantum}. Let $H_2^{\bbC^*}(W, \bbZ)$ denote the second integral homology group $H_2(W_{\bbC^*}, \bbZ)$ of the Borel construction $W_{\bbC^*} = (W \times E\bbC^*) / \bbC^*$, where $E\bbC^* \to B\bbC^*$ denotes the universal $\bbC^*$-bundle. In particular, the group $H_2^{\bbC^*}(\pt, \bbZ)$ is identified with the cocharacter group $\Hom(\bbC^*, \bbC^*) \cong \bbZ$ of $\bbC^*$. Let $N_1^{\bbC^*}(W)$ be the subgroup of $H_2^{\bbC^*}(W, \bbZ)$ generated by the following two types of classes:
\begin{enumerate}[wide]
    \item The classes of algebraic curves contained in a fiber of the $W$-bundle $W_{\bbC^*} \to B\bbC^*$.

    \item The classes $s_{x,*}(\beta)$, $\beta \in H_2(B\bbC^*, \bbZ)$, resulting from the pushforward along a section $s_x \colon B\bbC^* = \{x\}\times B\bbC^* \subset W_{\bbC^*}$ induced by a $\bbC^*$-fixed point $x \in W^{\bbC^*}$. 
\end{enumerate}
There is a commutative diagram
\begin{equation}
\begin{tikzcd}
\label{N_1-exact sequence}
    0 \arrow[r] & N_1(W) \arrow[r] \arrow[d] & N_1^{\bbC^*}(W) \arrow[r] \arrow[d] & H_2^{\bbC^*}(\pt, \bbZ) \arrow[r] & 0 \\
    0 \arrow[r] & H_2(W, \bbZ) \arrow[r] & H_2^{\bbC^*}(W, \bbZ) \arrow[r]  & H_2^{\bbC^*}(\pt, \bbZ) \arrow[r] \arrow[u, equal] & 0
\end{tikzcd}
\end{equation}
where each row is a short exact sequence \cite[Lemma 2.4]{iritani2023quantum}. For any point $x\in W^{\bbC^*}$, the pushforward $H_2^{\bbC^*}(\{x\},\bbZ)\to H_2^{\bbC^*}(W,\bbZ)$ along the inclusion $x \in W$ induces a splitting of both rows of \eqref{N_1-exact sequence}.

For each connected component $F_i$ of $W^{\bbC^*}$ and $k \in H_2^{\bbC^*}(\pt, \bbZ) \cong \bbZ$, we define the section class
$$
    \sigma_{F_i}(k) \coloneqq s_{x,*}(k) \in N_1^{\bbC^*}(W)
$$
where $x \in F_i$.

We use $(-, -)$ to denote the pairing between $H_2^{\bbC^*}(W)$ and $H^2_{\bbC^*}(W)$ as well as the induced pairing between $N_1^{\bbC^*}(W)$ and $N^1_{\bbC^*}(W)$. For $\bbF=\bbQ,\bbR,\bbC$, we denote the $\bbF$-vector spaces 
$$
    N_1^{\bbC^*}(W)_{\bbF} = N_1^{\bbC^*}(W) \otimes_{\bbZ} \bbF, \quad N^1_{\bbC^*}(W)_{\bbF} = N^1_{\bbC^*}(W) \otimes_{\bbZ} \bbF.
$$
The induced pairing between $N_1^{\bbC^*}(W)_{\bbF}$ and $N^1_{\bbC^*}(W)_{\bbF}$ is perfect over ${\bbF}$.

\subsection{Shift operators on equivariant quantum $D$-modules}\label{sect:ShiftOperator}
In this subsection, we introduce the shift operators on the equivariant quantum $D$-module $\QDM_{\bbC^*}(W)$. The shift operators were first introduced by Okounkov-Pandharipande \cite{okounkov2010quantum} for cocharacters of $\bbC^*$, which generalize the Seidel representations from quantum cohomology. Iritani \cite{iritani2023quantum} subsequently lifted this construction to equivariant classes in $N_1^{\bbC^*}(W)$.

For any cocharacter $k\in\Hom(\bbC^*,\bbC^*)\cong \bbZ$, there is an induced action of $k(\bbC^*) \cong \bbC^*$ on $W$. Let $F_{k,\max}\subset W^{k(\bbC^*)}$ denote the highest fixed component of this action and set $\sigma_{\max}(k) = \sigma_{F_{k,\max}}(k)$. 
Given any $\beta\in N_1^{\bbC^*}(W)$, let $\Bar{\beta}$ denote its image under the pushforward to $H_2^{\bbC^*}(\pt; \bbZ)$. It holds that $\beta+\sigma_{\max}(-\Bar{\beta})$ lies in the image of $N_1(W)\to N_1^{\bbC^*}(W)$. We therefore identify $\beta+\sigma_{\max}(-\Bar{\beta})$ with its preimage in $N_1(W)$. 

\begin{proposition}[\cite{okounkov2010quantum, iritani2023quantum}]\label{Shift operator proposition}
There is a family of $\bbC[z]\formal{Q_W,\theta}$-linear operators
$$
    \hS^{\beta}=\hS^{\beta}(\theta)\colon \QDM_{\bbC^*}(W)\to Q_W^{\beta+\sigma_{\max}(-\Bar{\beta})}\QDM_{\bbC^*}(W)$$
parameterized by $ \beta\in N_1^{\bbC^*}(W)$, called the \emph{extended shift operators} of $W$, that satisfies the following properties:
\begin{enumerate}[wide]
    \item $\hS^{\beta}(f(\lambda,z)\alpha)=f(\lambda-\Bar{\beta}z,z)(\hS^{\beta}\alpha)=e^{-\Bar{\beta}z\partial_{\lambda}}(f(\lambda,z))(\hS^{\beta}\alpha)$ for any $f(\lambda,z)\in \bbC[\lambda,z]$ and $\alpha\in \QDM_{\bbC^*}(W)$.
    
    \item \label{shift operator (2)} $\hS^{\beta_1+\beta_2}=\hS^{\beta_1}\circ \hS^{\beta_2}$ for any $\beta_1,\beta_2\in N_1^{\bbC^*}(W)$.
    
    \item \label{shift operator (3)} $\hS^{\beta}=Q_W^{\beta}$ when $\beta\in N_1(W)\subset N_1^{\bbC^*}(W)$. In particular, $\hS^0=\Id$.
    
    \item $[\nabla_{\theta^{i,n}},\hS^{\beta}]=[\nabla_{z\partial_z},\hS^{\beta}]=0$ and $[\nabla_{\xi Q\partial_Q},\hS^{\beta}] = (\beta, \xi) \hS^{\beta}$ for any $\xi \in H^2_{\bbC^*}(W)$.
    
    \item $e^{-\Bar{\beta}z\partial_{\lambda}}P_{W}^{\bbC^*}(f,g)=P_W^{\bbC^*}(\hS^{-\beta}f,\hS^{\beta}g)$ for any $f, g \in  \QDM_{\bbC^*}(W)$.

    \item $\hS^{\beta}$ is homogeneous of degree $2(\beta, c_1^{\bbC^*}(W))$.
\end{enumerate}
\end{proposition}

We refer to \cite[Definition 2.6]{iritani2023quantum} for the explicit construction of the extended shift operators. By Proposition~\ref{Shift operator proposition}\eqref{shift operator (2)}\eqref{shift operator (3)}, these operators generalize the action of the Novikov variables to the equivariant context. For this reason, we will also refer to them as \emph{equivariant Novikov variables} throughout the paper.

\subsection{Givental formalism}\label{subsec:Givental formalism}
In this subsection, we review the Givental formalism of genus-zero Gromov-Witten theory. The \emph{Givental space} $\cH_X$ of $X$ is the symplectic vector space $H^*(X)\dbp{z^{-1}}\formal{Q_X}$ equipped with the symplectic form 
$$
    \Omega(f,g) = -\mathrm{Res}_{z=\infty} \bigg(\int_X f(-z)\cup g(z)\bigg) dz.
$$
It admits a maximal isotropic subspace decomposition $\cH_X=\cH_X^+\oplus \cH_X^-$ based on the degree of $z$, where $\cH_X^+=H^*(X)[z]\formal{Q_X}$ and $\cH_X^-=z^{-1}H^*(X)\formal{z^{-1}}\formal{Q_X}$. The symplectic form $\Omega$ identifies $\cH_X^{-}$ with the dual of $\cH_X^{+}$ and identifies $\cH_X$ with $T^*\cH_{X}^+$. Similarly, the equivariant Givental space $\cH_W$ of $W$ is the symplectic vector space $H^*_{\bbC^*}(W)\dbp{z^{-1}}\formal{Q_W}$) equipped with the symplectic form 
$$
    \Omega(f,g) = -\mathrm{Res}_{z=\infty} \bigg(\int_W f(-z)\cup g(z) \bigg) dz.
$$
It admits a maximal isotropic subspace decomposition $\cH_W = \cH_W^{+}\oplus \cH_W^{-}$ where $\cH_W^{+} = H^*_{\bbC^*}(W)[z]\formal{Q_W}$ and $\cH_W^{-}=z^{-1}H^*_{\bbC^*}(W)\formal{z^{-1}}\formal{Q_W}$. The symplectic form $\Omega$ identifies $\cH_{W}^{-}$ with the dual of $\cH_W^{+}$ and identifies $\cH_W$ with $T^*\cH_{W}^+$. 

The genus-zero descendant Gromov-Witten potential $\cF_X$ and its equivariant counterpart $\cF_W$ are respectively defined by
$$
    \cF_X(z + \bft(z)) = 
    \sum_{\substack{n \geq 0,\, \beta \in \NE_{\bbN}(X) \\ (n,\beta) \ne (0,0), (1,0), (2,0)}}
    \big\langle \bft(-\psi), \ldots, \bft(-\psi) \big\rangle^{X}_{0,n,\beta} \frac{Q_X^\beta}{n!},
$$
$$
    \cF_W(z + \bft(z)) = 
    \sum_{\substack{n \geq 0,\, \beta \in \NE_{\bbN}(W) \\ (n,\beta) \ne (0,0), (1,0), (2,0)}}
    \big\langle \bft(-\psi), \ldots, \bft(-\psi) \big\rangle^{W,\bbC^*}_{0,n,\beta} \frac{Q_W^\beta}{n!}.
$$
In the former case, $\bft(z)=\sum\limits_{i=0}^{\infty}t_nz^n\in \cH_X^{+}$ is a formal coordinate of $\cH_{X}^{+}$ centered at $z=0$, with $t_n\in H^*(X)\formal{Q_X}$. In the latter case, $\bft(z)=\sum\limits_{i=0}^{\infty}t_nz^n\in \cH_W^{+,\bbC^*}$ is a formal coordinate of $\cH_W^{+,\bbC^*}$ centered at $z=0$, with $t_n \in H^*_{\bbC^*}(W)\formal{Q_W}$.


The \emph{Givental cone} $\cL_X\subset \cH_X$ of $X$ is the graph of $d\cF_X$ under the identification $\cH_X \cong T^*\cH_{X}^+$. The equivariant Givental cone $\cL_W \subset \cH_W$ of $W$ is the graph of $d\cF_{W}$ under the identification $\cH_W \cong T^*\cH_{W}^{+}$. Explicitly, the Givental cones consist of points of form 
\begin{equation}
\label{expression of big J-function}
\begin{split}
    \cJ(\bft(z))&= z + \bft(z) + \sum_{i=0}^{n_X-1} \sum_{\substack{n \geq 0,\, \beta \in \NE_{\bbN}(X) \\ (n,\beta) \ne (0,0), (1,0)}}
    \phi^i \Big{\langle} \bft(-\psi), \ldots, \bft(-\psi), \frac{\phi_i}{z - \psi} \Big{\rangle}^{X}_{0,n+1,\beta} \frac{Q_X^\beta}{n!},\\
    \cJ(\bft(z))&=z + \bft(z) + \sum_{i=0}^{n_W-1} \sum_{\substack{n \geq 0,\, \beta \in \NE_{\bbN}(W) \\ (n,\beta) \ne (0,0), (1,0)}}
    \varphi^i \Big{\langle} \bft(-\psi), \ldots, \bft(-\psi), \frac{\varphi_i}{z - \psi} \Big{\rangle}^{W,\bbC^*}_{0,n+1,\beta} \frac{Q_W^\beta}{n!}
\end{split}
\end{equation}
respectively. The $J$-functions satisfy that $zJ_X(\tau)\in \cL_X$ for any $\tau\in H^*(X)\formal{Q_X}$ and $zJ_W^{\bbC^*}(\theta)\in \cH_{W}^{\rat} \cap \cL_{W}$ for any $\theta\in H^*_{\bbC^*}(W)\formal{Q_W}$. Here, $\cH_W^{\rat}$ is the \emph{rational} Givental space $H_{\bbC^*}^*(W)\otimes_{\bbC[\lambda]}\bbC(\lambda,z)_{\mathrm{hom}}\formal{Q_W}$ and the statement that $zJ_W(\theta)\in \cH_{W}^{\rat}$ follows from virtual localization.

\begin{proposition}[see \cite{givental2004symplectic}]\label{proposition of Givental cone}
The Givental cones satisfy the following properties:
\begin{enumerate}[wide]
    \item $\cL_{W}\subset\cH_W^{\bbC^*}\cap \cH_W^{\rat}$.
    \item \label{prop: Givental cone 2}Given any $\tau\in H^*(X)\formal{Q_X}$ (resp.\ $\theta\in H_{\bbC^*}^*(W)\formal{Q_W}$), we denote the tangent space of $zJ_X(\tau)$ (resp.\ $zJ_W(\theta)$) in $\cL_{X}$ (resp.\ $\cL_{W}$) by $T_{\tau}$ (resp.\ $T_{\theta}$). Then $T_{\tau}$ (resp.\ $T_\theta$) is a $\bbC[z]\formal{Q_X}$-module (resp.\ $\bbC[z]\formal{Q_W}$-module) freely generated by $M_{X}(\tau)\phi_i=z\partial_{\tau^i}J_X(\tau)$ (resp.\ $M_{W}(\theta)\varphi_i\lambda^k=z\partial_{\theta^{i,k}}J_W(\theta)$). We have $T_{\tau}=M_{X}(\tau)\cH_{X}^{+}$ and $T_{\theta}=M_{W}(\tau)\cH_{W}^{+}$.
    
    \item \label{prop:Givental cone 3} $\cL_X=\cup_{\tau\in H^*(X)\formal{Q_X}}zT_{\tau}$, $\cL_{W}=\cup_{\theta\in H^*_{\bbC^*}(W)\formal{Q_W}}zT_{\theta}$.
    \item Given any $\tau\in H^*(X)\formal{Q_X}$ (resp.\ $\theta\in H^*_{\bbC^*}(W)\formal{Q_W}$), the tangent space $T_{\tau}$ (resp.\ $T_\theta$) is preserved by the differential operators $z\partial_{\tau^i}$, $z\xi Q\partial_Q+\xi$, and $z^2\partial_z-c_1(X)+z\mu_X$ (resp.\ $z\partial_{\theta^i}$, $z\xi Q\partial_Q+\xi$, and $z^2\partial_z-c_1^{\bbC^*}(W)+z\mu_W^{\bbC^*}$), where $\xi$ is any element in $H^2(X)$ (resp.\ in $H^2_{\bbC^*}(W)$).
    
    \item Any point in the (equivariant) Givental cone is uniquely determined by its part with nonnegative $z$-powers. Specifically, for any point $\cJ$ in $\cL_{X}$ (resp.\ $\cL_{W}$) considered as a Laurent series in $z^{-1}$, if $z+\bft(z)$ denotes the part with nonnegative $z$-powers of $\cJ$, we have $\cJ=\cJ(\bft(z))$ as in the expression \eqref{expression of big J-function}. 
\end{enumerate}
\end{proposition}

\subsection{Shift operators on equivariant Givental cones}
The shift operators defined in Section~\ref{sect:ShiftOperator} also have an analogue on the equivariant Givental cone $\cL_{W}$. In fact, they are defined on the rational Givental space $\cH_{W}^{\rat}$ in the following explicit way, following \cite[Definition 2.8]{iritani2023quantum}.

\begin{definition}
\label{formula for shift operator on Givental cone}
Given any $\beta\in N_1^{\bbC^*}(W)$, the shift operator $\hcS^{\beta} \colon \cH_W^{\rat}\to Q_W^{\beta+\sigma_{\max}(-\Bar{\beta})}\cH_W^{\rat}$ is defined via localization by
$$
    \hcS^{\beta}(\bff)|_{F} = Q_W^{\beta+\sigma_F(-\Bar{\beta})} \prod\limits_{c \in \bbZ} \frac{\prod_{m=-\infty}^{0}e_{\lambda+mz}(N_{F,c})}{\prod_{m=-\infty}^{-(\Bar{\beta}, c\lambda)}e_{\lambda+mz}(N_{F,c})}e^{-z\Bar{\beta}\partial_{\lambda}} \bff|_F,
$$
where $\bff \in \cH_{W}^{\rat}$, $F$ ranges through connected components of $W^{\bbC^*}$, $N_{F,c}$ is the eigenbundle in \eqref{eqn:NormalWtDecomp}, and $e_{\lambda}$ is the equivariant Euler class.
\end{definition}

\begin{proposition}[\cite{iritani2023quantum}]\label{proposition of shift operators on Givental cone}
The shift operators $\hcS^{\beta}$ are $\bbC[z]\formal{Q_W}$-module homomorphisms with the following properties:
    \begin{enumerate}[wide]
        \item $\hcS^{\beta_1+\beta_2}=\hcS^{\beta_1}\circ \hcS^{\beta_2}$ for any $\beta_1,\beta_2\in N_1^{\bbC^*}(W)$. In particular, $\hcS^{0}=\Id$.
        
        \item \label{shift operator commute with funda} $\hcS^{\beta}\circ M_W(\theta)=M_W(\theta)\circ \hS^{\beta}$, where $\hS^{\beta}$ is the shift operator on $\QDM_{\bbC^*}(W)$ defined in Proposition~\ref{Shift operator proposition}.
        
        \item $[\partial_{\theta^{i,n}},\hcS^{\beta}]=[z\partial_z-z^{-1}c_1^{\bbC^*}(W) +\mu_W^{\bbC^*}, \hcS^{\beta}]=0, [\xi Q\partial_Q+z^{-1}\xi, \hcS^{\beta}]=(\beta, \xi)\hcS^{\beta}$ for any $\xi \in H^2_{\bbC^*}(W)$.
        
        \item $\hcS^{\beta}(T_{\theta})\subset Q_W^{\beta+\sigma_{\max}(-\Bar{\beta})}T_{\theta}$ for any $\theta \in H^*_{\bbC^*}(W)\formal{Q_W}$. In particular, for $\beta+\sigma_{\max}(-\Bar{\beta})\in \NE_{\bbN}(W)$, $\hcS^{\beta}(\cL_{W})\subset \cL_{W}$, i.e.\ $\hcS^{\beta}$ preserves the Givental cone.
    \end{enumerate}
\end{proposition}

\subsection{Twisted Gromov-Witten invariants and quantum Riemann-Roch theorem}
In this section, we review the twisted Gromov-Witten invariants and the quantum Riemann-Roch theorem of \cite{coates2007quantum}.
Let $\bfs(-) = \mathrm{exp}\big(\sum\limits_{i=0}^{\infty}s_i\mathrm{ch}_i(-)\big)$ be the universal invertible multiplicative characteristic class, where $s_i$ is a formal variable with degree $-2i$. For instance, the inverse equivariant Euler class $e_{\lambda}^{-1}$ corresponds to the parameters 
$$
    s_i = \begin{cases}
        -\mathrm{log}\:\lambda & \text{for } i = 0,\\
        (-1)^i(i-1)!\lambda^{-i} & \text{for } i>0. 
    \end{cases}
$$

Let $V$ be a vector bundle on $X$. Let
$$
    \langle\alpha_1\psi^{k_1},\cdots,\alpha_n\psi^{k_n}\rangle^{X,(V,\bfs)}_{0,n,\beta}
$$
denote the genus-zero, degree-$\beta$, $n$-pointed, $(V,\bfs)$-twisted (descendant) Gromov-Witten invariant of $X$. The $(V,\bfs)$-twisted Givental space of $X$ is the symplectic vector space 
$$
    \cH^{\tw}_X=H^*(X)[z,z^{-1}]\formal{Q_X,\bfs}=H^*(X)[z,z^{-1}]\formal{Q_X}\formal{s_0,s_1,\ldots}
$$ 
equipped with the symplectic form 
$$
    \Omega^{\tw}(f,g) = -\mathrm{Res}_{z=\infty} \bigg( \int_{X, (V,\bfs)} f(-z) \cup g(z) \bigg) dz,
$$
where $\int_{X, (V,\bfs)}$ is the $(V,\bfs)$-twisted Poincar\'e pairing such that $\int_{X, (V,\bfs)} \alpha \cup \beta = \int_X \alpha \cup \beta \cup \bfs(V)$. There is a maximal isotropic decomposition $\cH_X^{\tw} = \cH_{X}^{+,\tw}\oplus\cH_{X}^{-,\tw}$ where $\cH_{X}^{+,\tw} = H^*(X)[z]\formal{Q_X,\bfs}$ and $\cH_{X}^{-,\tw} = z^{-1}H^*(X)[z^{-1}]\formal{Q_X,\bfs}$. The $(V,\bfs)$-twisted Givental cone $\cL^{\tw}_X\subset \cH_X^{\tw}$ is defined as the graph of the genus-zero twisted Gromov-Witten potential of $X$ under the identification $\cH^{\tw}_X \cong T^*\cH_{X}^{+,\tw}$. Explicitly, $\cL_X^{\tw}$ consists of points of form
$$
    \cJ^{\tw}_X(\bft(z))=
    z + \bft(z) + \sum_{i=0}^{n_X-1} \sum_{\substack{n \geq 0,\, \beta \in \NE_{\bbN}(X) \\ (n,\beta) \ne (0,0),(1,0)}} \frac{\phi^i_{\tw}}{\bfs(V)} \Big{\langle} \bft(-\psi), \ldots, \bft(-\psi), \frac{\phi_i}{z - \psi} \Big{\rangle}^{X,(V,\bfs)}_{0, n+1, \beta} \frac{Q_X^\beta}{n!}
$$
where $\{\phi^i_{\tw}\}$ is the $\bbC$-dual basis of $\{\phi_i\}$ with respect to the pairing $\int_{X, (V,\bfs)}$. We will also use the notation $\cL_{X, (V,\bfs)}$ for $\cL^{\tw}_X$.


The quantum Riemann-Roch theorem relates the $(V,\bfs)$-twisted theory to the untwisted theory through the operator 
$$
    \Delta_{(V,\bfs)}=\mathrm{exp}\Bigg(\sum_{\substack{l\ge 0,m\ge 0,l+m\ge 1}}s_{l+m-1}\frac{B_m(-z)^{m-1}}{m!}\mathrm{ch}_l(V)\Bigg)\colon\cH_X \longrightarrow \cH^{\tw}_X
$$ 
where $B_m$ are the Bernoulli numbers given by $\frac{x}{e^x-1}=\sum\limits_{m=0}^\infty \frac{B_m}{m!}x^m$.

\begin{theorem}[{\cite[Corollary 4]{coates2007quantum}}] \label{thm:QRR}
We have $\Delta_{(V,\bfs)}(\cL_X)=\cL_X^{\tw}$.
\end{theorem}

\section{Geometry of VGIT wall-crossings and $\bbC^*$-actions}\label{sect:GIT}
In this section, we discuss the framework of variations of GIT quotients. We further specialize to the case of $\bbC^*$-actions and describe the geometry.

\subsection{GIT setup}\label{sect:Setup}
Let $W$ be a smooth, quasi-projective variety that is projective over an affine variety, and $G$ be a reductive algebraic group with a linear action on $W$. Let $\NS^G(W)$ be the $G$-equivariant N\'eron-Severi group of $G$-linearized line bundles on $W$ modulo algebraic equivalence. Given an ample $G$-linearization $L$ in $\NS^G(W)$, or more generally $\NS^G(W) \otimes \bbQ$ or $\NS^G(W) \otimes \bbR$, let $W^s(L)$ and $W^{ss}(L)$ denote the stable and semistable loci respectively, and let $W \gitquot_L G$ denote the induced GIT quotient.

Let
$$
    E^G(W) \coloneqq \{ L \in \NS^G(W) \otimes \bbR \ \big|\ W^{ss}(L) \neq \emptyset\}
$$
which is a convex cone. It is shown in \cite{dolgachev1998variation, thaddeus1996geometric} that $E^G(W)$ admits a finite wall-and-chamber structure, where the union of all walls consists of all linearizations $L$ such that $W^{ss}(L) \neq W^{s}(L)$, and each chamber is an equivalence class of linearizations $L$ for which $W^{ss}(L) = W^{s}(L)$ is the same. In particular, within each chamber, the GIT quotient is invariant.

Let $L \in E^G(W)$ such that $W^{ss}(L) = W^{s}(L)$ and let $X = W\gitquot_L G$. The associated \emph{Kirwan map} is defined by the pullback
$$
     \kappa_{X} \colon H_G^*(W) \xrightarrow{i^*_{W^s(L)}} H_G^*(W^{s}(L)) \cong H^*(X)
$$
along the inclusion $i_{W^s(L)} \colon W^s(L) \to W$ and is surjective \cite{kirwan1984cohomology}. 
We make the following simple observation.

\begin{lemma}\label{lem:KirwanC1}
Suppose $G$ acts freely on $W^{s}(L)$ and $X$ is smooth. We have $\kappa_{X}(c_1^G(W)) = c_1(X)$.
\end{lemma}

\begin{proof}
We have 
$$
    \kappa_{X}(c_1^G(W)) = i_{W^{s}(L)}^*(c_1^G(TW)) =c_1^G(i^*_{W^{s}(L)}TW) =c_1^G(TW^{s}(L)) =c_1(X)
$$
where the third equality uses that $TW^{s}(L)=i^*_{W^{s}(L)}TW$, since $W^{s}(L)$ is open dense in $W$, and the last equality follows from the freeness of $G$-action on $W^{s}(L)$.
\end{proof}

\subsection{VGIT wall-crossings and standard flips}\label{sect:VGIT}
We now consider the birational transformation between two GIT quotients that are related by a single wall-crossing in the space of ample linearizations. 
Let $L_+, L_- \in \NS^G(W)$ be two ample $G$-linearized line bundles that lie in adjacent chambers of $E^G(W)$. For any $t \in [-1, 1]\cap \bbQ$, write $L_t \coloneqq L_+^{(1+t)/2} \otimes L_-^{(1-t)/2}$. We refer to the birational map $W\gitquot_{L_-}G\dashrightarrow W\gitquot_{L_+}G$ as a \emph{($G$-)VGIT wall-crossing}. In this paper, we focus on the following type of wall-crossings, defined following \cite{dolgachev1998variation,thaddeus1996geometric,liu2025invariance}.


\begin{definition}\label{def:SimpleWall}
The $G$-VGIT wall-crossing $W\gitquot_{L_-}G\dashrightarrow W\gitquot_{L_+}G$ is said to be \emph{simple} if it satisfies the following conditions:
\begin{enumerate}[wide]
    \item \label{cond:t0} There exists $t_0 \in (-1, 1)\cap\bbQ$ such that $W^{ss}(L_t) = W^{ss}(L_+)$ for any $t > t_0$ and $W^{ss}(L_t) = W^{ss}(L_-)$ for any $t < t_0$.
    
    \item The set $W^0 \coloneqq W^{ss}(L_{t_0}) \setminus (W^{ss}(L_+) \cup W^{ss}(L_-))$ is nonempty and connected.
    
    \item \label{cond:SimpleWall} (\cite[Hypothesis 4.4]{thaddeus1996geometric})   For any $x\in W^0$, we have $G_{x} \cong \bbC^*$.
    
    \item \label{cond:WtGcd} For any $x\in W^0$, if $v_{\pm}$ are the weights of the $G_x$-action on the fiber of $L_{\pm}$ at $x$, we have $\mathrm{gcd}(v_+,v_-)=1$.
    
    \item \label{cond:WtPm1} For any $x \in W^0$, the nonzero weights of the $G_x$-action on $N_{G\cdot x/W}$ can only be $\pm 1$.
    
    \item The GIT quotients $W \gitquot_{L_\pm} G$ are smooth and \emph{projective}.

\end{enumerate}
\end{definition}

When $G$ is an algebraic torus, conditions \eqref{cond:SimpleWall} and \eqref{cond:WtGcd} in Definition~\ref{def:SimpleWall} always hold (see \cite{thaddeus1996geometric}). We follow the notation of \cite{thaddeus1996geometric} and set $W^{ss}(L_\pm) =\colon W^{ss}(\pm)$, $W^{ss}(L_{t_0}) =\colon W^{ss}(0)$, and $W^{\pm}=W^{ss}(0)\backslash W^{ss}(\mp)$. For the quotients, we denote $W \gitquot_{L_\pm} G =\colon W \gitquot G(\pm)$ and $W \gitquot_{L_{t_0}} G =\colon W \gitquot G (0)$. The geometry of the situation can be described as follows.

\begin{proposition}[\cite{dolgachev1998variation,thaddeus1996geometric}]\label{prop:VGITGeometry}
For a simple VGIT wall-crossing $W\gitquot G(-)\dashrightarrow W\gitquot G(+)$, we have:
\begin{enumerate}[wide]
    \item The natural morphisms $\pi_\pm\colon W\gitquot G(\pm)\to W\gitquot G(0)$ are proper and birational.
    
    \item $W^{\pm}\gitquot G(\pm)$ is isomorphic to $\bbP(N_{\pm})$ for some vector bundle $N^{\pm}$ over $W^0\gitquot G(0)$.
    
    \item $W^{\pm}\gitquot G(\pm)$ admits a regular embedding into $W\gitquot G(\pm)$. The normal bundle of $W^{\pm}\gitquot G(\pm)$ in $W\gitquot G$ is $\psi_{\pm}^*(N^{\mp})(-1)$, where $\psi_{\pm}$ is the projection $W^{\pm}\gitquot G(\pm)\to W^0\gitquot G(0)$.
    
    \item The blowups of $W \gitquot G(\pm)$ at $W^{\pm}\gitquot G(\pm)$ and the blowup of $W\gitquot G(0)$ at $W^0\gitquot G(0)$ are all isomorphic to $W \gitquot G(+)\times_{W \gitquot G(0)} W \gitquot G(-)$.
    
    \item The normal bundle of the exceptional locus $W^+\gitquot G(+)\times_{W^0\gitquot G(0)}W^-\gitquot G(-) \cong \bbP(N_+) \times_{W^0\gitquot G(0)} \bbP(N_-)$ in $W \gitquot G(+)\times_{W \gitquot G(0)} W \gitquot G(-)$ is identified with $\cO(-1,-1)$.
\end{enumerate}
\end{proposition}

Therefore, the birational transformation $W \gitquot G(-) \dashrightarrow W \gitquot G(+)$ is a \emph{standard flip} of type $(\rank(N_+)-1, \rank(N_-)-1)$ with wall $W^0\gitquot G(0)$, in the following sense.

\begin{definition}
\label{def:Flip}
A birational transformation $f\colon X_- \dashrightarrow X_+$ between smooth projective varieties is a \emph{standard flip} of type $(r_+-1, r_--1)$ with \emph{wall} $S$ if it fits into the diagram
$$
    \begin{tikzcd}
        && E \arrow[rrdd, "p_+"]\arrow[lldd, "p_-", swap]\arrow[d, hookrightarrow]&&\\
        && \tX \arrow[ld, "\pi_-",swap]\arrow[rd, "\pi_+"]&&\\
        P_- \arrow[r,hookrightarrow, "j_-"]\arrow[rrdd,"\psi_-",swap] & X_- \arrow[rr,"f",dashed]\arrow[rd,"q_-",swap]&& X_+ \arrow[ld,"q_+"] & P_+ \arrow[l,hookrightarrow, "j_+",swap]\arrow[lldd,"\psi_+"]\\
        && X_0 &&\\
        && S \arrow[u,hookrightarrow]&&
    \end{tikzcd}
$$
where:
\begin{enumerate}[wide]
    \item The morphisms $q_\pm\colon X_\pm \to X_0$ are proper and birational with exceptional locus $\psi_\pm\colon P_\pm \to S$ respectively.
    
    \item There exist vector bundles $V_\pm$ over $S$ of rank $r_\pm $ such that $P_\pm \cong \bbP(V_\pm)$ and $\psi_\pm$ are identified with the projections.

    \item The normal bundles $N_{P_\pm/X_\pm}$ are identified with $\psi_\pm^*(V_\mp) (-1)$ respectively.
    
    \item $\tX$ is the common blowup of $X_\pm$ at $P_\pm$ and $X_0$ at $S$, with the exceptional locus $E \cong P_+ \otimes_S P_-$ whose normal bundle $N_{E/\tX}$ is identified with $\cO(-1, -1)$. We denote the projections $\tX\to X_{\pm}$ by $\pi_{\pm}$ and $E\to P_{\pm}$ by $p_{\pm}$.

\end{enumerate}
\end{definition}
There are two special cases. When $r_+ = 1$ (or $r_- = 1$), we recover the case of a blowup. When $r_+ = r_- (\neq 1)$, the birational transformation $f$ is referred to as an \emph{ordinary flop}.


\begin{remark}
\label{rem:definition of standard flip}
In the definition of standard flips and ordinary flops in the literature, e.g. \cite{lee2010flops, jiang2023chow, belmans2022derived, shen2025quantum}, the contraction $q_-$ is often required to be log-extremal, 
which implies that $q_+$ is also log-extremal (see \cite[Section 1]{lee2010flops}).
Our Definition~\ref{def:Flip} does not impose this condition. Nevertheless, the previous results on standard flips and ordinary flops used in this paper, specifically, the Chow motives decomposition (Proposition~\ref{prop:Chow motive decom of flips}) for standard flips and the invariance of quantum cohomology under ordinary flops up to analytic continuation (Proposition~\ref{Prop: big quantum ring identify in flop case}), remain valid without this condition.
\end{remark}

Jiang \cite[Corollary 3.10]{jiang2023chow} proved that for a standard flip, the correspondence given by graph closure of $f$ induces a decomposition of Chow motives, extending the results of Lee-Lin-Wang \cite{lee2010flops}. 

\begin{proposition}[\cite{lee2010flops,jiang2023chow}]\label{prop:Chow motive decom of flips}
Let $f\colon X_- \dashrightarrow X_+$ be a standard flip of type $(r_+-1, r_--1)$ with wall $S$ where $X_\pm$ are smooth and quasi-projective. Assume without loss of generality that $r_+ \le r_-$. There is an isomorphism of Chow motives
$$
    [\Gamma] \oplus \left( \bigoplus_{i = r_+ }^{r_--1} j_{-,*} \circ h_-^{r_--1 -i} \circ \psi^* \right)\colon 
    \mathfrak{h}(X_+) \oplus \left( \bigoplus_{i = r_+}^{r_--1} \mathfrak{h}(S)(i) \right)
    \overset{\sim}{\longrightarrow} \mathfrak{h}(X_-)
$$
where $[\Gamma]$ is the correspondence given by the graph closure $\Gamma\subset X_+ \times X_-$ of $f$, $j_-\colon P_- \to X_-$ is the inclusion, $h_-$ is the fiber class of $P_-$, and $\mathfrak{h}(-)$ denotes the Chow motive class.
\end{proposition}


When additionally $X_\pm$ are \emph{projective} and $S$ is smooth and projective, Manin's identity principle implies an isomorphism of cohomology groups
$$
    \Bigg( \phi,\bigoplus\limits_{i=0}^{r_- - r_+ -1}j_{-,*}\circ (\cup h_-^{i})\circ \psi^* \Bigg)\colon H^*(X_+)\oplus\bigoplus\limits_{i=0}^{r_- - r_+ -1}H^{*-2i-2r_+}(S)\longrightarrow H^*(X_-)
$$
where $\phi=\pi_{-,*}\circ \pi_+^*$ is the injective map induced by the correspondence $[\Gamma]^t$. Applying this to the VGIT wall-crossing framework, we have the following statement.

\begin{corollary}\label{cor:deRhamIsoVGIT}
    For a simple VGIT wall-crossing $W\gitquot G(-)\dashrightarrow W\gitquot G(+)$, let $r_\pm \coloneqq \rank(N_{\pm})$ and assume $r_+ \le r_-$ without loss of generality. We have an isomorphism of cohomology groups
$$
    \Bigg(\phi_1,\bigoplus\limits_{i=0}^{r_- - r_+ -1}j_{-,*}\circ (\cup h_-^{i})\circ \psi_-^*\Bigg)\colon H^*(W\gitquot G(+))\oplus\bigoplus\limits_{i=0}^{r_- - r_+-1}H^{*-2i-2r_+}(W^0\gitquot G(0))\longrightarrow H^*(W\gitquot G(-))
$$ 
where $\phi_1$ is the injective map induced by the graph closure of $W\gitquot G(-) \dashrightarrow W\gitquot G(+)$, $j_{\pm}$ is the inclusion $W^{\pm}\gitquot G(\pm)\hookrightarrow W\gitquot G(\pm)$, and $h_{\pm}$ is the fiber class of $W^{\pm}\gitquot G(\pm)$.
\end{corollary}

\begin{remark}\label{rem:KDomination}
Assuming $r_+ \le r_-$, \cite{halpern2015derived,ballard2019variation} provided a semi-orthogonal decomposition of the derived category of coherent sheaves $D^b(\mathrm{Coh}(W\gitquot G(-)))$ in which $D^b(\mathrm{Coh}(W\gitquot G(+)))$ is a summand. In particular, $W\gitquot G(-)$ \emph{$D$-dominates} $W\gitquot G(+)$ (see \cite[Section 4.3]{ballard2019variation}). Applying the Hochschild-Kostant-Rosenberg isomorphism provides an alternative decomposition of the cohomology groups. It is further shown in \cite[Section 4.3]{ballard2019variation} that $W\gitquot G(-)$ also \emph{$K$-dominates} $W\gitquot G(+)$, and hence $D$-dominance is equivalent to $K$-dominance in this VGIT wall-crossing setting. In particular, when $r_+ = r_-$, the varieties $W\gitquot G(\pm)$ are both $D$-equivalent and $K$-equivalent (see \cite{wang1998birational,kawamata2002Dequivalence,wang2003Kequivalence}). We also note that in the case of a standard flip, a semi-orthogonal decomposition of the derived categories of coherent sheaves is obtained by \cite{belmans2022derived}.
\end{remark}

\subsection{$\bbC^*$-actions and Bia{\l}ynicki-Birula decompositions}\label{subsec:BB decom}
For the rest of Section \ref{sect:GIT}, we specialize to the case where $W$ is a smooth \emph{projective} variety equipped with an action of $G= \bbC^*$.

Consider the decomposition 
$$
    W^{\bbC^*} = F_1 \sqcup \cdots \sqcup F_r
$$
as in Section~\ref{sect:C*FixedLocus}, where $F_1$ and $F_r$ are the highest and lowest components respectively. For $i = 1, \dots, r$, in view of \eqref{eqn:NormalWtDecomp}, we have a decomposition
$$
    N_{F_i/W} = N_{F_i,+} \oplus N_{F_i,-}
$$
where $N_{F_i,+}$ (resp.\ $N_{F_i,-}$) is the eigenbundle with positive (resp.\ negative) weights. We denote the rank of $N_{F_i,\pm}$ by $r_{F_i,\pm}$. In particular, we have $r_{F_1, +} = r_{F_r, -} = 0$.

The \emph{Bia{\l}ynicki-Birula decomposition} of $W$ with respect to the $\bbC^*$-action is a decomposition into the attractive/repelling sets of the fixed components:
$$
    W_{F_i,\pm} \coloneqq \bigg\{ x \in W \ \bigg|\ \lim_{t \in \bbC^*, t^{\pm 1} \to 0} t \cdot x \in F_i \bigg\}.
$$


\begin{theorem}[Bia{\l}ynicki-Birula decomposition \cite{bialynicki1973some}]\label{BB decomposition} 
We have two decompositions
$$
    W = \bigsqcup_{i=1}^r W_{F_i,+} = \bigsqcup_{i=1}^r W_{F_i,-}
$$
of $W$ into locally closed algebraic subvarieties. Moreover, taking the limit $x \mapsto \lim_{t^{\pm 1} \to 0} t \cdot x$ gives a morphism $W_{F_i,\pm} \to F_i$ of algberaic varieties for each $i$ that is an affine bundle of rank $r_{F_i,\pm}$.
\end{theorem}

In particular, the cells $W_{F_1, -}$ and $W_{F_r, +}$ are Zariski open and dense in $W$.

The fibration $W_{F_i,\pm} \to F_i$ is in general not a vector bundle, as the $\bbC^*$-action on fibers may involve nonlinear algebraic automorphisms. However, if the $\bbC^*$-action on $N_{F_i,+}$ only has weight $1$, the fibration $W_{F_i,+}\to F_i$ may be identified with the vector bundle $N_{F_i,+}$, and similarly for $W_{F_i,-}\to F_i$.

\begin{definition}\label{def:WtPm1}
We say the $\bbC^*$-action on $W$ is a \emph{weight-($\pm 1$) action} if for any $i = 1, \dots, r$, the $\bbC^*$-action on $N_{F_i/W}=N_{F_i,+}\oplus N_{F_i,-}$ only has weights $\pm 1$.
\end{definition}

\subsection{$\bbC^*$-GIT quotients and curve classes}\label{sect:C*CurveClass}
We now consider the wall-and-chamber structure on $E^{\bbC^*}(W)$ for the $\bbC^*$-action on $W$ and characterize the image in $N^1_{\bbC^*}(W)_{\bbR}$. For $L \in E^{\bbC^*}(W)$ with image $\widehat{\omega} \in N^1_{\bbC^*}(W)_{\bbR}$, we write $W^{ss}(\widehat{\omega}) = W^{ss}(L)$ and $W^{s}(\widehat{\omega}) = W^{s}(L)$. Given such an equivariant ample class $\widehat{\omega}$, for any fixed component $F_i$, there exists a constant $\omega_i\in\bbR$ such that $\widehat{\omega}|_{x_i}=\omega_i\lambda$ for any point $x_i\in F_i$. On the other hand,  by the Hilbert-Mumford criterion, a point $x \in W$ belongs to $W^{s}(\widehat{\omega})$ if and only if $\widehat{\omega}|_{x_0}\in \bbR_{>0}\lambda$ and $\widehat{\omega}|_{x_{\infty}}\in \bbR_{<0}\lambda$, where $x_{0} = \lim_{t\to0} t\cdot x$ and $x_{\infty}=\lim_{t\to\infty}t\cdot x$. In other words, we have
$$
    W^{s}(\widehat{\omega}) = \{x \in W \ \big|\ x_0\in F_i, x_{\infty} \in F_j \text{ with } \omega_i>0, \omega_j<0 \}.
$$
Therefore, the wall-and-chamber structure is determined by the linear functions $\omega_i \colon N^1_{\bbC^*}(W)_{\bbR} \to \bbR$, $i = 1, \dots, r$.

There are two special GIT chambers whose corresponding stable loci are
$$
    W_{F_1,-}\backslash F_1 = \{x \in W \ \big|\ x_0 \notin F_1, x_{\infty} \in F_1\}, \quad
    W_{F_r,+}\backslash F_r = \{x \in W \ \big|\ x_0 \in F_r, x_{\infty} \notin F_r\}
$$
respectively. We give the following definition.

\begin{definition}\label{def:highest/lowest GIT}
We refer to the GIT quotient $X = W \gitquot_{L} \bbC^*$ as the \emph{highest} (resp.\ \emph{lowest}) GIT quotient of $W$ if $W^{ss}(L) = W^{s}(L)$ is equal to $W_{F_1,-}\backslash F_1$ (resp.\ $W_{F_r,+}\backslash F_r$).
\end{definition}


For a GIT quotient $X$, let $C_X$ denote the corresponding chamber in $N^1_{\bbC^*}(W)_{\bbR}$. Then we have 
$$
    C_{X}=\begin{cases} 
        \{\widehat{\omega}\in N^1_{\bbC^*}(W)_{\bbR} \text{ ample} \ \big|\ \omega_1<0, \omega_i>0 \text{ for } i \ne 1\} & \text{if $X$ is the highest quotient,}\\ 
        \{\widehat{\omega}\in N^1_{\bbC^*}(W)_{\bbR} \text{ ample} \ \big|\ \omega_r>0, \omega_i<0 \text{ for } i \ne r\} & \text{if $X$ is the lowest quotient.}
    \end{cases}
$$

Moreover, note that if the action is weight-($\pm 1$), by Theorem~\ref{BB decomposition}, the highest (resp.\ lowest) GIT quotient is isomorphic to $\bbP(N_{F_1,-})$ (resp.\ $\bbP(N_{F_r,+})$).

Let $C_{X}^{\vee} \subset N_1^{\bbC^*}(W)_{\bbR}$ denote the closure of the dual cone of $C_X$. We now compute $C_{X}^{\vee}$ for the highest/lowest GIT quotients $X$. For simplicity, we assume that the action is weight-($\pm 1$). We fix a point $x_r \in F_r$ in the lowest component and let $i_{x_r}\colon\pt\to W$ denote the inclusion. As discussed in Section~\ref{sect:EquivCurveClass}, the pushforward 
along the inclusion induces a splitting of both rows of \eqref{N_1-exact sequence}. We define the section class
$$
    \lambda^* \coloneqq \sigma_{F_r}(1)
$$
where $1 \in \bbZ$ corresponds to the cocharacter $\Id \in \Hom(\bbC^*, \bbC^*)$. We introduce the shorthand notation
$$
    S = \hS^{\lambda^*}, \qquad \cS = \hcS^{\lambda^*}.
$$
In addition, for $i = 1, \dots, r$, we define
$$
    a_i \coloneqq \sigma_{F_r}(1) - \sigma_{F_i}(1).
$$
Since $F_r$ is the lowest component, for $i \ne r$, $a_i$ is the class of an effective $\bbC^*$-invariant curve connecting $F_i$ and $F_r$ and lies in $\NE_{\bbN}(W)$.

\begin{proposition}\label{dual ample cone} 
Suppose the $\bbC^*$-action on $W$ is weight-($\pm 1$). For the lowest GIT quotient $X$ of $W$, we have
$$
    C_{X}^{\vee} = \NE(W)+\bbR_{\ge0}\{a_1-\lambda^*, \dots, a_{r-1} - \lambda^*, \lambda^*\}.
$$
For the highest GIT quotient $X$ of $W$, we have
$$
    C_{X}^{\vee} = \NE(W)+\bbR_{\ge0}\{a_1-\lambda^*,\lambda^*-a_2, \dots, \lambda^* - a_r\}. 
$$
\end{proposition}

We thus introduce the following two monoids in $N_1^{\bbC^*}(W)$:
\begin{equation}\label{dual ample cone's monoid}
    C_{X,\bbN}^{\vee}=\begin{cases}
    \NE_{\bbN}(W)+\bbN\{a_1-\lambda^*, \dots, a_{r-1} - \lambda^*, \lambda^*\} & \text{if $X$ is the lowest quotient}, \\
    \NE_{\bbN}(W)+\bbN\{a_1-\lambda^*,\lambda^*-a_2, \dots, \lambda^* - a_r\} & \text{if $X$ is the highest quotient}.
    \end{cases}
\end{equation}

\begin{proof}
We prove the proposition for the lowest GIT quotient $X$; the case of the highest quotient is similar. We denote the projection $N^1_{\bbC^*}(W)_{\bbR}\to N^1(W)_{\bbR}$ by $p$. Note that $\NE(W)\subset C_{X}^{\vee}$ because $p(\widehat{\omega})$ is ample for any $\widehat{\omega}\in C_X$. Moreover, under the splitting of \eqref{N_1-exact sequence} fixed by $i_{x_r}$, we have for any $\widehat{\omega}\in C_X$ that $(\lambda^*, \widehat{\omega}) = \omega_r>0$ and $(\lambda^*-a_i, \widehat{\omega}) = (\sigma_{F_i}(1), \widehat{\omega}) = \omega_i<0$ for $i \ne r$. Thus, we have shown that
$\NE(W)+\bbR_{\ge0}\{a_1-\lambda^*, \dots, a_{r-1} - \lambda^*, \lambda^*\} \subseteq C_{X}^{\vee}$.
    
Conversely, let $u\in C_{X}^{\vee}$, which can be decomposed as $u=u_0+c(u)\lambda^*$ for $u_0\in N_1(W)_{\bbR}$ and $c(u)\in \bbR$  through the fixed splitting of \eqref{N_1-exact sequence}. If $c(u)=0$, we must have $u=u_0\in \NE(W)$. For the other cases, we consider the dual exact sequence of the first row of \eqref{N_1-exact sequence} over $\bbQ$:
\begin{equation}\label{N^1-exact sequence}
    0\to H^2_{\bbC^*}(\pt; \bbQ)\to N^1_{\bbC^*}(W)_{\bbQ}\to N^1(W)_{\bbQ}\to 0.
\end{equation}
The pullback $i_{x_r}^*\colon N^1_{\bbC^*}(W)_{\bbQ} \to H^2_{\bbC^*}(\pt;\bbQ)$ induces a splitting of \eqref{N^1-exact sequence}. Thus, for any ample class $\widehat{\omega} \in N^1_{\bbC^*}(W)_{\bbQ}$, we can compute the pairing as 
$$
    (u, \widehat{\omega}) = c(u)\omega_r + (u_0, p(\widehat{\omega})).
$$

Suppose first that $c(u)>0$. If $u_0\notin \NE(W)$, there exists an ample class $\omega_0\in N^1(W)_{\bbQ}$ with $(u_0, \omega_0) <0$. Using the splitting of \eqref{N^1-exact sequence}, we can choose an ample class $\widehat{\omega} \in C_X$ for which $p(\widehat{\omega})=\omega_0$ and $\omega_r$ is a sufficiently small positive number such that $c(u)\omega_r + (u_0, \omega_0) < 0$. This contradicts the assumption $u\in C_{X}^{\vee}$. Hence, we have $u \in \NE(W)+\bbR_{\ge0}\lambda^*$. 

Now suppose that $c(u)<0$. Consider the simplex 
$$
    \Delta(u) = \bigg\{\sum_{i\ne r}s_ia_i \ \bigg|\  s_i\ge 0, \sum_{i\ne r}s_i=-c(u)\bigg\}\subset \NE(W).
$$
Then $u\in \NE(W)+\bbR_{\ge 0}\{a_i-\lambda^* \ |\  i\ne r\}$ if and only if $u_0\in \NE(W)+\Delta(u)$. Now, if $u_0\notin \NE(W)+\Delta(u)$, by the Hahn-Banach theorem under the condition that $u_0-\Delta(u)$ and $\NE(W)$ are disjoint convex sets in the finite-rank real vector space $N_1(W)_{\bbR}$, there exist an ample class $\omega_0\in N^1(W)_{\bbQ}$ and a constant $c<0$ such that $(u_0-\sum_i s_ia_i, \omega_0) <c$ for any $\sum_i s_ia_i \in \Delta(u)$. For any such $\sum_i s_ia_i$ and any  $\widehat{\omega}  \in N^1_{\bbC^*}(W)_{\bbQ}$ with $p(\widehat{\omega})=\omega_0$, we have
$$
    (u, \widehat{\omega}) = c(u)\omega_r + \bigg( u_0-\sum\limits_{i\ne r} s_ia_i, \omega_0 \bigg) + \sum\limits_{i\ne r}s_i(\omega_r-\omega_i) 
    < c - \sum\limits_{i\ne r}s_i\omega_i.
$$
However, using the splitting of \eqref{N^1-exact sequence}, we can choose $\widehat{\omega} \in C_X$ for which $p(\widehat{\omega})=\omega_0$ and $\max_{i\ne r}\{-\omega_i\}>0$ is sufficiently small such that $c - \sum\limits_{i\ne r}s_i\omega_i < 0$. Again, this contradicts the assumption $u\in C_{X}^{\vee}$. Hence, we have $u\in \NE(W)+\bbR_{\ge 0}\{a_i-\lambda^* \ \big|\ i\ne r\}$. 
\end{proof}

\subsection{3-component $\bbC^*$-VGIT wall-crossings}\label{sect:3Component}
Now, we further specialize the discussion to the case of 3-component $\bbC^*$-VGIT wall-crossings, in the sense of the following.


\begin{assumption}\label{3-component assumption}
Let $W$ be a smooth projective variety with weight-($\pm 1$) $\bbC^*$-action such that the fixed locus only has three connected components 
$$
    W^{\bbC^*} = F_+ \sqcup F_0 \sqcup F_-
$$
where $F_{\pm}$ are the highest and lowest components respectively. Without loss of generality, we assume that $r_{F_0,+}\le r_{F_0,-}$ and denote $c_{F_0}=r_{F_0,+}-r_{F_0,-}$.
\end{assumption}

For instance, in the case of blowups in Example~\ref{ex:Blowup}, the variety $W$ together with the $\bbC^*$-action satisfies Assumption~\ref{3-component assumption}.

Under Assumption~\ref{3-component assumption}, the only Bia{\l}ynicki-Birula cells are
$$
    W_{F_\pm, \pm} = F_\pm, \quad W_{F_+,-} \cong N_{F_+/W}, \quad W_{F_-,+}\cong N_{F_-/W}, \quad W_{F_0,\pm}\cong N_{F_0,\pm}.
$$
In addition, there are only two chambers in $E^{\bbC^*}(W)$ and we may take the ample linearizations $L_\pm$ in Section~\ref{sect:VGIT} in these chambers respectively, giving the wall-crossing between the two GIT quotients $X_\pm$ which are the highest and lowest GIT quotients respectively. Note that all conditions in Definition~\ref{def:SimpleWall} are satisfied and the wall-crossing is simple. In the notation set up there, we have
$$
    W^{ss}(\pm)=W_{F_{\pm},\mp} \setminus F_{\pm} 
    , \quad
    X_\pm \cong \bbP(N_{F_\pm/W}), \quad W^0=W^0\gitquot \bbC^* (0)=F_0,
$$
$$
    W^{\pm}=W^{ss}(0)\backslash W^{ss}(\pm)=W_{F_0,\pm}, \quad W^{\pm}\gitquot \bbC^*\cong \bbP(N_{F_0,\pm}).
$$
We illustrate the geometry of $W$ in Figure~\ref{fig:3-component geometry}.

\begin{figure}[h]
\centering

\tikzset{every picture/.style={line width=0.75pt}} 

\begin{tikzpicture}[x=0.75pt,y=0.75pt,yscale=-0.75,xscale=0.75]

\draw   (121.5,50) -- (351.5,50) -- (390.5,90) -- (160.5,90) -- cycle ;
\draw   (243,69.5) .. controls (243,64.25) and (256.43,60) .. (273,60) .. controls (289.57,60) and (303,64.25) .. (303,69.5) .. controls (303,74.75) and (289.57,79) .. (273,79) .. controls (256.43,79) and (243,74.75) .. (243,69.5) -- cycle ;
\draw  [color={rgb, 255:red, 6; green, 0; blue, 0 }  ,draw opacity=0.98 ][fill={rgb, 255:red, 30; green, 2; blue, 6 }  ,fill opacity=1 ] (275.98,191.07) .. controls (276.25,190.56) and (276.05,189.92) .. (275.54,189.65) .. controls (275.02,189.39) and (274.39,189.59) .. (274.12,190.1) .. controls (273.86,190.62) and (274.06,191.25) .. (274.57,191.52) .. controls (275.08,191.78) and (275.72,191.58) .. (275.98,191.07) -- cycle ;
\draw  [dash pattern={on 4.5pt off 4.5pt}] (91.5,110) -- (420.5,110) -- (459.5,152) -- (130.5,151) -- cycle ;
\draw  [dash pattern={on 4.5pt off 4.5pt}] (231.5,129.5) .. controls (231.5,123.15) and (250.98,118) .. (275,118) .. controls (299.02,118) and (318.5,123.15) .. (318.5,129.5) .. controls (318.5,135.85) and (299.02,141) .. (275,141) .. controls (250.98,141) and (231.5,135.85) .. (231.5,129.5) -- cycle ;
\draw  [dash pattern={on 4.5pt off 4.5pt}] (81.5,217) -- (440.5,218) -- (491.5,269) -- (132.5,269) -- cycle ;
\draw  [dash pattern={on 4.5pt off 4.5pt}] (219,244) .. controls (219,236.27) and (245.3,230) .. (277.75,230) .. controls (310.2,230) and (336.5,236.27) .. (336.5,244) .. controls (336.5,251.73) and (310.2,258) .. (277.75,258) .. controls (245.3,258) and (219,251.73) .. (219,244) -- cycle ;
\draw    (101,300) -- (154.5,355) ;
\draw    (154.5,355) -- (484.5,355) ;
\draw    (429.5,300) -- (484.5,355) ;
\draw    (101,300) -- (224.5,300) ;
\draw    (331.5,300) -- (429.5,300) ;
\draw  [dash pattern={on 4.5pt off 4.5pt}]  (224.5,300) -- (325.5,300) ;
\draw  [draw opacity=0][dash pattern={on 4.5pt off 4.5pt}] (239.79,330.6) .. controls (244.35,324.69) and (259.66,320.53) .. (277.91,320.82) .. controls (298.3,321.15) and (315.19,326.9) .. (317.58,334) -- (278.14,335.01) -- cycle ; \draw  [dash pattern={on 4.5pt off 4.5pt}] (239.79,330.6) .. controls (244.35,324.69) and (259.66,320.53) .. (277.91,320.82) .. controls (298.3,321.15) and (315.19,326.9) .. (317.58,334) ;  
\draw   (219,244) .. controls (219.5,220) and (235.5,192) .. (275.5,192) .. controls (315.5,192) and (336.5,219) .. (336.5,244) .. controls (336.5,269) and (327.5,325) .. (317.58,334) .. controls (307.66,342.99) and (289.5,344) .. (278.5,344) .. controls (267.5,344) and (252.29,341.81) .. (239.5,331) .. controls (226.71,320.19) and (218.5,268) .. (219,244) -- cycle ;
\draw    (231.5,129.5) .. controls (235.5,161) and (236.5,189) .. (275.5,192) ;
\draw    (275.5,192) .. controls (311.5,192) and (316.5,163) .. (318.5,129.5) ;
\draw    (235.5,90) .. controls (233.5,98) and (231.5,120) .. (231.5,129.5) ;
\draw  [dash pattern={on 4.5pt off 4.5pt}]  (235.5,90) .. controls (238.5,77) and (242.5,68) .. (243,69.5) ;
\draw    (318.5,129.5) .. controls (318.5,111) and (316.5,102) .. (312.5,90) ;
\draw  [dash pattern={on 4.5pt off 4.5pt}]  (303,69.5) .. controls (306.5,75) and (309.5,82) .. (312.5,90) ;
\draw  [dash pattern={on 0.84pt off 2.51pt}]  (296,134) -- (364,174) ;
\draw  [dash pattern={on 0.84pt off 2.51pt}]  (264,246) -- (188,193) ;
\draw  [dash pattern={on 0.84pt off 2.51pt}]  (260,98) -- (66,98) ;
\draw [shift={(64,98)}, rotate = 360] [color={rgb, 255:red, 0; green, 0; blue, 0 }  ][line width=0.75]    (10.93,-3.29) .. controls (6.95,-1.4) and (3.31,-0.3) .. (0,0) .. controls (3.31,0.3) and (6.95,1.4) .. (10.93,3.29)   ;
\draw  [dash pattern={on 0.84pt off 2.51pt}]  (258,284) -- (73,284.99) ;
\draw [shift={(71,285)}, rotate = 359.69] [color={rgb, 255:red, 0; green, 0; blue, 0 }  ][line width=0.75]    (10.93,-3.29) .. controls (6.95,-1.4) and (3.31,-0.3) .. (0,0) .. controls (3.31,0.3) and (6.95,1.4) .. (10.93,3.29)   ;
\draw [draw opacity=1]   (271.59,79.45) .. controls (260.59,89.45) and (259.59,126.45) .. (260.59,140.45) .. controls (261.59,154.45) and (262.59,182.45) .. (274.12,190.1) ;
\draw [draw opacity=1]   (278.5,344) .. controls (261,352) and (260,243) .. (261,258) .. controls (262,273) and (256,218) .. (274.57,191.52) ;
\draw    (661,347) -- (658.02,68) ;
\draw [shift={(658,66)}, rotate = 89.39] [color={rgb, 255:red, 0; green, 0; blue, 0 }  ][line width=0.75]    (10.93,-3.29) .. controls (6.95,-1.4) and (3.31,-0.3) .. (0,0) .. controls (3.31,0.3) and (6.95,1.4) .. (10.93,3.29)   ;

\draw (375,52) node [anchor=north west][inner sep=0.75pt]   [align=left] {$F_+$};
\draw (459,120) node [anchor=north west][inner sep=0.75pt]   [align=left] {$X_+=\bbP(N_{F_+/W})$};
\draw (481,235) node [anchor=north west][inner sep=0.75pt]   [align=left] {$X_-=\bbP(N_{F_-/W})$};
\draw (470,316) node [anchor=north west][inner sep=0.75pt]   [align=left] {$F_-$};
\draw (268,169) node [anchor=north west][inner sep=0.75pt]  [rotate=-0.02] [align=left] {$F_0$};
\draw (366,177) node [anchor=north west][inner sep=0.75pt]   [align=left] {$\bbP(N_{F_0,+})$};
\draw (124,174) node [anchor=north west][inner sep=0.75pt]   [align=left] {$\bbP(N_{F_0,-})$};
\draw (46,87) node [anchor=north west][inner sep=0.75pt]   [align=left] {$b$};
\draw (52,275) node [anchor=north west][inner sep=0.75pt]   [align=left] {$a$};
\draw (122,149) node [anchor=north west][inner sep=0.75pt]   [align=left] {\textsuperscript{}};
\draw (672,176) node [anchor=north west][inner sep=0.75pt]   [align=left] {$\mu$};

\end{tikzpicture}

\caption{Geometry of the $3$-component $\bbC^*$-VGIT wall-crossing}
\label{fig:3-component geometry}
\end{figure}
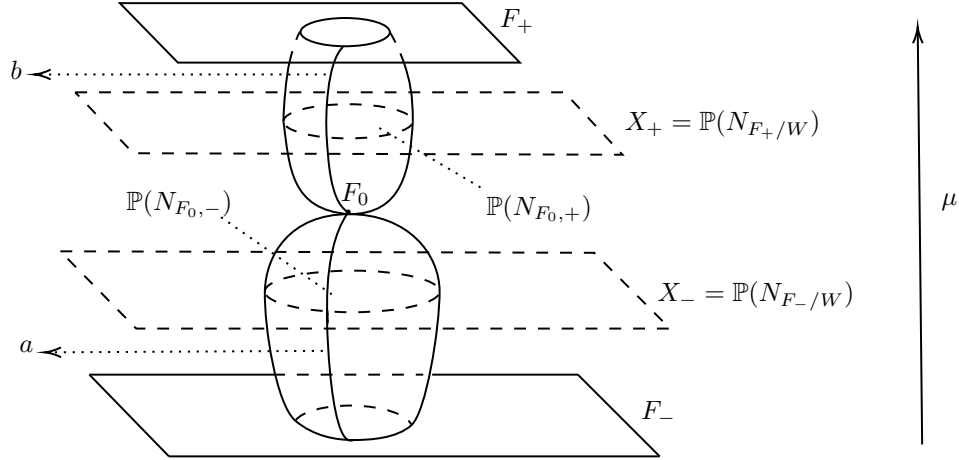

We denote the classes of the $\bbC^*$-invariant curves connecting $F_0$ and $F_{\mp}$ by $a$ and $b$ respectively. Then
$$
    a = \sigma_{F_-}(1) - \sigma_{F_0}(1), \quad b = \sigma_{F_0}(1) - \sigma_{F_+}(1), \quad \lambda^* = \sigma_{F_-}(1).
$$
Because $a+b-\lambda^*\in \NE_{\bbN}(W)+\bbN\{a-\lambda^*\}$ and $\lambda^*\in\NE_{\bbN}(W)+\bbN\{\lambda^*-a\}$, by Proposition~\ref{dual ample cone} and \eqref{dual ample cone's monoid}, the dual ample cones and the induced monoids are given by:
\begin{equation}\label{eqn:DualCone3Comp}
    \begin{aligned}
    & C_{X_-}^{\vee}=\NE(W)+\bbR_{\ge 0}\{a-\lambda^*,\lambda^*\}, && C_{X_-,\bbN}^{\vee}=\NE_{\bbN}(W)+\bbN\{a-\lambda^*,\lambda^*\},\\
    & C_{X_+}^{\vee}=\NE(W)+\bbR_{\ge 0}\{a+b-\lambda^*,\lambda^*-a\}, && C_{X_+,\bbN}^{\vee}=\NE_{\bbN}(W)+\bbN\{a+b-\lambda^*,\lambda^*-a\}.
    \end{aligned}
\end{equation}
See Figure~\ref{fig:dualamplecone}. We denote the Novikov variable for $\lambda^* - a = \sigma_{F_0}(1)$ by
$$
    Q_W^{-a}S = \widehat{S}^{\sigma_{F_0}(1)} =\colon S_{F_0}.
$$
We also introduce the following monoid: 
$$
    \NE_{\bbN}^{\bbC^*}(W) \coloneqq C_{X_-,\bbN}^{\vee}\cap C_{X_+,\bbN}^{\vee}=\text{NE}_{\bbN}(W)+\bbN\{a+b-\lambda^*,\lambda^*\}.
$$

\begin{figure}[h]
\centering
\begin{tikzpicture}[scale=1.2]

\draw[->] (0,0) -- (2,0) node[below] {$\lambda^* - a$};
\draw[->] (0,0) -- (0,2) node[above] {$\NE(W)$};
\draw[->] (0,0) -- (-2, 0) node[below] {$a-\lambda^*$};
\draw[->] (0,0) -- (2,1.2) node[above] {$\lambda^*$};
\draw[->] (0,0) --(-2, 1.4) node[above] {$a+b-\lambda^*$}; 
\draw[thick, dotted, ->] (-0.9, 0.63) arc(145:0:1.099);
\node at (0.85, 1.2){$X_+$};
\draw[thick, dotted, ->] (0.8574, 0.5144) arc(31:180:1);
\node at (-0.4, 0.6){$X_-$};

\end{tikzpicture}
\caption{The dual ample cones $C_{X_{\pm}}^{\vee}$ in $N_1^{\bbC^*}(W)_{\bbR}$}
\label{fig:dualamplecone}
\end{figure}
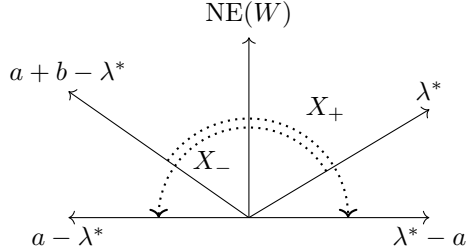

In the 3-component case, the defining equation for the shift operator $\cS = \hcS^{\lambda^*}$ in Definition~\ref{formula for shift operator on Givental cone} specializes as follows. Given any $\bff \in \cH_{W}^{\rat}$, we have
\begin{equation}\label{eqn:ShiftOpLocal}
\begin{aligned}
    &\cS (\bff)|_{F_{+}}=Q_W^{a+b}\frac{1}{e_{-\lambda+z}(N_{F_+/W})}e^{-z\partial_{\lambda}} \bff|_{F_+},
    &&\cS^{-1} (\bff)|_{F_{+}}=Q_W^{-a-b}e_{-\lambda}(N_{F_+/W})e^{z\partial_{\lambda}} \bff|_{F_+},\\ 
    &\cS (\bff)|_{F_0}=Q_W^a\frac{e_{\lambda}(N_{F_0,+})}{e_{-\lambda+z}(N_{F_0,-})}e^{-z\partial_{\lambda}} \bff|_{F_0},
    &&\cS^{-1} (\bff)|_{F_0}=Q_W^{-a}\frac{e_{-\lambda}(N_{F_0,-})}{e_{\lambda+z}(N_{F_0,+})}e^{z\partial_{\lambda}} \bff|_{F_0},\\ 
    &\cS (\bff)|_{F_-} = e_{\lambda}(N_{F_-/W})e^{-z\partial_{\lambda}} \bff|_{F_-},
    &&\cS^{-1} (\bff)|_{F_-}=\frac{1}{e_{\lambda+z}(N_{F_-/W})}e^{z\partial_{\lambda}} \bff|_{F_-}.
\end{aligned}
\end{equation}

Furthermore, the cohomology decomposition in  Corollary~\ref{cor:deRhamIsoVGIT} specializes as follows.

\begin{proposition}\label{prop:deRhamIso3Comp}
Under Assumption~\ref{3-component assumption}, we have an isomorphism of cohomology groups
$$
    \Big(\phi_2,\bigoplus\limits_{i=0}^{\abs{c_{F_0}}-1}j_{-,*}\circ (\cup h_-^{i})\circ \psi_-^*\Big)\colon H^*(X_+)\oplus\bigoplus\limits_{i=0}^{\abs{c_{F_0}}-1}H^{*-2i-2r_{F_0,+}}(F_0)\longrightarrow H^*(X_-)
$$
where $\phi_2$ is the injective map induced by the graph closure of $X_- \dashrightarrow X_+$, $j_{\pm}$ is the inclusion $\bbP(N_{F_0,\pm})\hookrightarrow X_{\pm}$, and $h_{\pm}$ is the fiber class of $\bbP(N_{F_0,\pm})$.
\end{proposition}

\subsection{Interactions between Kirwan maps and cohomology decomposition}\label{subsec:interaction between K and C} 
In this subsection, we study the Kirwan maps $\kappa_{X_{\pm}}$ and their interactions with the cohomology decomposition (Proposition~\ref{prop:deRhamIso3Comp}) in the situation of 3-component $\bbC^*$-VGIT wall-crossings (Assumption~\ref{3-component assumption}).

Recall that in this case, each Bia{\l}ynicki-Birula cell is identified with a vector bundle over the corresponding fixed component. Applying the Thom-Gysin sequence to the embeddings 
$$
    W_{F_0,+} \longrightarrow W_{F_0,+} \cup W_{F_-, +} = W \setminus F_+, \quad F_+ = W_{F_+,+} \longrightarrow  W_{F_+,+} \cup W_{F_0,+} \cup W_{F_-, +} = W, 
$$
of the positive cells, we obtain the following two long exact sequences:
$$
\begin{tikzcd}
    \cdots \arrow[r,"\partial"] & H^{*-2r_{F_0}^-}_{\bbC^*}(W_{F_0,+})  \arrow[r,"i_{F_0,*}^+"] & H^*_{\bbC^*}(W\backslash F_+)  \arrow[r,"i_{W_{F_-,+}^*}"] & H^*_{\bbC^*}(W_{F_-,+})  \arrow[r,"\partial"] & \cdots
\end{tikzcd}
$$
and
$$
\begin{tikzcd}
    \cdots \arrow[r,"\partial"] & H^{*-2r_{F_+}^-}_{\bbC^*}(F_+)  \arrow[r,"i_{F_+,*}^+"] & H^*_{\bbC^*}(W)  \arrow[r,"i_{W\backslash F_+}^*"] & H^*_{\bbC^*}(W\backslash F_+)  \arrow[r,"\partial"] & \cdots
\end{tikzcd}
$$
Moreover, since the equivariant Euler classes of $N_{W_{F_0,+}/W\backslash F_+}$ and $N_{F_+/W}$ are not zero divisors in $H^*_{\bbC^*}(W_{F_0,+})\cong H_{\bbC^*}^*(F_0)$ and $H_{\bbC^*}^*(F_+)$ respectively, the maps $i_{F_0,*}^+$ and $i_{F_+,*}^+$ above are injective (see \cite{kirwan1984cohomology}). Therefore, the long exact sequences split into the following short exact sequences:
\begin{equation}
\begin{tikzcd}\label{3-component Gysin 3}
    0 \arrow[r] & H^{*-2r_{F_0}^-}_{\bbC^*}(F_0)  \arrow[r,"i_{F_0,*}^+"] & H^*_{\bbC^*}(W\backslash F_+)  \arrow[r,"i_{F_-}^*"] & H^*_{\bbC^*}(F_-)  \arrow[r] & 0,
\end{tikzcd}
\end{equation}
\begin{equation}
\begin{tikzcd}\label{3-component Gysin 4}
    0 \arrow[r] & H^{*-2r_{F_+}^-}_{\bbC^*}(F_+)  \arrow[r,"i_{F_+,*}"] & H^*_{\bbC^*}(W)  \arrow[r,"i_{W\backslash F_+}^*"] & H^*_{\bbC^*}(W\backslash F_+)  \arrow[r] & 0.
\end{tikzcd}
\end{equation}
We also give the counterparts for the embeddings of the negative Bia{\l}ynicki-Birula cells below:
\begin{equation}\label{3-component Gysin 5}
\begin{tikzcd}
    0 \arrow[r] & H^{*-2r_{F_0}^+}_{\bbC^*}(F_0)  \arrow[r,"i_{F_0,*}^-"] & H^*_{\bbC^*}(W\backslash F_-)  \arrow[r,"i_{F_+}^*"] & H^*_{\bbC^*}(F_+)  \arrow[r] & 0,
\end{tikzcd}
\end{equation}
\begin{equation}\label{3-component Gysin 6}
\begin{tikzcd}
    0 \arrow[r] & H^{*-2r_{F_-}^+}_{\bbC^*}(F_-)  \arrow[r,"i_{F_-,*}"] & H^*_{\bbC^*}(W)  \arrow[r,"i_{W\backslash F_-}^*"] & H^*_{\bbC^*}(W\backslash F_-)  \arrow[r] & 0.
\end{tikzcd}
\end{equation}

Recall that the Kirwan map $\kappa_{X_{\pm}}$ is the pullback $H^*_{\bbC^*}(W)\to H_{\bbC^*}^*(W_{F_{\pm},\mp}\backslash F_{\pm})$, which factors through the composition 
$$
    H^*_{\bbC^*}(W) \longrightarrow H_{\bbC^*}^*(W_{F_{\pm},\mp}) \longrightarrow H_{\bbC^*}^*(W_{F_{\pm},\mp}\backslash F_{\pm}).
$$
This composition is the same as 
\begin{equation}\label{eqn:KirwanFactorize}
    H_{\bbC^*}^*(W) \longrightarrow H_{\bbC^*}^*(F_{\pm})\longrightarrow H_{\bbC^*}^*(F_{\pm})/e_{\bbC^*}(N_{F_{\pm}/W}).
\end{equation}
With this, we compute the kernel of the Kirwan map as follows.

\begin{lemma}\label{lem:KirwanKernel}
We have $\Ker(\kappa_{X_{\pm}}) = \big(i_{W\backslash F_{\mp}}^*\big)^{-1} \Im(i_{F_0,*}^{\mp}) + \Im(i_{F_{\pm},*})$.
\end{lemma}

\begin{proof}
We prove the statement for $X_-$; the case of $X_+$ is similar. We first prove the inclusion ($\supseteq$). For any $\alpha\in (i_{W\backslash F_{+}}^*)^{-1} \Im(i_{F_0,*}^{+})$, using \eqref{3-component Gysin 3}, we may compute that $\alpha|_{F_-}=(\alpha|_{W\backslash F_+})|_{F_-}=0$, which implies that $\alpha \in \Ker(\kappa_{X_{-}})$. In addition, for any $\alpha \in H^*_{\bbC^*}(F_-)$, since $i_{F_-,*}(\alpha) = e_{\bbC^*}(N_{F_-/W})\alpha$, we have $\kappa_{X_-}(i_{F_-,*}(\alpha)) = 0$.


We now prove the opposite direction ($\subseteq$). For any $\alpha\in \Ker(\kappa_{X_-})$, by \eqref{eqn:KirwanFactorize}, there exists  $\beta\in H_{\bbC^*}^*(F_-)$ such that $\alpha|_{F_-} = e_{\bbC^*}(N_{F_-/W})\beta$. This means that $(\alpha-i_{F_-,*}(\beta))|_{F_-}=0$. We may then conclude by \eqref{3-component Gysin 3} and \eqref{3-component Gysin 4}. 
\end{proof}
 
In the next two lemmas, we describe the relation between the two Kirwan maps $\kappa_{X_{\pm}}$ and the cohomology decomposition (Proposition~\ref{prop:deRhamIso3Comp}). Let
$$
    U = X_+\backslash \bbP(N_{F_0,+}) \cong X_-\backslash\bbP(N_{F_0,-})
$$
denote the isomorphic complements of the exceptional loci and let $i_{U_{\pm}}\colon U \hookrightarrow X_{\pm}$ denote the inclusions.
Since $\phi_2$ is induced by the correspondence, we have the commutative relation $i_{U_+}^*=i_{U_-}^*\circ \phi_2\colon H^*(X_+)\to H^*(U)$.

\begin{lemma}\label{lem:KirwanIdentifyCoh}
For $\alpha_+ \in H^*(X_+)$ and $\alpha_- \in H^*(X_-)$, $i_{U_+}^*(\alpha_+) = i_{U_-}^*(\alpha_-)$ if and only if there exists $\alpha \in H_{\bbC^*}^*(W)$ such that $\kappa_{X_{\pm}}(\alpha)=\alpha_{\pm}$.
\end{lemma}

\begin{proof}
First, the ``if'' direction follows directly from the observation that $i_{U_+}^*\circ \kappa_{X_+}$ and $i_{U_-}^*\circ \kappa_{X_-}$ both coincide with the pullback $H_{\bbC^*}^*(W)\to H_{\bbC^*}^*(W_{F_+,-}\cap W_{F_-,+}) \cong H^*(U)$. Conversely, suppose that $i_{U_+}^*(\alpha_+) = i_{U_-}^*(\alpha_-)$. By the surjectivity of $\kappa_{X_+}$, there exists $\alpha_0\in H_{\bbC^*}^*(W)$ such that $\kappa_{X_+}(\alpha_0)=\alpha_+$, which implies that $i_{U_-}^*(\kappa_{X_-}(\alpha_0)) = i_{U_+}^*(\kappa_{X_+}(\alpha_0)) = i_{U_-}^*(\alpha_-)$. Therefore, $\kappa_{X_-}(\alpha_0)-\alpha_-\in \Ker(i_{U_-}^*)=\Im(j_{-,*})$, where recall that $j_-$ denotes the inclusion $\bbP(N_{F_0,-})\hookrightarrow X_-$. Let $\beta\in H^*(\bbP(N_{F_0,-}))\cong H_{\bbC^*}^*(W_{F_0,-}\backslash F_0)$ such that $j_{-,*}(\beta) = \kappa_{X_-}(\alpha_0)-\alpha_-$. Consider the commutative diagram
\begin{equation}\label{Cohomology excess formula for iF0-}
    \begin{tikzcd}
       H^{*-2r_{F_0,+}}_{\bbC^*}(F_0)\cong H^*_{\bbC^*}(W\backslash F_-, W_{F_+,-}) \arrow[r,"i_{F_0,*}^{+}"]\arrow[d, "g^*"]& H^*_{\bbC^*}(W\backslash F_-)\arrow[d, "i_-^*"] \\H^{*-2r_{F_0,+}}(\bbP(N_{F_0,-}))\cong H^*_{\bbC^*}(W^{ss}(-), W^{ss}(-)\cap W^{ss}(+))\arrow[r,"j_{-,*}"]& H^*_{\bbC^*}(W^{ss}(-))\cong H^*(X_-)
    \end{tikzcd}
\end{equation}
induced by the inclusion of relative pairs $g\colon (W^{ss}(-), W^{ss}(-)\cap W^{ss}(+))\to (W\backslash F_-, W_{F_+,-})$. Here, $g^*$ can be identified with the projection $H_{\bbC^*}^*(F_0)\to H^*(\bbP(N_{F_0,-}))$ and is surjective. We take $\gamma \in H_{\bbC^*}^*(F_0)$ such that $g^*(\gamma) = \beta$. Moreover, by the surjectivity of the map $i_{W\backslash F_-}^*$ in \eqref{3-component Gysin 6}, we can find $\delta \in H_{\bbC^*}^*(W)$ such that $i_{W\backslash F_-}^*(\delta) = i^+_{F_0,*}(\gamma)$. Lemma~\ref{lem:KirwanKernel} then implies that $\kappa_{X_+}(\delta) = 0$. Therefore, the element $\alpha = \alpha_0-\delta$ satisfies that
$$
    \kappa_{X_+}(\alpha) = \kappa_{X_+}(\alpha_0) - \kappa_{X_+}(\delta) = \alpha_+,
$$
$$
    \kappa_{X_-}(\alpha) = \kappa_{X_-}(\alpha_0) - \kappa_{X_-}(\delta) = \alpha_- + j_{-,*}(\beta) - i_-^*(i_{F_0,*}^+(\gamma)) = \alpha_- + j_{-,*}(\beta) - j_{-,*}(\beta) = \alpha_-.
$$
\end{proof}

\begin{lemma}\label{lem:KirwanPhi2}
    For any $\alpha_+\in H^*(X_+)$, we can find a unique $\alpha\in H_{\bbC^*}^*(W)$ such that $\kappa_{X_+}(\alpha)=\alpha_+$, $\kappa_{X_-}(\alpha)=\phi_2(\alpha_+)$, and $\deg_{\lambda}(\alpha|_{F_0})\le r_{F_0,+}-1$ where we view $\alpha|_{F_0}\in H^*_{\bbC^*}(F_0)\cong H^*(F_0) [\lambda]$ as a polynomial in $\lambda$. 
\end{lemma}

\begin{proof}
By Lemma~\ref{lem:KirwanIdentifyCoh}, there always exists $\alpha\in H_{\bbC^*}^*(W)$ that satisfies $\kappa_{X_+}(\alpha)=\alpha_+$ and $\kappa_{X_-}(\alpha)=\phi_2(\alpha_+)$. We show that the choice of $\alpha$ is unique under the degree constraint on $\alpha |_{F_0}$. Notice that for $\alpha_1,\alpha_2\in H_{\bbC^*}^*(W)$ that satisfy $\kappa_{X_+}(\alpha_1)=\kappa_{X_+}(\alpha_2)=\alpha_+$ and $\kappa_{X_-}(\alpha_1)=\kappa_{X_-}(\alpha_2)=\phi_2(\alpha_+)$, by Lemma~\ref{lem:KirwanKernel}, there exist $\beta_+,\beta_-\in H_{\bbC^*}^*(F_0)$ such that $(\alpha_1-\alpha_2)|_{W\backslash F_{\pm}}=i_{F_0,*}^{\pm}(\beta_{\pm})$. Therefore, $(\alpha_1-\alpha_2)|_{F_0}$ is divisible by both $e_{\lambda}(N_{F_0,+})$ and $e_{-\lambda}(N_{F_0,-})$, and thus by the least common multiple 
$\mathrm{lcm}(e_{\lambda}(N_{F_0,+}),e_{-\lambda}(N_{F_0,-}))$. This is a monic polynomial in $\lambda$ whose degree $d$ is at least $r_{F_0,-}-1$. Therefore, there exists a unique $\alpha\in H_{\bbC^*}^*(W)$ such that $\kappa_{X_+}(\alpha)=\alpha_+$, $\kappa_{X_-}(\alpha)=\phi_2(\alpha_+)$, and $\deg_{\lambda}(\alpha|_{F_0})\le d-1$. 

Our computation below will use that $X_-$ and $X_+$ share the common blowup $\tX = X_+\times_{X_0}X_-$ with common exceptional divisor $E = \bbP(N_{F_0,+})\times_{F_0} \bbP(N_{F_0,-})$. For convenience, we briefly review the cohomology decomposition of a general smooth blowup $\pi \colon \tX=\Bl_{Z}X \to X$ with exceptional divisor $E\cong \bbP(N_{Z/X})$. We have a commutative diagram
\begin{equation}\label{general commutative diagram for blow up}
\begin{tikzcd}
       H^{*}(\tX)\arrow[d,"i_{E}^*"]\arrow[r, "\pi_*"]& H^*(X)\arrow[d, "i_Z^*"] \\H^{*}(E)\arrow[r,"\varphi"]& H^*(Z).
\end{tikzcd}
\end{equation}
Under
$$
    H^*(E)\cong H^*(Z)h^0\oplus H^{*-2}(Z)h^1\oplus\cdots\oplus H^{*-2r+2}(Z)h^{r-1}
$$ 
where $r$ is the codimension of $Z$ in $X$ and $h=c_1(\cO_E(1))=-c_1(N_{E/\tX})$, if we write an element of $f \in H^*(E)$ as a polynomial in $h$ with coefficients in $H^{*-2}(Z)$, the map $\varphi$ in \eqref{general commutative diagram for blow up} is identified with the projection to the first summand, i.e. taking the constant term of $f(h)$. The commutativity of \eqref{general commutative diagram for blow up} can be checked by expanding any element $\tilde{\alpha} \in H^*(\tX)$ as $\pi^*(\delta) + i_{E,*}(g(h))$ for some $\delta \in H^*(X)$ and $g(h) \in H^*(E)$ (see \cite[Theorem 7.31]{voisin2002hodge}). 


Now we return to the proof of the lemma.
The map $\phi_2$ in Proposition~\ref{prop:deRhamIso3Comp} is induced by the correspondence and is given by $\pi_{-,*}\circ \pi_+^*\colon H^*(X_+)\to H^*(X_-)$, where $\pi_{\pm}\colon X_+\times_{X_0}X_- \cong \Bl_{\bbP(N_{F_0,\pm})}X_{\pm}\to X_{\pm}$ are the blowups. 
We have a commutative diagram
$$
\begin{tikzcd}
       H^{*}(X_+)\arrow[r,"\pi_+^*"]\arrow[d, "j_+^*"]& H^*(X_+\times_{X_0}X_-)\arrow[r,"\pi_{-,*}"]\arrow[d, "i_E^*"] &H^*(X_-)\arrow[d,"j_-^*"]\\H^{*}(\bbP(N_{F_0,+}))\arrow[r,"p_+^*"]& H^*(\bbP(N_{F_0,+})\times_{F_0} \bbP(N_{F_0,-}))\arrow[r,"\varphi"]& H^*(\bbP(N_{F_0,-}))
\end{tikzcd}  
$$
where the square on the right is \eqref{general commutative diagram for blow up} applied to $\pi_-$.
Let $h_{\pm}$ denote the fiber class of $\bbP(N_{F_0,\pm})$ and, by an abuse of notation, also its pullback to $\bbP(N_{F_0,+})\times_{F_0} \bbP(N_{F_0,-})$.  By Proposition~\ref{prop:VGITGeometry}, the normal bundle of $\bbP(N_{F_0,+})\times_{F_0} \bbP(N_{F_0,-})$ in $X_+\times_{X_0}X_-$ is isomorphic to $\cO(-1,-1)$, and $c_1(\cO(-1,-1)) = -h_+ - h_-$.
Then for any $\Tilde{f}=\Tilde{f}(h_+,h_-)\in H^*(\bbP(N_{F_0,+})\times_{F_0} \bbP(N_{F_0,-}))$, if we view $\Tilde{f}$ as a polynomial in $h_+ + h_-$ with coefficients in $H^*(\bbP(N_{F_0,-}))$, it follows from the discussion on \eqref{general commutative diagram for blow up} above that $\varphi(\Tilde{f})$ is the projection to the constant term with respect to $h_+ + h_-$.
Now, for $\alpha_+\in H^*(X_+)$, if we write $\alpha_+|_{\bbP(N_{F_0,+})}=f(h_+)$ where $f$ is a polynomial in $h_+$ with degree at most $r_{F_0,+}-1$, we compute that
$$
    \phi_2(\alpha_+)|_{\bbP(N_{F_0,-})} = \pi_{-,*}\circ \pi_+^*(\alpha_+)|_{\bbP(N_{F_0,-})} = \varphi\circ p_+^*(\alpha_+|_{\bbP(N_{F_0,+})}) = \varphi(f(h_+)).
$$
This is equal to the constant term of $f(h_+)$ with respect to $h_+ + h_-$ and can be computed to be $f(-h_-)$. Therefore, we have 
$$
    \kappa_{X_+}(\alpha)|_{\bbP(N_{F_0,+})} = \alpha_+|_{\bbP(N_{F_0,+})} =f(h_+), \quad \kappa_{X_-}(\alpha)|_{\bbP(N_{F_0,-})} = \phi_2(\alpha_+)|_{\bbP(N_{F_0,-})} =f(-h_-).
$$
It then follows from \eqref{Cohomology excess formula for iF0-} that $(\alpha|_{F_0})_{\lambda\mapsto \pm h_{\pm}}=f(\pm h_{\pm})$, which implies that 
$$
    \alpha|_{F_0}\equiv f(\lambda) \text{ mod }e_{\pm\lambda}(N_{F_0,\pm}).
$$
Finally, because $\deg_{\lambda}(\alpha|_{F_0})\le d-1$, $\alpha|_{F_0}$ must be equal to $f(\lambda)$, which is a polynomial in $\lambda$ of degree at most $r_{F_0,+}-1$.
\end{proof}

We conclude the subsection with a computation of the interaction between the cohomology decomposition (Proposition~\ref{prop:deRhamIso3Comp}) and the Poincar\'e pairings.

\begin{lemma}\label{lem:DecompClassicalPairing}
In the context of Proposition~\ref{prop:deRhamIso3Comp}, for any $\alpha_1,\alpha_2\in H^*(X_+)$,  $\beta_1,\beta_2\in H^*(F_0)$, and $0\le l_1,l_2\le\abs{c_{F_0}}-1$, we have the following identities:
\begin{align}
    &\int_{X_-}\phi_2(\alpha_1)\cup\phi_2(\alpha_2) = \int_{X_+}\alpha_1\cup\alpha_2,\label{poincare pairing phi2 cup phi2} \\
    &\int_{X_-}\phi_2(\alpha_1)\cup j_{-,*}(h_-^{l_1} \psi_-^*(\beta_1)) = 0, \label{Poincare pairing phi2 cup proj}
    \\&
    \int_{X_-}j_{-,*}(h_-^{l_1} \psi_-^*(\beta_1))\cup j_{-,*}(h_-^{l_2} \psi_-^*(\beta_2))= \begin{cases}
        0, & l_1+l_2<\abs{c_{F_0}}-1,\\
        (-1)^{r_{F_0,+}}\int_{F_0}\beta_1\cup\beta_2, & l_1+l_2=\abs{c_{F_0}}-1.
    \end{cases}\label{Poincare pairing proj cup proj}
\end{align}
\end{lemma}

\begin{proof}
The last statement \eqref{Poincare pairing proj cup proj} follows from the computation 
\begin{align*}
    \int_{X_-}j_{-,*}(h_-^{l_1} \psi_-^*(\beta_1)) \cup j_{-,*}(h_-^{l_2} \psi_-^*(\beta_2)) &= \int_{\bbP(N_{F_0,-})}j_{-}^*j_{-,*}(h_-^{l_1+l_2} \psi_-^*(\beta_1 \cup \beta_2))\\
    &= e(\psi_-^*(N_{F_0,+})(-1))h_-^{l_1+l_2}\psi_-^*(\beta_1\cup\beta_2),
\end{align*} 
where $e(\psi_-^*(N_{F_0,+})(-1))$ is a polynomial in $h_-$ with coefficients in $H^*(F_0)$ whose leading term is $(-h_-)^{r_{F_0,+}}$. For the statement \eqref{Poincare pairing phi2 cup proj}, recall from the proof of Lemma~\ref{lem:KirwanPhi2} that if $\alpha_1|_{\bbP(N_{F_0,+})}=f(h_+)$, where $f$ is a polynomial of degree at most $r_{F_0,+}-1$, it holds that $\phi_2(\alpha_1)|_{\bbP(N_{F_0,-})}=f(-h_-)$. Therefore, we have
\begin{align*}
    \int_{X_-}\phi_2(\alpha_1) \cup j_{-,*}(h_-^{l_1} \psi_-^*(\beta_1)) &=\int_{\bbP(N_{F_0,+})} \phi_2(\alpha_1)|_{\bbP(N_{F_0,-})}\cup h_-^{l_1} \psi_-^*(\beta_1)\\
    &= \int_{\bbP(N_{F_0,-})}f(-h_-)h_-^{l_1}\psi_-^*(\beta_1).
\end{align*}
Notice that $\deg_{h_-} (f(-h_-)h_-^{l_1}) \le \abs{c_{F_0}}-1+r_{F_0,+}-1= r_{F_0,-}-2$. This proves \eqref{Poincare pairing phi2 cup proj}. 

Finally, we show \eqref{poincare pairing phi2 cup phi2}. This statement appears in \cite[Section 2.3]{lee2010flops} in terms of Chow motives. Here, we give a proof by explicit calculations. We still first consider the general blowup model $\pi \colon \tX=\Bl_{Z}X\to X$ with exceptional divisor $E\cong \bbP(N_{Z/X})$ as in the proof of Lemma~\ref{lem:KirwanPhi2}. For any $\tilde{\alpha}_1, \tilde{\alpha}_2\in H^*(\tX)$ expanded as $\pi^*(\delta_1)+i_{E,*}(g_1(h))$ and $\pi^*(\delta_2)+i_{E,*}(g_2(h))$ respectively, where $\delta_1, \delta_2 \in H^*(X)$ and $g_1, g_2$ are polynomials in $h$ with degrees at most $r-2$, we notice that
$$   
    \int_{\tX}\tilde{\alpha}_1\cup \tilde{\alpha}_2 = \int_X \pi_*(\tilde{\alpha}_1) \cup \pi_*(\tilde{\alpha}_2)
$$
if and only if 
$$
    \int_{E} h\cup g_1(h) \cup g_2(h)=0.
$$
In addition, denoting the constant terms of $\tilde{\alpha}_1|_E$ and $\tilde{\alpha}_2|_{E}$ with respect to $h$ by $u_1$ and $u_2$ respectively, 
we have $g_i(h) = (\tilde{\alpha}_i|_{E}-u_i)/h$ for $i=1,2$. 

Now we return to proving \eqref{poincare pairing phi2 cup phi2}. 
Notice that $\pi_{+,*}\circ\pi_+^* = \Id_{H^*(X_+)}$, which implies that 
$$
    \int_{X_+}\alpha_1\cup\alpha_2 = \int_{X_+\times_{X_0}X_-} \pi_{+}^*(\alpha_1) \cup \pi_+^*(\alpha_2).
$$
Then, it suffices to prove 
$$
    \int_{X_+\times_{X_0}X_-} \pi_{+}^*(\alpha_1) \cup \pi_+^*(\alpha_2) = \int_{X_-}\pi_{-,*}(\pi_{+}^*(\alpha_1)) \cup \pi_{-,*}(\pi_{+}^*(\alpha_2)).
$$
Denote $\alpha_1|_{\bbP(N_{F_0,+})}$ and $\alpha_2|_{\bbP(N_{F_0,+})}$ by $f_1(h_+)$ and $f_2(h_+)$ respectively, where $\deg_{h_+}(f_1)$ and $\deg_{h_+}(f_2)$ are both at most $r_{F_0,+}-1$. 
By the above discussion on the blowup model applied to $\pi_-$, it suffices to prove that
\begin{equation}\label{int on double proj bundle}
    \int_{\bbP(N_{F_0,+})\times_{F_0}\bbP(N_{F_0,-})}(h_+ + h_-)\frac{f_1(h_+)-f_1(-h_-)}{h_+ + h_-}\frac{f_2(h_+)-f_2(-h_-)}{h_+ + h_-}=0,
\end{equation} 
where recall $h_+ + h_-$ is the first Chern class of $N_{E/\tX}$ and $f_1(-h_-)$, $f_2(-h_-)$ are the constant terms of $f_1(h_+)$, $f_2(h_+)$ with respect to $h_++h_-$ respectively. The integrand on the left-hand side of \eqref{int on double proj bundle} can be rewritten as
$$
    (f_1(h_+)-f_1(-h_-))\frac{f_2(h_+)-f_2(-h_-)}{h_+ + h_-}
$$
where $\frac{f_2(h_+)-f_2(-h_-)}{h_++h_-}$ is a polynomial in $h_+,h_-$ with coefficients in $H^*(F_0)$ and degree at most $r_{F_0,+}-2$. Because $f_1(h_+)-f_1(-h_-)$ only contains monomials in only $h_+$ or $h_-$, for any monomial in the integrand, either its $h_+$-degree is strictly less than $r_{F_0,+}-1$, or its $h_-$-degree is less than $r_{F_0,-}-1$. Therefore, the integral \eqref{int on double proj bundle} is zero.
\end{proof}

\section{Fourier transformations on quantum cohomology}
\label{sect:Fourier}
In this section, we discuss Fourier transformations on equivariant quantum cohomology and quantum $D$-modules.
As in Section~\ref{subsec:BB decom}, let $W$ be a smooth projective variety with a $\bbC^*$-action whose fixed locus is $W^{\bbC^*} = F_1 \sqcup \cdots \sqcup F_r$. We do not assume that $r = 3$. 


\subsection{Continuous Fourier transformations on Givental spaces}\label{sect:ContinuousFT}
In this subsection, we review the continuous Fourier transformation associated to a fixed component $F = F_i$ of the $\bbC^*$-action on $W$, following \cite[Section 4.2]{iritani2023quantum}. The following result is due to \cite[Theorem 2(i)]{brown2014gromov} and \cite[Theorem 3.5(1)]{fan2020on}; see also \cite[Proposition 4.2]{iritani2023quantum}.

\begin{proposition}[\cite{brown2014gromov,fan2020on}]\label{prop:WRestriction}
Let $\bff$ be a point on the $\bbC^*$-equivariant Givental cone $\cL_W$ of $W$ and $\bff |_F$ be its restriction to $F$. Then the Laurent expansion of $\bff |_F$ at $z = 0$ is a point on the twisted Givental cone $\cL_{F, (N_{F/W}, e_{\bbC^*}^{-1})}$.
\end{proposition}

Here, the Givental space of $F$ is extended from the Novikov ring $\bbC\formal{Q_F}$ of $F$ to the Novikov ring $\bbC\formal{Q_W}$ of $W$ via the pushforward of curve classes $Q_F^\beta \mapsto Q_W^{i_{F,*}(\beta)}$. 

For a point $\bff$ in the $\bbC^*$-equivariant Givental space $\cH_W^{\rat}$, the continuous Fourier transformation of $\bff$ is defined by the stationary phase expansion of the Mellin-Barnes integral
\begin{equation}\label{eqn:MellinBarnes}
    \int e^{\lambda \log S_F/z}G_F \bff |_F d\lambda
\end{equation}
as $z \to 0$. Here, $S_F \coloneqq \hS^{\sigma_F(1)}$ and $G_F$ is the operator
$$
    G_F \coloneqq \prod_\delta \frac{1}{\sqrt{-2\pi z}}(-z)^{-\delta/z}\Gamma \left( -\frac{\delta}{z} \right)
$$
where the product is taken over the $\bbC^*$-equivariant Chern roots $\delta$ of $N_{F/W}$. The asymptotic expansion of $G_F$ via the Stirling approximation of the Gamma function is related to quantum Riemann-Roch operators. Specifically, recall the weight decomposition~\eqref{eqn:NormalWtDecomp} of $N_{F/W}$. For $c \in \bbZ$, let
$$
    \rho_c \coloneqq c_1(N_{F,c}), \qquad r_c \coloneqq \rank(N_{F,c}), \qquad \Delta_c \coloneqq \Delta_{(N_{F,c}, e_{\bbC^*}^{-1})}. 
$$
Then as $z \to 0$ along an appropriate angular region, we have
$$
    \log G_F \sim \sum_{c \in \bbZ} - r_c \frac{c\lambda \log (c\lambda) - c\lambda}{z} - \log \Delta_c.
$$
This leads to the following asymptotics of the integral~\eqref{eqn:MellinBarnes}:
$$
    \int e^{\eta(\lambda)/z} \left( \prod_{c \in \bbZ} \Delta_c^{-1} \right) \bff |_F d\lambda
$$
where
$$
    \eta(\lambda) \coloneqq \lambda \log S_F - \sum_{c \in \bbZ} r_c(c\lambda \log (c\lambda) - c\lambda)
$$
is the phase function. Set
$$
    c_F \coloneqq \sum_{c \in \bbZ} r_c c, \qquad \rho_F \coloneqq c_1(N_{F/W}) = \sum_{c \in \bbZ} \rho_c, \qquad r_F \coloneqq \rank(N_{F/W}) = \sum_{c \in \bbZ} r_c.
$$
We assume that $c_F \neq 0$. Then $\eta(\lambda)$ has $\abs{c_F}$ critical points, which are given by
\begin{equation}\label{eqn:Lambdaj}
    \lambda_j \coloneqq e^{2\pi\sqrt{-1}j/c_F} (S_F)^{1/c_F} \prod_{c \in \bbZ} c^{-r_c c/c_F}, \qquad j = 0, \dots, \abs{c_F}-1.
\end{equation}

\begin{definition}\label{def:ContinuousFT}
For $j = 0, \dots, \abs{c_F}-1$, the \emph{continuous Fourier transformation}
$$
    \sF_{F, j}\colon \cH_W^{\rat} \longrightarrow S_F^{-\rho_F/(c_F z) - (r_F - 1)/(2c_F)} H^*(F)[z, z^{-1}]\biglaurent{S_F^{-1/c_F}}\formal{Q_W}
$$
is defined by the stationary phase expansion of~\eqref{eqn:MellinBarnes} at the critical point $\lambda_j$ as $z \to 0$: For $\bff \in \cH_W^{\rat}$, we have
$$
    \int e^{\lambda \log S_F/z}G_F \bff |_F d\lambda \sim \sqrt{2\pi z} e^{c_F \lambda_j/z} \sF_{F,j}(\bff).
$$
\end{definition}

The leading order of the continuous Fourier transformation as Laurent series in $S_F^{-1/c_F}$ can be described as follows. For $j = 0, \dots, \abs{c_F} - 1$, define
\begin{equation}\label{eqn:qFj}
    q_{F, j} \coloneqq \sqrt{\frac{\lambda_j}{c_F}} \prod_{c \in \bbZ} (c\lambda_j)^{-r_c/2} \quad \in \bbC S_F^{-(r_F-1)/(2c_F)}
\end{equation}
and
$$
    h_{F, j} \coloneqq -\frac{2\pi\sqrt{-1} j}{c_F} \rho_F + \sum_{c \in \bbZ} \left(\frac{r_c c}{c_F} \rho_F - \rho_c \right) \log c \quad \in H^2(F).
$$
Then, for $\bff$ such that $\bff |_F$ takes form $\alpha \lambda^n + O(\lambda^{n-1})$ for some $\alpha \in H^*(F)$, $\alpha \neq 0$, we have
\begin{equation}\label{eqn:ContFTLeading}
    \sF_{F, j}(\bff) = q_{F, j} e^{h_{F, j}/z} S_F^{-\rho_F/(c_F z)} \lambda_j^n \left( \alpha + O(S_F^{1/c_F}) \right).
\end{equation}

We collect useful properties of the continuous Fourier transformation from \cite[Section 4.2]{iritani2023quantum}. We first make the following definition.


\begin{definition}\label{def:KirwanFixedComp}
We define the map
$$
    \kappa_F \colon H_{\bbC^*}^*(W) \longrightarrow H^*(F), \qquad \alpha \longmapsto (\alpha |_F) |_{\lambda = -\frac{\rho_F}{c_F}}.
$$
The dual map of the restriction of $\kappa_F$ to $H^2$ is defined by
$$
    \kappa_F^* \colon N_1(F) \longrightarrow N_1^{\bbC^*}(W)_{\bbQ}, \quad \beta \longmapsto i_{F,*}(\beta) - \frac{(\beta, \rho_F)}{c_F} \sigma_F(1).
$$
\end{definition}

We note that in the general case where $c_F \neq \pm 1$, the evaluation of $\lambda$ in $\kappa_F$ is rational, and $\kappa_F^*$ is only defined over $\bbQ$. Moreover, $\kappa_F$ is in general not surjective, and $\kappa_F^*$ not injective. The map $\kappa_F^*$ induces a monoid homomorphism $\NE_{\bbN}(F) \to \NE_{\bbN}(W) + \bbZ \frac{1}{2c_F} \sigma_F(-1)$ and a ring map
$$
    \bbC\formal{Q_F} \longrightarrow \bbC\Biglaurent{S_F^{-\frac{1}{2c_F}}}\formal{Q_W}, \quad Q_F^{\beta} \longmapsto \hS^{\kappa_F^*(\beta)} = Q_W^{i_{F,*}\beta}S_F^{-\frac{(\beta, \rho_F)}{c_F}}.
$$

\begin{proposition}\label{prop:ContFTCone}
The continuous Fourier transformation $\sF_{F,j}$ satisfies the following properties:
\begin{enumerate}[wide]
    \item \label{prop:prop1 of Continuous FT on cone level} $\sF_{F,j}\circ\hcS^{\beta} = \hS^{\beta} \circ\sF_{F,j}$ for any $\beta \in N_1^{\bbC^*}(W)$.
    
    \item \label{prop:prop2 of Continuous FT on cone level} $\sF_{F,j}\circ\lambda = (z \lambda \hS\partial_{\hS} + \lambda_j )\circ \sF_{F,j}$.
    
    \item \label{prop:prop3 of Continuous FT on cone level}
    $\sF_{F,j}\circ(z\partial_z-z^{-1}c_1^{\bbC^*}(W)+\mu_W^{\bbC^*}+\frac{1}{2})= (z\partial_z-z^{-1}c_1(F)-z^{-1}c_F\lambda_j+\mu_F)\circ \sF_{F,j}$.
    
    \item \label{prop:prop4 of Continuous FT on cone level} 
    For any $\xi\in H^2_{\bbC^*}(W)$ such that $i_{\pt}^*(\xi) = 0 \in H^2_{\bbC^*}(\pt)$, where $i_{\pt}$ is the inclusion of a point in $F$ into $W$, we have $\sF_{F,j}\circ z\xi Q\partial_Q = z\xi \hS\partial_{\hS} \circ\sF_{F,j}$ and $\sF_{F,j}\circ (\kappa_F(\xi)) = \xi \circ\sF_{F,j}$.
    
    \item \label{prop:prop5 of Continuous FT on cone level}
    $\lambda_j^{c_1(N_{F/W})/z}\circ \sF_{F,j}$ is homogeneous of degree $1-r_F$.
    
    \item \label{prop:prop6 of Continuous FT on cone level} 
    There exist  $\sigma_j(\theta) \in H^*(F)\big[S_F^{1/c_F}, S_F^{-1/c_F}\big]\formal{Q_W,\theta}$ and $\vartheta(\theta) \in q_{F,j} H^*(F)[z]\biglaurent{S_F^{-1/c_F}}\formal{Q_W,\theta}$ such that 
    $$
        \sF_{F,j}(J_W(\theta)) = S_F^{-\rho_F/(c_Fz)} M_F(\sigma_j(\theta))\big|_{Q_F \mapsto \hS}\text{ }\vartheta(\theta)
    $$ 
    where the subscript $Q_F \mapsto \hS$ means to replace $Q_F^\beta$ with $\hS^{\kappa_{F}^*(\beta)}$ for any $\beta \in \NE_{\bbN}(F)$. 
    In addition, $\sigma_j(\theta)$ and $\vartheta(\theta)$ satisfy
    $$
        \sigma_j(\theta) \big|_{Q_W= 0} \in h_{F,j} + \fm' H^*(F)\formal{S_F^{-1/c_F}, \theta_F S_F^{\bullet/c_F}},
    $$
    $$
        \vartheta(\theta) \big|_{Q_W= 0} \in q_{F,j} \big( 1 + \fm'' H^*(F)[z]\bigformal{S_F^{-1/c_F}, \theta_F S_F^{\bullet/c_F}} \big)
    $$
    where:
    \begin{itemize}
        \item $\theta_F = \{\theta_F^{i,k}\}_{k \in \bbN, i}$ is the collection of parameters for the restriction $\theta|_F = \sum_{i,k} \theta_F^{i,k} \phi_{F,i} \lambda^k$, where $\{\phi_{F,i}\}$ is a $\bbC$-basis of $H^*(F)$, and each $\theta_F^{i,k}$ can be written as a $\bbC$-linear combination of parameters in $\theta$, 
        
        \item $\theta_F S_F^{\bullet/c_F}$ is the infinite set $\big\{\theta_F^{i,k} S_F^{k/c_F} \big\}_{k \in \bbN, i}$ of variables,
        
        \item $\fm' \subset \bbC\bigformal{ S_F^{-1/c_F}, \theta_F S_F^{\bullet/c_F}}$ and $\fm'' \subset \bbC[z]\bigformal{S_F^{-1/c_F}, \theta_F S_F^{\bullet/c_F}}$ are the ideals generated by $S_F^{-1/c_F}$ and $\theta_F^{i,k} S_F^{k/c_F}$ respectively.
    \end{itemize}
\end{enumerate}
\end{proposition}

\subsection{Continuous Fourier transformations on quantum $D$-modules}\label{sect:ContinuousQDM}
We translate the continuous Fourier transformations defined above into transformations on quantum $D$-modules. For a $\bbC^*$-fixed component $F$ of $W$ satisfying $c_F \neq 0$, we first extend the base of the quantum $D$-module of $F$.


\color{black}

\begin{definition}\label{Extension of QDK for a fixed component} 
We define the formal Laurent extension
$$
    \QDM(F)^{\La} \coloneqq \QDM(F) \otimes_{\bbC[z]\formal{Q_F,\tau_F}} \bbC[z]\Biglaurent{S_F^{-\frac{1}{2c_F}}}\formal{Q_W,\tau_F} = H^*(F)[z]\Biglaurent{S_F^{-\frac{1}{2c_F}}}\formal{Q_W,\tau_F}
$$
through the ring extension $\bbC[z]\formal{Q_F,\tau_F}\to \bbC[z]\Biglaurent{S_F^{-\frac{1}{2c_F}}}\formal{Q_W,\tau_F}$ induced by the dual map $\kappa_F^*$. The pullback connection $\nabla$ is defined as follows:
\begin{align*}
\nabla_{\tau_F^i} &= \nabla_{\tau_F^i} \otimes \Id + \Id \otimes \partial_{\tau_F^i} 
= \partial_{\tau_F^i} + z^{-1} \left( \phi_{F,i} \star_{\tau_F} \right), \\
\nabla_{z \partial_z} &= \nabla_{z \partial_z} \otimes \Id + \Id \otimes z \partial_z 
= z \partial_z - z^{-1} \left( E_F \star_{\tau_F} \right) + \mu_F, \\
\nabla_{\xi \hS \partial_{\hS}} &= \nabla_{\kappa_F(\xi) Q \partial_Q} \otimes \Id + \Id \otimes \xi \hS \partial_{\hS} 
= \xi \hS \partial_{\hS} + z^{-1} \left( \kappa_F(\xi) \star_{\tau_F} \right).
\end{align*} 
\end{definition}

By an abuse of notation, we will also use $M_F(\tau_F)$ to denote the base change $M_F(\tau_F)\big|_{Q_F\to\hS}$ induced by $\kappa_F^*$. To define continuous Fourier transformations on the level of quantum $D$-modules, we need the pullback of $\QDM(F)^{\La}$ under the change of coordinates $\sigma_j(\theta)$ in Proposition~\ref{prop:ContFTCone}\eqref{prop:prop6 of Continuous FT on cone level}. We first remark on the well-definedness of this pullback.

\begin{remark}\label{rem:ContPullbackDefined}
The element $\sigma_j(\theta)$ in Proposition~\ref{prop:ContFTCone}\eqref{prop:prop6 of Continuous FT on cone level} does not directly induce formal changes of coordinates $\bbC[z]\formal{Q_F,\tau_F} \to \bbC[z]\formal{Q_F,\theta}$ because of the existence of the nonzero constant term 
$h_{F,j}\in H^2(F)$. However, we may still apply the divisor equation to define the pullback quantum $D$-modules $\sigma_j^*\QDM(F)^{\La}$.
More precisely, we consider the subring $\bbC[z]\formal{Q_Fe^{\tau_F^{(2)}},\tau_F^{(\ne 2)}} \subset \bbC[z]\formal{Q_F,\tau_F}$, where $\tau_F=\tau_F^{(2)}+\tau_F^{(\ne 2)}$ decomposes $\tau_F$ into the part of degree 2 and the rest, and the notation $Q_Fe^{\tau_F^{(2)}}$ means that $(Q_Fe^{\tau_F^{(2)}})^\beta = Q_F^\beta e^{(\beta, \tau_F^{(2)})}$.
The submodule $H^*(F)[z]\formal{Q_Fe^{\tau_F^{(2)}},\tau_F^{(\ne2)}}$ of $\QDM(F)$ is closed under $\nabla_{z^2\partial_z}$, $\nabla_{z\tau_F^i}$, and $\nabla_{z\xi Q\partial_Q}$ by the divisor equation. 
Now suppose $\sigma(\theta)\in H^*(F)\formal{Q_F,\theta}$ is such that $\sigma(\theta)|_{\theta=Q_F=0} = \sigma_F\in H^2(F)$. Then $\sigma(\theta)$ induces a homomorphism $ \bbC[z]\formal{Q_Fe^{\tau_F^{(2)}},\tau_F^{(\ne 2)}} \to \bbC[z]\formal{Q_F,\theta}$ without the convergence problem. Therefore, we can base change the submodule $H^*(F)[z]\formal{Q_Fe^{\tau_F^{(2)}},\tau_F^{(\ne2)}}$ to $\bbC[z]\formal{Q_F,\theta}$ and define the pullback quantum $D$-module $\sigma^*\QDM(F)$. Explicitly, $\sigma^*\QDM(F)$ is the module $H^*(F)[z]\formal{Q_F,\theta}$ together with the pullback quantum connection:
\begin{align*}
    \nabla_{\theta^{i,k}} &= \partial_{\theta^{i,k}} + z^{-1} \left( \partial_{\theta^{i,k}} \sigma(\theta) \right) \star_{\sigma(\theta)},\\
    \nabla_{z \partial_z}&= z \partial_z - z^{-1} \left( E_F \star_{\sigma(\theta)} \right) + \mu,\\
    \nabla_{\xi Q\partial_Q} &= \xi Q\partial_Q + z^{-1} \left( \xi \star_{\sigma(\theta)} \right) + z^{-1} \left( \xi Q\partial_Q \sigma(\theta) \right) \star_{\sigma(\theta)}.
\end{align*}
Similarly, for $\sigma(\theta) \in H^*(F)\Biglaurent{S_F^{-\frac{1}{2c_F}}}\formal{Q_W,\theta}$ such that 
$
    \sigma(\theta)|_{Q_W=S_F^{-\frac{1}{2c_F}}=\theta_F=0} \in H^2(F),
$
the pullback $\sigma^*\QDM(F)^{\La}$ is well-defined. It is the module $H^*(F)[z]\Biglaurent{S_F^{-\frac{1}{2c_F}}}\formal{Q_W,\theta}$ together with the pullback connection:
\begin{align*}
\nabla_{\theta^{i,k}} &= \partial_{\theta^{i,k}} + z^{-1} \left( \partial_{\theta^{i,k}} \sigma(\theta) \right) \star_{\sigma(\theta)},
\\
\nabla_{z \partial_z}&= z \partial_z - z^{-1} \left( E_F \star_{\sigma(\theta)} \right) + \mu,
\\
\nabla_{\xi \hS \partial_{\hS}}& = \xi \hS \partial_{\hS} + z^{-1} \left( \kappa_F(\xi) \star_{\sigma(\theta)} \right) + z^{-1} \left( \xi \hS \partial_{\hS} \sigma(\theta) \right) \star_{\sigma(\theta)}.
\end{align*}
Notice the base change of the Novikov variables induced by $\kappa_{F}^*$ above.
\end{remark}

We now construct continuous Fourier transformations on the level of quantum $D$-modules.

\begin{proposition}\label{prop:ContFTDmodule}
Let $F$ be a $\bbC^*$-fixed component of $W$, $j = 0, \dots, |c_F|-1$, and $\sigma_j(\theta)\in H^*(F)\Biglaurent{S_F^{-\frac{1}{2c_F}}}\formal{Q_W,\theta}$ be given in Proposition~\ref{prop:ContFTCone}\eqref{prop:prop6 of Continuous FT on cone level}. There is a $\bbC[z]\formal{Q_W,\theta} \big[\hS,Q_W^{-1} \big]$-module homomorphism
$$
    \FT_{F,j} \colon \QDM_{\bbC^*}(W)[Q_W^{-1}] \to \sigma_j^* \QDM(F)^{\La}[Q_W^{-1}]
$$
satisfying the following properties:
\begin{enumerate}[wide]
\item \label{item:ContFTDmoduleTarget} $\FT_{F,j}(\QDM_{\bbC^*}(W)) \subset \sigma_j^* \QDM(F)^{\La}$.
\item \label{item:ContFTDmoduleConnection} $\FT_{F,j}$ intertwines $z\nabla_{\xi Q\partial_Q}$ with $z\nabla_{\xi\hS\partial_{\hS}}+i_{\pt}^* (\xi)|_{\lambda=\lambda_j}$, $\nabla_{\theta^{i,k}}$ with $\nabla_{\theta^{i,k}}$, and $\nabla_{z\partial_z}+\frac{1}{2}$ with $\nabla_{z\partial_z}-\frac{c_F\lambda_j}{z}$. 
\item \label{item:ContFTDmoduleHomo}$\FT_{F,j}$ is homogeneous of degree $1-r_F$ as a $\bbC[z]\formal{Q_W,\theta}\big[\hS,Q_W^{-1}\big]$-module homomorphism.

\item \label{item:ContFTDmoduleCommute} $\FT_{F,j}$ satisfies the commutative diagram
$$
\begin{tikzcd}
\QDM_{\bbC^*}(W)[Q_W^{-1}] \arrow[r, "\FT_{F,j}"] \arrow[d, "M_W(\theta)"'] & \sigma_j^* \QDM(F)^{\La}[Q_W^{-1}] \arrow[d, "M_F(\sigma_j(\theta))"] \\
\cH_W^{\mathrm{rat}}[Q_W^{-1}] \arrow[r, " S_F^{\rho_F/(zc_F)} \sF_{F,j}"] & \cH_F^{\La}[Q_W^{-1}]
\end{tikzcd}
$$
where we set $\cH_F^{\La} \coloneqq H^*(F)[z,z^{-1}]\Biglaurent{S_F^{-\frac{1}{2c_F}}}\formal{Q_W}$. 
\color{black}
\end{enumerate}
\end{proposition}


\begin{proof}
By differentiating the equation $S_F^{\frac{\rho_F}{zc_F}}\sF_{F,j}(J_W(\theta))=M_F(\sigma_j(\theta))\vartheta_j(\theta)$ in Proposition~\ref{prop:ContFTCone}\eqref{prop:prop6 of Continuous FT on cone level}, we define $\FT_{F,j}$ to be the unique $\bbC[z]\formal{Q_W,\theta}[Q_W^{-1}]$-module homomorphism that sends $\phi_i\lambda^k$ to $\nabla_{\theta^{i,k}}\vartheta_j(\theta)$ for all $i, k$. Then the commutative diagram~\eqref{item:ContFTDmoduleCommute} naturally holds. Proposition~\ref{prop:ContFTCone}\eqref{prop:prop1 of Continuous FT on cone level} and the relation $M_W(\theta)\circ\hS^{\beta}=\hcS^{\beta}\circ M_W(\theta)$ together imply that $\FT_{F,j}$ is a $\bbC[z]\formal{Q_W,\theta}\big[\hS,Q_W^{-1}\big]$-module homomorphism. In addition, property~\eqref{item:ContFTDmoduleTarget} follows from that $\{\phi_i\lambda^k\}_{i,k}$ is a $\bbC[z]\formal{Q_W,\theta}$-basis of $\QDM_{\bbC^*}(W)$ and thus we can define $\FT_{F,j}$ without localizing $Q_W$. Property~\eqref{item:ContFTDmoduleConnection} follows from the commutative relations~\eqref{prop:prop1 of Continuous FT on cone level}\eqref{prop:prop2 of Continuous FT on cone level}\eqref{prop:prop3 of Continuous FT on cone level}\eqref{prop:prop4 of Continuous FT on cone level} in Proposition~\ref{prop:ContFTCone}. Property~\eqref{item:ContFTDmoduleHomo} follows from Proposition~\ref{prop:ContFTCone}\eqref{prop:prop5 of Continuous FT on cone level}.
\end{proof}

Notice that $\FT_{F,j}$ intertwines $\nabla_{z\partial_z}+\frac{1}{2}$ with $\nabla_{z\partial_z}-\frac{c_F\lambda_j}{z}$. We make the shift 
$$
    \zeta_j(\theta) \coloneqq \sigma_j(\theta) + c_F\lambda_j.
$$
Because 
$$
    \zeta_j(\theta)|_{Q_W=S_F^{-\frac{1}{2c_F}}=\theta_F=0} = \sigma_j(\theta)|_{Q_W=S_F^{-\frac{1}{2c_F}}=\theta_F=0} \in H^2(F),
$$
the pullback quantum $D$-module $\zeta_j^*\QDM(F)^{\La}$ is well-defined by Remark \ref{rem:ContPullbackDefined}. 
As a module, $\zeta_j^*\QDM(F)^{\La}$ is isomorphic to $H^*(F)[z]\Biglaurent{S_F^{-\frac{1}{2c_F}}}\formal{Q_W,\theta}$. By adding in the shift $c_F\lambda_j$, we only change the coordinate $\tau_F^0$ in the direction of $1\in H^0(F)$, which preserves the quantum product and shifts the Euler vector field by $c_F\lambda_j$. In addition, because 
$$
    \xi\hS\partial_{\hS}(c_F\lambda_j) = (\frac{1}{c_F}\sigma_F(1), \xi) c_F\lambda_j = (\sigma_F(1), \xi) \lambda_j = i_{\pt}^*(\xi)|_{\lambda=\lambda_j},
$$
we can compute the shift on the pullback quantum connection as 
$$
    \zeta_j^*\nabla = \sigma_j^*\nabla-c_F\lambda_jz^{-2}dz+\lambda_jz^{-1}\frac{dS_F}{S_F}.
$$
By composing $\FT_{F,j}$ with the identity on the underlying module 
$$
    \sigma_j^*\QDM(F)^{\La}[Q_W^{-1}]\to\zeta_j^*\QDM(F)^{\La}[Q_W^{-1}],
$$
we can rewrite our the continuous Fourier transformation as follows.

\begin{corollary}\label{Corollary of continuous FT(shift on the pullback)}
Let $F$ be a $\bbC^*$-fixed component of $W$ and $j = 0, \dots, |c_F|-1$. There is a $\bbC[z]\formal{Q_W,\theta}[\hS,Q_W^{-1}]$-module homomorphism 
$$
    \FT_{F,j}\colon \QDM_{\bbC^*}(W)[Q_W^{-1}]\to \zeta_j^*\QDM(F)^{\La}[Q_W^{-1}]
$$
which is homogeneous of degree $1-r_F$, intertwines $\nabla+\frac{1}{2}\frac{dz}{z}$ with $\nabla$, and satisfies $\FT_{F,j}(\QDM_{\bbC^*}(W))\subset \zeta_j^*\QDM(F)^{\La}$.
\end{corollary}

\subsection{Discrete Fourier transformations}\label{sect:DiscreteFT}
In this subsection, we review the discrete Fourier transformation associated to a smooth GIT quotient $X = W \gitquot \bbC^*$, following \cite[Section 4.1]{iritani2023quantum}, and prove Conjecture \ref{conj:introReduction} in the case of the highest/lowest GIT quotients. We assume that the ample linearization is chosen such that $W^{ss} = W^{s} \neq \emptyset$ and we drop the notation for the linearization in the subscript. As in Section~\ref{sect:C*CurveClass}, we use the section class $\lambda^*=\sigma_{F_r}(1)$ associated to the lowest $\bbC^*$-fixed component $F_r$ of $W$ and write $S = \hS^{\lambda^*}$, $\cS = \hcS^{\lambda^*}$.

Consider the dual map 
$$
    \kappa_X^* \colon N_1(X) \longrightarrow N_1^{\bbC^*}(W)
$$
of the restriction of the Kirwan map $\kappa_X$ to $H^2$. It is injective. By Lemma~\ref{lem:KirwanC1}, $\kappa_X^*$ induces a homogeneous embedding of rings
$$
    \bbC \formal{\NE_{\bbN}(X)} \longrightarrow \bbC \formal{C_{X,\bbN}^{\vee}}, \quad  Q_X^{\beta} \longmapsto \hS^{\kappa_X^*(\beta)}.
$$
We define the following extension of the Givental space of $X$:
$$
    \cH_X^{\text{ext}} \coloneqq H^*(X)[z,z^{-1}]\formal{C_{X,\bbN}^{\vee}}.
$$ 

\begin{definition}\label{def:DiscreteFT}
The \emph{discrete Fourier transformation} 
$$
    \sfF_X\colon \cH_W^{\rat} \dashrightarrow \cH_X^{\ext}
$$
is defined by
$$
    \sfF_X(\bff) \coloneqq \sum_{k \in \bbZ} S^{-k} \kappa_X(\cS^k \bff).
$$
\end{definition}


The map $\sfF_X$ is defined on $\bff$ if $\cS^k \bff \big|_X$ is regular at $\lambda = 0$ (so that $\kappa_X(\cS^k \bff)$ is defined) for all $k \in \bbZ$, and the defining summation as a power series in $\hS = (Q_W, S)$ is supported on $\beta + C_{X,\bbN}^{\vee}$ for some $\beta \in N_1(W)$. We collect useful properties of the discrete Fourier transformation from \cite[Section 4.1]{iritani2023quantum}.

\begin{proposition}\label{prop:DiscreteProp}
The discrete Fourier transformation $\sfF_X$ satisfies the following properties (conditioned on definedness):
\begin{enumerate}[wide]
    \item \label{prop:prop1 of Discrete FT on cone level} $\sfF_X\circ \cS^k = S^k\circ \sfF_X$ for any $k \in\bbZ$.
    
    \item \label{prop:prop2 of Discrete FT on cone level} $\sfF_X\circ (z\xi Q\partial_Q+\xi) = (z\xi\hS\partial_{\hS} + \kappa_X(\xi))\circ \sfF_X$ for any $\xi\in H_{\bbC^*}^2(W)$. In particular, $\sfF_X\circ \lambda =(z\lambda\hS\partial_{\hS}+\kappa_{X}(\lambda))\circ\sfF_X$.
    
    \item \label{prop:prop3 of Discrete FT on cone level} $\sfF_X\circ (z\partial_z-z^{-1}c_1^{\bbC^*}(W) +\mu_W^{\bbC^*}+\frac{1}{2}) = (z\partial_z-c_1(X) + \mu_X)\circ \sfF_X$.
    
    \item \label{prop:prop4 of Dicrete FT on cone level} $\sfF_X\circ \partial_{\theta^{i,k}}=\partial_{\theta^{i,k}}\circ \sfF_X$ for any $i, k$.
    
    \item \label{prop:prop5 of Discrete FT on cone level} $\sfF_X$ is homogeneous.
\end{enumerate}
\end{proposition}

As discussed in Conjecture~\ref{conj:introReduction}, $\sfF_X$ is conjectured to be defined on the $\bbC^*$-equivariant Givental cone of $W$ and to preserve the Givental cones. We prove Conjecture~\ref{conj:introReduction} in the case where $X$ is the highest/lowest GIT quotient of $W$ in the sense of Definition~\ref{def:highest/lowest GIT}, which is a slight generalization of the case of divisors proved in \cite[Theorem 1.13]{iritani2023quantum}.

\begin{theorem}\label{thm:Reduction}
Conjecture~\ref{conj:introReduction} holds for the highest/lowest GIT quotient $X$. More generally, let $x = (x_1, x_2, \dots)$ be a collection of formal parameters and $\bff$ be a $\bbC[\lambda]\formal{Q_W, x}$-valued point on the Givental cone $\cL_W$. Then $z\sfF_X(\bff)$ is defined and is a $\bbC\formal{Q_W, x}$-valued point on the Givental cone of $X$ (with the same extension as in the definition of $\cH_X^{\ext}$).
\end{theorem}

To prove the theorem, we first show that the discrete Fourier transformation $\sfF_X$ is defined on the tangent space of the equivariant Givental cone of $W$.

\begin{lemma}\label{lem:DiscreteDefined}
For the highest/lowest GIT quotient $X$, if $\bff \in \cH_W^{\rat}$ lies in a tangent space of $\cL_W$, then $\sfF_X(\bff)$ is defined.
\end{lemma}

Since the highest case can be turned into the lowest case by reversing the $\bbC^*$-action on $W$, the proofs in the rest of this subsection will focus on the lowest case. 

\begin{proof}
We prove the lemma for the lowest GIT quotient $X$. The lemma directly follows from the argument in \cite[Appendix A]{iritani2023quantum} (which generalizes the proof of \cite[Proposition 4.1]{iritani2023quantum}). The only part in the argument that needs to be generalized to the case where $F_r$ is not necessarily a divisor is the following claim. Expand $\bff$ in $Q_W$ as $\bff = \sum_{\beta \in \NE_{\bbN}(W)} \bff_{\beta} Q_W^\beta$. If $k \in \bbZ_{>0}$ and $\beta \in \NE_{\bbN}(W)$ are such that in the prefactor of 
$$
    (\cS^k \bff) |_{F_r} = \frac{\prod_{c = -\infty}^0 e_{\lambda + cz}(N_{F_r/W})}{\prod_{c = -\infty}^{-k} e_{\lambda + cz}(N_{F_r/W})} e^{-kz\partial_\lambda} \bff |_{F_r},
$$
the factor $e_\lambda(N_{F_r/W})$ (corresponding to $c=0$) appearing in the numerator is not cancelled by the denominator of $e^{-kz\partial_\lambda} \bff_{\beta} |_{F_r}$, then $\kappa_X(\cS^k \bff_\beta) = 0$. To prove this claim, observe that under the presentation
$$
    H^*(X) \cong H^*(F_r)[p] / (p^l + c_1(N_{F_r/W}) p^{l-1} + \cdots + c_l(N_{F_r/W}))
$$
of the cohomology of $X = \bbP(N_{F_r/W})$, where $p \coloneqq c_1(\cO(1))$, the Kirwan map $\kappa_X$ may be factorized as
$$
    \xymatrix{
        \kappa_X\colon H^*_{\bbC^*}(W) \ar[r]^-{|_{F_r}} & H^*_{\bbC^*}(F_r) \ar[r]^{\lambda \mapsto p} & H^*(X).
    }        
$$
Then the claim follows from the fact that 
$
    e_\lambda(N_{F_r/W}) = \lambda^l + c_1(N_{F_r/W}) \lambda^{l-1} + \cdots + c_l(N_{F_r/W})
$
is sent to zero in $H^*(X)$ under $\lambda \mapsto p$.
\end{proof}

\begin{proof}[Proof of Theorem~\ref{thm:Reduction}]
We prove the theorem for the lowest GIT quotient $X$. We denote $N \coloneqq N_{F_r/W}$. In this proof, let $L$ be a line bundle on $F_r$ such that $(N \otimes L)^\vee$ is globally generated. We use the notation $N\otimes L$ for both the vector bundle and its total space. Then $N\otimes L$ is semi-projective and is equipped with a $\bbC^*$-action under which the only nontrivial GIT quotient is $\bbP(N \otimes L) \cong \bbP(N) = X$. We have the presentations
\begin{align*}
    H^*(X) & \cong H^*(F_r)[p] / (p^l + c_1(N) p^{l-1} + \cdots + c_l(N)) \\
    & \cong H^*(F_r)[p'] / (p'^l + c_1(N \otimes L) p'^{l-1} + \cdots + c_l(N \otimes L)) 
\end{align*}
where $p \coloneqq c_1(\cO_{\bbP(N)}(1))$ and $p' \coloneqq c_1(\cO_{\bbP(N \otimes L)}(1))$. We have
$$
    p' = p - c_1(L).
$$
Let $\lambda'$ denote the equivariant parameter for the $\bbC^*$-action on $N \otimes L$. We have $H^*_{\bbC^*}(F_r) \cong H^*(F_r)[\lambda] \cong H^*(F_r)[\lambda']$ under the change of parameters
\begin{equation}\label{eqn:LambdaChange}
    \lambda' = \lambda - c_1(L).
\end{equation}

We relate the discrete Fourier transformation $\sfF_X$ defined for the GIT quotient $X = W \gitquot \bbC^*$, which we denote by $\sfF_X^W$ in this proof, to the discrete Fourier transformation $\sfF_X^{N \otimes L}$ defined for the GIT quotient $X = N \otimes L \gitquot \bbC^*$. Specifically, we factorize $\sfF_X^W$ into the following sequence of maps each of which preserves the relevant (extended) Givental cones listed below the (rational, extended) Givental spaces:
\begin{equation}\label{eqn:FTFactorize}
    \xymatrix{
        \cH_W^{\rat} \ar[r]^-{|_{F_r}}_-{\text{\ref{item:InitRes}}} & \cH_{F_r}^{\rat} \formal{Q_W} \ar[r]^-{\Delta_{N \otimes L}/\Delta_N}_-{\text{\ref{item:DoubleTwist}}} & \cH_{F_r}^{\rat}  \formal{Q_W} \ar[r]^-{\left(|_{F_r}\right)^{-1}}_-{\text{\ref{item:VBRes}}} & \cH_{N \otimes L}^{\rat} \formal{Q_W}  \ar[r]^-{\sfF_X^{N \otimes L}}_-{\text{\ref{item:IKFT}}} & \cH_X^{\ext}. \\
        \cL_W \ar@{}[u]|-{\cup} & \cL_{F_r, (N, e_{\bbC^*}^{-1})} \ar@{}[u]|-{\cup} & \cL_{F_r, (N \otimes L, e_{\bbC^*}^{-1})} \ar@{}[u]|-{\cup} & \cL_{N \otimes L} \ar@{}[u]|-{\cup} & \cL_X \ar@{}[u]|-{\cup}
    }
\end{equation}
We first provide the details of the maps above.
\begin{enumerate}[label=(\roman*), wide]
    \item \label{item:InitRes} 
    The restriction map $|_{F_r}$ from $W$ to $F_r$ is considered in Section~\ref{sect:ContinuousFT}. The relation between the Givental cones is given by Proposition~\ref{prop:WRestriction}.
    
    \item \label{item:DoubleTwist}
    This map is the composition of $\Delta_N^{-1}$ with $\Delta_{N \otimes L}$, where $\Delta_N$ (resp.\ $\Delta_{N \otimes L}$) denotes the quantum Riemann-Roch operator associated to the $(N, e_{\bbC^*}^{-1})$-twisted (resp.\ $(N \otimes L, e_{\bbC^*}^{-1})$-twisted) theory. The map preserves the Givental cones (Theorem~\ref{thm:QRR}). In particular, it passes through the Givental cone of the untwisted theory.\footnote{As remarked in \cite[Remark 2.6]{iritani-koto2023quantum} and \cite[Remark 2.15]{iritani2023quantum}, there is a technical issue that in the $e_{\bbC^*}^{-1}$-twisted theory it is difficult to treat the Givental cone and quantum Riemann-Roch operator formally in the $\lambda^{-1}$- (or $\lambda'^{-1}$-)adic topology, and this issue may be resolved by passing through the $\te_{\bbC^*}^{-1}$-twisted theory via a change of variables where $\te_{\bbC^*}(-) \coloneqq \sum_{i \ge 0} \lambda^{-i}c_i(-)$.}

    \item \label{item:VBRes}
    The map is a canonical lift that is inverse to the restriction map $|_{F_r}$ from $N \otimes L$ to $F_r$, and the relation between the Givental cones is given by $\bbC^*$-equivariant virtual localization applied to $N \otimes L$. In more detail, if $\{\phi_i\}$ is a $\bbC$-basis of $H^*(F_r) \cong H^*(N \otimes L)$ and $\{\phi^i\}$ is the dual basis of $H^*(F_r)$, then $\{\phi^i e_{\bbC^*}(N \otimes L)\}$ is dual to $\{\phi_i\}$ under the $\bbC^*$-equivariant Poincar\'e pairing on $N \otimes L$. Moreover, virtual localization implies that 
    $$ 
        \langle \alpha_1\psi^{k_1},\dots,\alpha_n\psi^{k_n} \rangle^{N \otimes L}_{0,n,\beta} = \langle \alpha_1\psi^{k_1},\dots,\alpha_n\psi^{k_n}\rangle^{F_r,(N \otimes L,e_{\bbC^*}^{-1})}_{0,n,\beta}
    $$
    for any $n, \beta, k_1, \dots, k_n$, and $\alpha_1, \dots, \alpha_n \in H^*(F_r)$. Thus, any point $\cJ^{(N \otimes L, e_{\bbC^*}^{-1})}_{F_r}(\bft(z))$ of $\cL_{F_r,(N \otimes L, e_{\bbC^*}^{-1})}$ given by
    $$
        z + \bft(z) + \sum_{i} \sum_{\substack{n \geq 0,\, \beta \in \NE_{\bbN}(F_r) \\ (n,\beta) \ne (0,0),(1,0)}} \phi^i e_{\bbC^*}(N \otimes L) \Big{\langle} \bft(-\psi), \ldots, \bft(-\psi), \frac{\phi_i}{z - \psi} \Big{\rangle}^{F_r,(N \otimes L, e_{\bbC^*}^{-1})}_{0, n+1, \beta} \frac{Q_{F_r}^\beta}{n!}
    $$
    is canonically lifted to $\cL_{N \otimes L}$ under the identification of $\bft(z)$ as a point in $\cH_{N \otimes L}^{+, \bbC^*}$. 
    Below, we will also denote the map by $|^{N \otimes L}$. 
    
    \item \label{item:IKFT}
    The discrete Fourier transformation $\sfF_X^{N \otimes L}$ is defined in the same way as in Definition~\ref{def:DiscreteFT} and is studied extensively by Iritani-Koto \cite{iritani-koto2023quantum}. In particular, for the noncompact, semi-projective total space $N \otimes L$, they proved Conjecture~\ref{conj:introReduction} \cite[Theorems 1.1]{iritani-koto2023quantum} and its generalized version that $z\mathrm{F}_{X}(\bff)$ lies on $\cL_X$ for any $\bff\in \cL_{N\otimes L}$ regular at $\lambda'=0$ \cite[Theorem 3.3]{iritani-koto2023quantum}.

\end{enumerate}
To complete the proof, we verify that~\eqref{eqn:FTFactorize} is a factorization of $\sfF_X^W$. Recall that the Kirwan maps can be factorized as
$$
    \xymatrix{
        \kappa_X^W\colon H^*_{\bbC^*}(W) \ar[r]^-{|_{F_r}} & H^*_{\bbC^*}(F_r) \ar[r]^{\lambda \mapsto p} & H^*(X),
    }
    \quad      
    \xymatrix{
        \kappa_X^{N \otimes L}\colon H^*_{\bbC^*}(N \otimes L) \ar[r]^-{|_{F_r}} & H^*_{\bbC^*}(F_r) \ar[r]^{\lambda' \mapsto p'} & H^*(X).
    }
$$
The specializations of $\lambda, \lambda'$ are consistent with~\eqref{eqn:LambdaChange}. Moreover, if $\{\delta\}$ denotes the collection of nonequivariant Chern roots of $N$, the change of parameters~\eqref{eqn:LambdaChange} identifies the $\bbC^*$-equivariant Chern roots $\{\delta + \lambda\}$ of $N$ with the $\bbC^*$-equivariant Chern roots $\{\delta + c_1(L) + \lambda'\}$ of $N \otimes L$, which implies that $\Delta_N = \Delta_{N \otimes L}$. Therefore, we have
$$    
    \sfF_X^{N \otimes L} \left( ((\Delta_{N \otimes L}/\Delta_N)\bff |_{F_r}) |^{N \otimes L}\right) = \sum_{k \in \bbZ} S^{-k} (\Delta_{N \otimes L}/\Delta_N)\bff |_{F_r} \big|_{\lambda' \mapsto p'} = \sum_{k \in \bbZ} S^{-k} \bff |_{F_r} \big|_{\lambda \mapsto p} = \sfF_X^{W}(\bff).
$$
\end{proof}

\begin{corollary}\label{cor:MirrorMapDiscrete}
Let $J_W(\theta)$ be the equivariant $J$-function of $W$. For the highest/lowest GIT quotient $X$, there exist $\tau_{X}(\theta)\in H^*(X)\formal{C_{X,\bbN}^{\vee},\theta}$ and $\vartheta_X(\theta)\in H^*(X)[z]\formal{C_{X,\bbN}^{\vee},\theta}$ such that 
$$
    \sfF_X(J_W(\theta)) = M_X(\tau_X(\theta))\vartheta_X(\theta),
$$
where $M_X(\tau)$ is defined under the base change $\bbC\formal{\NE_{\bbN}(X)} \to \bbC\formal{C_{X,\bbN}^{\vee}}$, and  
$$
    \tau_{X}(0)\equiv 0,\quad \vartheta_X(0)\equiv 1 \quad \mathrm{mod} \text{ }  \overline{(\hS^{\beta})}_{\beta\in C_{X,\bbN}^{\vee}\backslash \{0\}},
$$
where $\overline{(\hS^{\beta})}_{\beta\in C_{X,\bbN}^{\vee}\backslash \{0\}}$ is the ideal of $H^*(X)\formal{C_{X,\bbN}^{\vee}}$ generated by $\hS^{\beta}$ of all $\beta\in C_{X,\bbN}^{\vee}\backslash\{0\}$. 
\end{corollary}

\begin{proof}
We prove the corollary for the lowest GIT quotient $X$. By Theorem~\ref{thm:Reduction} applied to $\bff = zJ_W(\theta)$, there exist $\tau_X(\theta)\in H^*(X)\formal{C_{X,\bbN}^{\vee},\theta}$ and $\vartheta_X(\theta)\in H^*(X)[z]\formal{C_{X,\bbN}^{\vee},\theta}$ such that $\sfF_X(zJ_W(\theta)) = z M_X(\tau_X(\theta))\vartheta_X(\theta)$. By taking the limit $\theta=\hS^{\beta}=0$ for all $\beta\in C_{X,\bbN}^{\vee}\backslash\{0\}$, we have
$$
    \sfF_X(J_W(\theta))\big|_{\theta=\hS^{\beta}=0}=1, \quad M_X(\tau_X(\theta))\big|_{\theta=\hS^{\beta}=0}=e^{\frac{\tau_X(0)}{z}|_{\hS^{\beta}=0}}.
$$
Therefore, the equation $\sfF_X(J_W(\theta))=M_X(\tau_X(\theta))\vartheta_X(\theta)$ implies that
$$
    \vartheta_X(0)\big|_{\hS^{\beta}=0} = e^{-\frac{\tau_X(0)}{z}|_{\hS^{\beta}=0}}.
$$
Since the left-hand side only contains nonnegative powers of $z$ and the right-hand side only contains nonpositive powers of $z$, they are both constant and we have $\tau_{X}(0)|_{\hS^{\beta}=0}=0$ and $\vartheta_X(0)|_{\hS^{\beta}=0}=1$.
\end{proof}

\subsection{Discrete Fourier transformations on quantum $D$-modules}\label{sect:DiscreteQDM}
We translate the discrete Fourier transformations defined above into transformations on quantum $D$-modules. We work in the setup of highest/lowest GIT quotients as in Theorem~\ref{thm:Reduction}, although the discussion can be generalized for a general GIT quotient for which Conjecture~\ref{conj:introReduction} holds.

\begin{definition}\label{Extension of QDM for a GIT quotient} 
For the highest/lowest GIT quotient $X$, we define the extension 
$$
    \QDM(X)^{\ext} \coloneqq \QDM(X) \otimes_{\bbC[z]\formal{Q_X,\tau_X}} \bbC[z]\formal{C_{X,\bbN}^{\vee},\tau_X}=H^*(X)[z]\formal{C_{X,\bbN}^{\vee},\tau_X}
$$
through the ring extension $\bbC[z]\formal{Q_X,\tau_X}\to \bbC[z]\formal{C_{X,\bbN}^{\vee},\tau_X}$. 
The pullback connection $\nabla$ is defined as follows:
\begin{align*}
    \nabla_{\tau_X^i} &= \nabla_{\tau_X^i} \otimes \Id + \Id \otimes \partial_{\tau_X^i} 
= \partial_{\tau_X^i} + z^{-1} \left( \phi_{X,i} \star_{\tau_X} \right), \\
\nabla_{z \partial_z} &= \nabla_{z \partial_z} \otimes \Id + \Id \otimes z \partial_z 
= z \partial_z - z^{-1} \left( E_X \star_{\tau_X} \right) + \mu_X, \\
\nabla_{\xi \hS \partial_{\hS}} &= \nabla_{\kappa_X(\xi) Q \partial_Q} \otimes \Id + \Id \otimes \xi \hS \partial_{\hS} 
= \xi \hS \partial_{\hS} + z^{-1} \left( \kappa_X(\xi) \star_{\tau_X} \right).
\end{align*}
\end{definition}


By an abuse of notation, we will also use $M_X(\tau_X)$ to denote its base change under $\bbC\formal{\NE_{\bbN}(X)}\hookrightarrow \bbC\formal{C_{X,\bbN}^{\vee}}$. Notice that for $\tau_X(\theta)$ given in Corollary~\ref{cor:MirrorMapDiscrete}, we have $\tau_X(\theta)\big|_{\theta=\widehat{S}=0}=0$. Hence, the pullback (extended) quantum $D$-modules $\tau_X^*\QDM(X)^{\ext}$ for the highest/lowest GIT quotients are well-defined. Explicitly, $\tau_X^*\QDM (X)^{\ext}$ is the module $H^*(X)[z]\formal{C_{X,\bbN}^{\vee},\theta}$ with the pullback connection:
\begin{align*}
    \nabla_{\theta^{i,k}} &=\partial_{\theta^{i,k}}+z^{-1}(\partial_{\theta^{i,k}}\tau_X(\theta)\star_{\tau_X(\theta)}),\\ 
    \nabla_{z\partial_z} &=z\partial_z-z^{-1}(E_X\star_{\tau_X(\theta)})+\mu_X,\\
    \nabla_{\xi\hS\partial_{\hS}} &=\xi\hS\partial_{\hS}+z^{-1}(\kappa_X(\xi)\star_{\tau_X(\theta)})+z^{-1}(\xi\hS\partial_{\hS}\tau_X(\theta))\star_{\tau_X(\theta)}.
\end{align*}

We now construct discrete Fourier transformations on the level of quantum $D$-modules.

\begin{proposition}\label{prop:DiscreteFTDmodule}
For the highest/lowest GIT quotient $X$ and $\tau_X(\theta)\in H^*(X)\formal{C_{X,\bbN}^{\vee},\theta}$ given in Corollary~\ref{cor:MirrorMapDiscrete}, there is a $\bbC[z]\formal{Q_W,\theta}\big[\hS,Q_W^{-1}\big]$-module homomorphism
$$
    \FT_{X}\colon \QDM_{\bbC^*} (W)[Q_W^{-1}]\to \tau_{X}^*\QDM(X)^{\ext}[Q_W^{-1}]
$$
satisfying the following properties:
\begin{enumerate}[wide]
    \item \label{item:DiscreteFTDmoduleTarget}$\FT_{X}(\QDM_{\bbC^*}(W)) \subset \tau_{X}^* \QDM(X)^{\ext}$.

    \item \label{item:DiscreteFTDmoduleConnection} $\FT_{X}$ intertwines the connection $\nabla+\frac{1}{2}\frac{dz}{z}$ with $\nabla$. More precisely, $\FT_{X}$ intertwines with $\nabla_{\theta^{i,k}}$, $z \nabla_{\xi Q\partial_{Q}}$ with $z \nabla_{\xi \hS \partial_{\hS}}$, and $\nabla_{z\partial_z}+\frac{1}{2}$ with $\nabla_{z\partial_z}$.

    \item \label{item:DiscreteFTDmoduleHomo}$\FT_{X}$ is a homogeneous $\bbC[z]\formal{Q_W,\theta}\big[\hS,Q_W^{-1}\big]$-module homomorphism.

    \item \label{item:DiscreteFTDmoduleCommute} $\FT_{X}$ satisfies the commutative diagram
    $$
        \begin{tikzcd}
            \QDM_{\bbC^*}(W)[Q_W^{-1}] 
            \arrow[r, "\FT_{X}"]
            \arrow[d, "M_W(\theta)"'] 
            & \tau_{X}^* \QDM(X)^{\ext}[Q_W^{-1}] 
            \arrow[d, "M_X(\tau(\theta))"] \\
            \cH_W^{\mathrm{rat}}[Q_W^{-1}] 
            \arrow[r, dotted, "\sfF_X"] 
            & \cH_X^{\ext}[Q_W^{-1}].
        \end{tikzcd}
    $$
    where note that $\sfF_X\circ M_W(\theta)$ is well-defined due to Theorem~\ref{thm:Reduction}.
    
\end{enumerate}
\end{proposition}

\begin{proof}
By differentiating the equation $\sfF_X(J_W(\theta))=M_X(\tau_X(\theta))\vartheta_X(\theta)$ in Corollary~\ref{cor:MirrorMapDiscrete}, we define $\FT_X$ to be the unique $\bbC[z]\formal{Q_W,\theta}[Q_W^{-1}]$-module homomorphism that sends $\phi_i\lambda^k$ 
to $z\nabla_{\theta^{i,k}}\vartheta(\theta)$ for all $i,k$. 
Then the commutative diagram~\eqref{item:DiscreteFTDmoduleCommute} naturally holds. Proposition~\ref{prop:DiscreteProp}\eqref{prop:prop1 of Discrete FT on cone level} and the relation $M_W(\theta)\circ \hS^{\beta}=\cS^{\beta} \circ M_W(\theta)$ together imply that $\FT_X$ is a $\bbC[z]\formal{Q_W,\theta}\big[\hS,Q_W^{-1}\big]$-module homomorphism.
Property~\eqref{item:DiscreteFTDmoduleTarget} holds because $\{\phi_i\lambda^k\}_{i,k}$ is also a $\bbC[z]\formal{Q_W,\theta}$-basis of $\QDM_{\bbC^*}(W)$ and thus we can define $\FT_X$ without localizing $Q_W$. Property~\eqref{item:DiscreteFTDmoduleConnection} follows from the commutative relations~\eqref{prop:prop2 of Discrete FT on cone level}\eqref{prop:prop3 of Discrete FT on cone level}\eqref{prop:prop4 of Dicrete FT on cone level} in Proposition~\ref{prop:DiscreteProp}.
Property~\eqref{item:DiscreteFTDmoduleHomo} follows from Proposition~\ref{prop:DiscreteProp}\eqref{prop:prop5 of Discrete FT on cone level}.
\end{proof}


\section{Decomposition of quantum $D$-modules for 3-component $\bbC^*$-VGIT wall-crossings}\label{sect:Decomposition}
In this section, we establish the decomposition of quantum $D$-modules in the case of 3-component $\bbC^*$-VGIT wall-crossings, under the setup of Assumption~\ref{3-component assumption}. 
We also assume that $c_{F_0} < 0$, i.e.\ $r_{F_0, +} < r_{F_0, -}$.
Specifically, we decompose the quantum $D$-module of $X_-$ into the quantum $D$-module of $X_+$ and $-c_{F_0}$ copies of the quantum $D$-module of $F_0$, after pulling back to a common base ring.
The main result of the section, Theorem~\ref{thm:QDMDecompRed}, is stated over the base ring
\[
    R \coloneqq \bbC[z]\Biglaurent{Q_{X_-}^{-\frac{a}{2c_{F_0}}}}\formal{Q_{X_-},\tau_{X_-}}.
\]
We first establish the decomposition over a larger base ring
\[
    \tR \coloneqq \bbC[z]\Biglaurent{S_{F_0}^{-\frac{1}{2c_{F_0}}}}\formal{Q_W,\theta}
\]
where the continuous Fourier transformations from $W$ to $F_0$ and the discrete Fourier transformations from $W$ to $X_\pm$ are naturally defined, see Theorem~\ref{thm:QDMDecompExt}.
Then we reduce the decomposition to the smaller base ring $R$ which depends only on the quotients $X_\pm$ and not on $W$.

\subsection{Completion of quantum $D$-module of $W$}\label{subsec:Completion}

We start by extending the base of the equivariant quantum $D$-module of $W$ via a completion along the dual ample cone of a GIT quotient. 


\begin{definition}\label{def:completion of qdm}
Let $X = X_+$ or $X_-$. Consider
$$
    \QDM_{\bbC^*}(W)_{X} \coloneqq \bbC[C_{X,\bbN}^{\vee}]\cdot\QDM_{\bbC^*}(W)\subset \QDM_{\bbC^*}(W)[Q_W^{-1}].
$$
Let $\fm_X\subset \bbC[C_{X,\bbN}^{\vee}][z]\formal{Q_W,\theta}$ be the ideal generated by all nonzero elements of the monoid $C_{X,\bbN}^{\vee}$. We define the \emph{completion}
$$
    \QDM_{\bbC^*}(W)_X^{\wedge} \coloneqq \varprojlim_k
 \QDM_{\bbC^*}(W)_X/\fm_X^k\QDM_{\bbC^*}(W)_X.
$$
\end{definition}


    
Here, the inverse limit is taken in the category of graded modules. The completion endows $\QDM_{\bbC^*}(W)_X^{\wedge}$ with the structure of a $\bbC[z]\formal{C_{X,\bbN}^{\vee},\theta}$-module. Moreover, the operators $z\nabla_{z\partial_z}$, $z\nabla_{\theta^{i,k}}$, and $z\nabla_{\xi Q\partial_Q}$ can be naturally extended to this module because of the commutative relations 
$$
    \big[\nabla_{z\partial_z},\hS^{\beta}\big] =\big[\nabla_{\theta^{i,k}},\hS^{\beta}\big]=0, \quad
    \big[\nabla_{\xi Q\partial_Q},\hS^{\beta}\big]= (\xi\cdot\beta)\hS^{\beta} \quad
    \text{for any } \beta\in H^2_{\bbC^*}(W), \xi\in H_2^{\bbC^*}(W).
$$
We show below that $\QDM_{\bbC^*}(W)_X^{\wedge}$ is a finite-rank free $\bbC[z]\formal{C_{X,\bbN}^{\vee},\theta}$-module, following \cite[Theorem 5.2]{iritani2023quantum}. As a consequence, the extended operator $\nabla$ is a flat connection on $\QDM_{\bbC^*}(W)^{\wedge}_X$.



\begin{proposition}\label{prop:CompletionFree}
Let $X = X_+$ or $X_-$. The completion $\QDM_{\bbC^*}(W)_X^{\wedge}$ is a free $\bbC[z]\formal{C_{X,\bbN}^{\vee},\theta}$-module of rank equal to $r\coloneqq \dim H^*(X)$. Specifically, if $s_1,\cdots,s_{r} \in H^*_{\bbC^*}(W)$ are homogeneous elements such that $\{\kappa_X(s_1),\cdots,\kappa_X(s_{r})\}$ is a $\bbC$-basis of $H^*(X)$, then $s_1,\cdots,s_{r}$ freely generate $\QDM_{\bbC^*}(W)^{\wedge}_X$ over $\bbC[z]\formal{C_{X,\bbN}^{\vee},\theta}$.
\end{proposition}

\begin{proof}
We prove the proposition for $X = X_-$; the case $X = X_+$ is similar.
First, we prove that $s_1,\cdots,s_{r}$ generate $\QDM_{\bbC^*}(W)_X^{\wedge}$. 
We denote $\cM \coloneqq \QDM_{\bbC^*}(W)_{X}$ 
and $\cM_0 \coloneqq \QDM_{\bbC^*}(W)$. Notice that $\cM_0$ consists of $Q_W$-regular elements in $\cM$.
For $f \in \cM_0$, we define the \emph{leading term} 
\begin{equation}\label{eqn:LTDef}
    \LT(f) \coloneqq f \big|_{Q_W=0,\theta=0}\in H_{\bbC^*}^*(W)[z]. 
\end{equation}

Observe that
$$
    \cM/\fm_X\cM= \frac{\cM_0}{S\cM_0+\cM_0\cap Q_W^{a}S^{-1}\bbC[Q_W^aS^{-1}]\cM_0+  \sum_{\beta\in \NE_{\bbN}(W)/\{0\}} Q_W^{\beta}\cM_0}.
$$
Let $f \in \cM_0$. Proceeding inductively from the lower order terms of $f$ as a power series in $Q_W$ and $\theta$, and applying Lemma~\ref{lem:LeadingTermKirwan} below at each step, we may find $A_1, \dots, A_{r} \in \bbC[z]\formal{Q_W, \theta}$ and $C_1,C_2,C_3\in H_{\bbC^*}^*(W)[z]\formal{Q_W, \theta}$ such that $Q_W^a S^{-1}C_2$ is $Q_W$-regular and 
$$
    f - \sum_{i = 1}^{r} A_i s_i = \LT(SC_1+Q_W^aS^{-1}C_2+Q_W^{a+b}S^{-1}C_3).
$$
Note that $Q_W^{a+b}S^{-1} C_3=Q_W^aS^{-1}(Q_W^b C_3)\in \cM_0\cap Q_W^a S^{-1}\bbC[Q_W^aS^{-1}]\cM_0$. We have thus shown that $s_1,\cdots,s_{r}$ generate $\cM/\fm_X\cM$ over $\bbC[z][C_{X,\bbN}^{\vee}]\formal{\theta}/\fm_X$. It follows that $s_1,\cdots,s_{r}$ generate $\displaystyle \QDM_{\bbC^*}(W)_X^{\wedge} = \varprojlim_k \cM/\fm_X^k\cM$ over $\displaystyle \bbC[z]\formal{C_{X,\bbN}^{\vee},\theta} =  \varprojlim_k \bbC[z][C_{X,\bbN}^{\vee}]\formal{\theta}/\fm_X^k$.


Next, we prove that $s_1,\cdots,s_{r}$ are linearly independent.
Let $A_i\in \bbC[z]\formal{C_{X,\bbN}^{\vee},\theta}$ such that $\sum_{i=1}^{r} A_is_i=0$. Suppose that not all $A_i$ are zero.
We express $A_i$ as
$$
    A_i= \sum_{\substack{\beta\in\NE_{\bbN}(W),\, k\in \bbZ \\ \beta+ka\in \NE_{\bbN}(W)}} \ \sum_{t = (t_{j,n})} c_{i,\beta,k, t} \theta^t Q_W^{\beta}S^{k}
$$ 
where $\theta^t$ is the shorthand notation for the monomial $\prod_{j, n} (\theta^{j,n})^{t_{j,n}}$ and $c_{i,\beta,k, t} \in \bbC[z]$ are coefficients. We choose a minimal element $\beta_0\in \NE_{\bbN}(W)$ in the nonempty set 
$$
    \{\beta\in\NE_{\bbN}(W) \ \big|\ \text{there exist } i,k,t \text{ such that } c_{i,\beta,k,t}\ne0\}.
$$
Let $k_0$ be the minimal $k \in \bbZ$ such that $\beta_0 + ka\in \NE_{\bbN}(W)$, or equivalently, $Q_W^{\beta_0}S^k \in C_{X,\bbN}^{\vee}$. Note that $k_0 \le 0$. 
Moreover, we choose a minimal element $t_0$ in the nonempty set
$$ 
    \{t = (t_{j,n}) \ \big|\ \text{there exist }i,k \text{ such that } c_{i,\beta_0,k, t} \ne 0\}.
$$ 

We consider
$$
    0 = M_W(\theta) \left(\sum_{i=1}^{r}A_is_i \right)
    = \sum_{i=1}^{r} \ \sum_{\substack{\beta\in\NE_{\bbN}(W), k\in \bbZ \\ \beta+ka\in \NE_{\bbN}(W)}} \ \sum_{t = (t_{j,n})} c_{i,\beta,k, t} \theta^t Q_W^{\beta}\cS^k(M_W(\theta)s_i).
$$
Restricting the above to $F_-$ and using the formula \eqref{eqn:ShiftOpLocal} of the shift operator, we obtain
$$
    0 =\sum_{i=1}^{r} \sum_{\substack{\beta\in\NE_{\bbN}(W), k\in \bbZ, \\ \beta+ka\in \NE_{\bbN}(W)}} \sum_{t = (t_{j,n})} c_{i,\beta,k, t} \theta^t Q_W^{\beta}\frac{\prod_{c=-\infty}^{0}e_{\lambda+cz}(N_{F_-/W})}{\prod_{c=-\infty}^{-k}e_{\lambda+cz}(N_{F_-/W})}e^{-kz\partial_{\lambda}}(M_W(\theta)s_i)|_{F_-}.
$$
Recall that $M_W(\theta)|_{\theta=0,Q_W=0}(s_i) = s_i$ for all $i$. Then, taking the coefficient of $\theta^{t_0} Q_W^{\beta_0}$ in the above, we have
$$
    0 = \sum_{i=1}^{r} \sum_{\substack{k\ge k_0}} c_{i,\beta_0, k, t_0} \frac{\prod_{c=-\infty}^{0}e_{\lambda+cz}(N_{F_-/W})}{\prod_{c=-\infty}^{-k}e_{\lambda+cz}(N_{F_-/W})}e^{-kz\partial_{\lambda}}s_i|_{F_-}.
$$
We write $c_{i,\beta_0, k, t_0} = c_{i, k}$ for short. Multiplying the above by $\prod_{c=1}^{-k_0}e_{\lambda+cz}(N_{F_-/W})$, we have
$$
    0 = \sum_{i=1}^{r}\sum_{\substack{ k\ge k_0}} c_{i, k} \prod_{c = -k+1}^{-k_0}e_{\lambda+cz}(N_{F_-/W})e^{-kz\partial_{\lambda}}(s_i|_{F_-}).
$$

Taking $z = 0$, we obtain the relation
$$
    0 =  \sum_{\substack{ k\ge k_0}} \sum_{i=1}^{r} c_{i,k}|_{z=0} e_{\lambda}(N_{F_-/W})^{k-k_0} s_i|_{F_-}
$$
in $H_{\bbC^*}^*(F_-)$. Now, recall the factorization \eqref{eqn:KirwanFactorize} of the Kirwan map $\kappa_X$. Since $\{\kappa_X(s_i)\}_{i=1}^{r}$ is a $\bbC$-basis of $H^*(X)=H^*_{\bbC^*}(F_-)/e_{\lambda}(N_{F_-/W})H_{\bbC^*}^*(F_-)$ and $e_{\lambda}(N_{F_-/W})$ is a monic polynomial in $\lambda$, we see that
$$
    \{e_{\lambda}(N_{F_-/W})^{k-k_0}s_i|_{F_-}\}_{i=1}^{r}
$$ 
is a $\bbC$-basis of the quotient $e_{\lambda}(N_{F_-/W})^{k-k_0}H_{\bbC^*}^*(F_-)/e_{\lambda}(N_{F_-/W})^{k+1-k_0}H_{\bbC^*}^*(F_-)$ for any $k \ge k_0$. It follows that $c_{i,k}|_{z=0}=0$ for all $i, k$. Proceeding in a similar way in increasing degrees of $z$, we see eventually that $c_{i,k}=0$ for all $i, k$, which is a contradiction.
\end{proof}

\begin{lemma}\label{lem:LeadingTermKirwan}
Let $X = X_+$ or $X_-$. For any $s \in H_{\bbC^*}^*(W)[z]$ such that $\kappa_X(s) = 0$, there exist $\alpha_1,\alpha_2,\alpha_3\in H_{\bbC^*}^*(W)[z]$ such that $Q_W^aS^{-1}\alpha_2 \in \QDM_{\bbC^*}(W)$ 
and $s=\LT(S\alpha_1+Q_W^aS^{-1}\alpha_2+Q_W^{a+b}S^{-1}\alpha_3)$, where $\LT$ is defined in \eqref{eqn:LTDef}.
\end{lemma}

\begin{proof}
We prove the lemma for $X = X_-$; the case $X = X_+$ is similar.
In preparation, we characterize the actions of the operators $S$, $Q_W^{a}S^{-1}$, and $Q_W^{a+b}S^{-1}$ on an element $\alpha \in H^*_{\bbC^*}(W)[z]$. Note from the definitions that $S\alpha, Q_W^{a+b}S^{-1}\alpha \in \QDM_{\bbC^*}(W)$. For $S\alpha$, we have
$$
    \LT(S\alpha)=M_W(\theta)(S\alpha)|_{Q_W=0,\theta=0} = \cS (M_W(\theta)\alpha)|_{Q_W=0,\theta=0}= \cS\alpha|_{Q_W=0,\theta=0},
$$
The formula \eqref{eqn:ShiftOpLocal} of the shift operator implies that
$$
    \LT(S\alpha)|_{F_+}=\LT(S\alpha)|_{F_0}=0, \qquad \LT(S\alpha)|_{F_-}=e_{\lambda}(N_{F_-/W})(e^{-z\partial_{\lambda}}\alpha|_{F_-}).
$$
Thus, we have
$$
    \LT(S\alpha)=i_{F_-,*}(e^{-z\partial_{\lambda}}\alpha|_{F_-}).
$$
By a similar computation, we have
$$
    \LT(Q_W^{a+b}S^{-1}\alpha)=i_{F_+,*}(e^{z\partial_{\lambda}}\alpha|_{F_+}).
$$
We now turn to $Q_W^{a}S^{-1}\alpha$ which in general does not necessarily lie in $\QDM_{\bbC^*}(W)$. We show that when $\alpha|_{F_+}=0$, we have $Q_W^{a}S^{-1}\alpha \in \QDM_{\bbC^*}(W)$.
Consider
$$
    M_W(\theta)(Q_W^aS^{-1}\alpha) = Q_W^a\cS^{-1}(M_W(\theta)\alpha).
$$
By virtual localization, the condition $\alpha|_{F_+}=0$ implies that in $(M_W(\theta)\alpha)|_{F_+}$, the Novikov variables which have nonzero contribution must correspond to $\bbC^*$-invariant curves containing an irreducible component connecting $F_+$ to $F_0$ or $F_-$, and are thus divisible by $Q_W^b$ or $Q_W^{a+b}$.
Then, the $Q_W$-regularity of $Q_W^{a+b}\cS^{-1}$ implies the $Q_W$-regularity of $Q_W^a\cS^{-1}(M_W(\theta)\alpha)|_{F_+}$, which further implies the $Q_W$-regularity of $Q_W^aS^{-1}\alpha|_{F_+}$ as $M_W(\theta)|_{Q_W=0,\theta=0}=\Id$. Since $Q_W^aS^{-1}\alpha|_{F_0}$ and $Q_W^aS^{-1}\alpha|_{F_-}$ are also $Q_W$-regular, we conclude that $Q_W^{a}S^{-1}\alpha$ is $Q_W$-regular. Note that we have the leading terms
$$
    \LT\big(Q_W^aS^{-1}\alpha\big)\big|_{F_-} = \LT\big(Q_W^a\cS^{-1}(M_W(\theta)\alpha)\big)\big|_{F_-} = \LT\bigg(Q_W^a\frac{1}{e_{\lambda+z}(N_{F_-/W})}e^{z\partial_{\lambda}}\alpha\bigg)\bigg|_{F_-}=0,
$$ $$
    \LT(Q_W^aS^{-1}\alpha)|_{F_0} = \LT(Q_W^a\cS^{-1}M_W(\theta)\alpha)|_{F_0} = \LT\bigg(\frac{e_{-\lambda}(N_{F_0,-})}{e_{\lambda+z}(N_{F_0,+})}e^{z\partial_{\lambda}}\alpha\bigg)\bigg|_{F_0} = \frac{e_{-\lambda}(N_{F_0,-})}{e_{\lambda+z}(N_{F_0,+})}e^{z\partial_{\lambda}}\alpha|_{F_0}.
$$

Now we prove the statement in the lemma. Let $s_- \coloneqq s|_{F_-}$ and $s_0 \coloneqq s|_{F_0}$. Since $\kappa_X(s) = 0$, by \eqref{eqn:KirwanFactorize}, there exists $t_- \in H^*_{\bbC^*}(F_-)$ such that $s_- = e_{\lambda}(N_{F_-/W}) t_-$. 
We take $\alpha_1 \in H_{\bbC^*}^*(W)[z]$ with
$
    \alpha_1 |_{F_-} = e^{z\partial_{\lambda}}t_-,
$
so that
$$
    \LT(S\alpha_1) = i_{F_-,*}(t_-), \qquad \LT(S\alpha_1)|_{F_-} = (i_{F_-,*}(t_-)) |_{F_-} = s_-.
$$
Now, by \eqref{3-component Gysin 3}, there exists $t_0 \in H_{\bbC^*}^*(F_0)[z]$ such that $(s-\LT(S\alpha_1))|_{W\backslash F_+} = i_{F_0,*}^+(t_0)$. Note that $s_0 = (s-\LT(S\alpha_1))|_{F_0} = i_{F_0,*}^+(t_0)|_{F_0} = e_{-\lambda}(N_{F_0,-})t_0$. By \eqref{3-component Gysin 6}, we may take $\alpha_2 \in H_{\bbC^*}^*(W)[z]$ such that 
$$
    \alpha_2|_{W\backslash F_-} = i_{F_0,*}^-(e^{-z\partial_{\lambda}}t_0).
$$
This implies that $\alpha_2|_{F_+} = (\alpha_2|_{W\backslash F_-})|_{F_+} = 0$
. It follows from our discussion above that $Q_W^aS^{-1}\alpha_2$ is $Q_W$-regular with the leading terms $\LT(Q_W^aS^{-1}\alpha_2)|_{F_-}=0$ and
$$
    \LT(Q_W^aS^{-1}\alpha_2)|_{F_0} = \frac{e_{-\lambda}(N_{F_0,-})}{e_{\lambda+z}(N_{F_0,+})}e^{z\partial_{\lambda}} i_{F_0,*}^-(e^{-z\partial_{\lambda}}t_0) |_{F_0} = e_{-\lambda}(N_{F_0,-})t_0 =  s_0.
$$
Therefore, we have 
$$
    \big(s-\LT(S\alpha_1 + Q_W^aS^{-1}\alpha_2)\big)|_{F_-} = \big(s-\LT(S\alpha_1 + Q_W^aS^{-1}\alpha_2)\big)|_{F_0} = 0.
$$
Finally, by \eqref{3-component Gysin 3} and \eqref{3-component Gysin 4}, there exists $t_+ \in H^*_{\bbC^*}(F_+)$ such that $s-\LT(S\alpha_1 + Q_W^aS^{-1}\alpha_2) = i_{F_+,*}(t_+)$. 
We then take $\alpha_3 \in H_{\bbC^*}^*(W)[z]$ with
$
    \alpha_3 |_{F_+} = e^{-z\partial_{\lambda}}t_+,
$
so that
$$
    \LT(Q_W^{a+b}S^{-1}\alpha_3) = i_{F_+,*}(t_+).
$$
\end{proof}

\subsection{QDM decomposition over extended bases}\label{subsec:QDMFlipExtended}
In this subsection, we prove the QDM decomposition theorem over the extended base 
$$
    \tR = \bbC[z]\Biglaurent{S_{F_0}^{-\frac{1}{2c_{F_0}}}}\formal{Q_W,\theta}.
$$
We first observe that the natural inclusions
$$
    \bbC\formal{C_{X_\pm,\bbN}^{\vee}} \subseteq \bbC\Biglaurent{S_{F_0}^{-\frac{1}{2c_{F_0}}}}\formal{Q_W}
$$
induce the inclusions
$$
    \bbC[z]\formal{C_{X_\pm,\bbN}^{\vee},\theta} \subseteq \bbC[z]\Biglaurent{S_{F_0}^{-\frac{1}{2c_{F_0}}}}\formal{Q_W,\theta} = \tR.
$$
To see this, for $X_-$, recall from \eqref{eqn:DualCone3Comp} that $\bbC\formal{C_{X_-,\bbN}^{\vee}}$ is generated by $Q_W^{\delta}$ for $\delta \in \NE_{\bbN}(W)$, $Q_W^aS^{-1}=S_{F_0}^{-1}$, and $S=Q_W^a S_{F_0}$. Given any homogeneous element $f \in \bbC\formal{C_{X,\bbN}^{\vee}}$, written in the form
$$
    f = \sum_{k_1,k_2,\delta} f_{k_1,k_2,\delta} Q_W^{\delta}(Q_W^aS_{F_0})^{k_1}(S_{F_0}^{-1})^{k_2} = \sum_{k_1,k_2,\delta} f_{k_1,k_2,\delta} S_{F_0}^{k_1-k_2} Q_W^{\delta+k_1a},
$$
since $S_{F_0}=Q_W^{-a}S$ has degree $2c_{F_0} < 0$, the power $k_1 - k_2$ is bounded below in terms of the fixed degree of $Q_W^{\delta+k_1a}$. Therefore, $f$ is contained in $\bbC\Biglaurent{S_{F_0}^{-\frac{1}{2c_{F_0}}}}\formal{Q_W}$. The argument for $X_+$ is similar.

In Section~\ref{subsec:Completion}, we have completed $\QDM_{\bbC^*}(W)$ into a free $\bbC[z]\formal{C_{X_-,\bbN}^{\vee},\theta}$-module $\QDM_{\bbC^*}(W)^{\wedge}_{X_-}$, which can be further extended to a free $\tR$-module
$$
    \QDM_{\bbC^*}(W)^{\wedge,\La}_{X_-} \coloneqq \QDM_{\bbC^*}(W)^{\wedge}_{X_-} \otimes_{\bbC[z]\formal{C_{X_\pm,\bbN}^{\vee},\theta}} \bbC[z]\Biglaurent{S_{F_0}^{-\frac{1}{2c_{F_0}}}}\formal{Q_W,\theta}.
$$
Recall from Section~\ref{sect:ContinuousQDM} that for $j = 0, \dots, \abs{c_{F_0}}-1$, we have the continuous Fourier transformation $\FT_{F_0,j}$ from $\QDM_{\bbC^*}(W)$ to the $\tR$-module $\zeta_j^*\QDM(F_0)^{\La}$. By the argument of \cite[Propositions 5.4, 5.7]{iritani2023quantum},
they can be naturally extended to $\tR$-module homomorphisms
$$
    \FT_{F_0,j}\colon \QDM_{\bbC^*}(W)^{\wedge,\La}_{X_-} \longrightarrow \zeta_j^*\QDM(F_0)^{\La}, \qquad j = 0, \dots, \abs{c_{F_0}}-1
$$
which satisfy the proposition in Corollary~\ref{Corollary of continuous FT(shift on the pullback)} such that the homomorphism intertwines $\nabla+\frac{dz}{2z}$ with $\nabla$.



Similarly, recall from Section~\ref{sect:DiscreteQDM} that for $X_\pm$, we have the discrete Fourier transformation $\FT_{X_\pm}$ from $\QDM_{\bbC^*}(W)$ to the $\bbC[z]\formal{C_{X_\pm,\bbN}^{\vee},\theta}$-module $\tau_{X_\pm}^*\QDM(X_\pm)^{\ext}$. We define the extensions
$$
    \QDM(X_\pm)^{\La} \coloneqq \QDM(X_\pm)^{\ext} \otimes_{\bbC[z]\formal{C_{X_\pm,\bbN}^{\vee},\tau_{X_\pm}}} \bbC[z]\Biglaurent{S_{F_0}^{-\frac{1}{2c_{F_0}}}}\formal{Q_W,\tau_{X_\pm}}, 
$$
$$
    \tau_{X_{\pm}}^*\QDM(X_{\pm})^{\La} \coloneqq \tau_{X_\pm}^*\QDM(X_\pm)^{\ext} \otimes_{\bbC[z]\formal{C_{X_\pm,\bbN}^{\vee},\theta}} \bbC[z]\Biglaurent{S_{F_0}^{-\frac{1}{2c_{F_0}}}}\formal{Q_W,\theta}.
$$
The latter is a free $\tR$-module equipped with induced flat connection.
By Proposition~\ref{prop:DiscreteFTDmodule}, and additionally the argument of \cite[Proposition 5.4]{iritani2023quantum} for $X_+$,
the discrete Fourier transformation can be naturally extended to an $\tR$-module homomorphism
$$
    \FT_{X_\pm}\colon\QDM_{\bbC^*}(W)^{\wedge,\La}_{X_-} \longrightarrow \tau_{X_\pm}^*\QDM(X_\pm)^{\La}
$$
which satisfies Proposition~\ref{prop:DiscreteFTDmodule}\eqref{item:ContFTDmoduleConnection}\eqref{item:ContFTDmoduleHomo}.

Now that we have extended the Fourier transformations as $\tR$-module homomorphisms, we are ready to state the decomposition theorem over this extended base ring.


\begin{theorem}[QDM decomposition over extended base]\label{thm:QDMDecompExt}
The $\tR$-module homomorphisms
$$
    \FT_{X_-}\colon\QDM_{\bbC^*}(W)^{\wedge,\La}_{X_-}\longrightarrow\tau_{X_-}^*\QDM(X_-)^{\La}
$$ 
and
$$
    \Psi_+ \coloneqq \FT_{X_+}\oplus \bigoplus_{j=0}^{\abs{c_{F_0}}-1}\FT_{F_0,j}\colon\QDM_{\bbC^*}(W)^{\wedge,\La}_{X_-}\longrightarrow \tau_{X_+}^*\QDM(X_+)^{\La}\oplus \bigoplus_{j=0}^{\abs{c_{F_0}}-1}\zeta_j^*\QDM(F_0)^{\La}
$$
are isomorphisms. Moreover, the composition
$$
    \Phi \coloneqq \Psi_+ \circ \FT_{X_-}^{-1}\colon\tau_{X_-}^*\QDM(X_-)^{\La}\longrightarrow\tau_{X_+}^*\QDM(X_+)^{\La}\oplus \bigoplus_{j=0}^{\abs{c_{F_0}}-1}\zeta_j^*\QDM(F_0)^{\La}
$$
intertwines the quantum connections and intertwines the pairing $P_{X_-}$ with $P_{X_+}\oplus \bigoplus_{j=0}^{\abs{c_{F_0}}-1}P_{F_0}$.
\end{theorem}

We prove Theorem~\ref{thm:QDMDecompExt} in the rest of the subsection. As discussed above, the $\tR$-module homomorphisms $\FT_{X_-}$ and $\Psi_+$ intertwine the quantum connections. We show in Proposition~\ref{prop:tRModIso} below that they are $\tR$-module isomorphisms. It follows that the composition $\Phi$ is defined, is an isomorphism, and intertwines the quantum connections. Furthermore, we show in Proposition~\ref{prop:IntertwinePairings} below that $\Phi$ intertwines the pairings.

\subsubsection{Reference bases}\label{sect:RefBases}
For the proof and for later use, we fix reference bases for the cohomology groups. 
Let $\{\alpha_{X_+,0},\cdots,\alpha_{X_+,n_+-1}\}$ be a homogeneous basis of $H^*(X_+)$ and $\{\alpha_{F_0,0},\cdots,\alpha_{F_0,n_0-1}\}$ be a homogeneous basis of $H^*(F_0)$, where we denote $\mathrm{dim}_{\bbC} (H^*(X_+))$ and $\mathrm{dim}_{\bbC}(H^*(F_0))$ by $n_+$ and $n_0$ respectively. For $0 \le j \le \abs{c_{F_0}} -1 $, let $\alpha_{F_0,i,j}$ represent the copy of $\alpha_{F_0,i}$ in 
$\zeta_j^*\QDM(F_0)^{\La}$. 
Under the isomorphism in Proposition~\ref{prop:deRhamIso3Comp}, the collection 
\begin{equation}\label{eqn:XMinusBasis}
    \{\phi_2(\alpha_{X_+,m})\}_{0\le m\le n_+-1} \cup \{j_{-,*}(h_-^l  \psi_-^*(\alpha_{F_0,i}))\}_{ 0\le l\le \abs{c_{F_0}}-1, 0\le i\le n_0-1}
\end{equation}
is a homogeneous basis of $H^*(X_-)$. By Lemma~\ref{lem:KirwanPhi2}, for $0 \le m \le n_+-1$, there exists (uniquely) $s_m \in H_{\bbC^*}^*(W)$ such that $\kappa_{X_+}(s_m)= \alpha_{X_+,m}$, $\kappa_{X_-}(s_m) = \phi_2(\alpha_{X_+,m})$, and $\deg_{\lambda}(s_m|_{F_0})\le r_{F_0,+}-1$. Moreover, due to the surjectivity of $i_{W\backslash F_-}^*$ \eqref{3-component Gysin 6}, for $0\le l\le \abs{c_{F_0}}-1$ and $0\le i \le n_0-1$, there exists $s_{l,i}\in H^*_{\bbC^*}(W)$ such that $s_{l,i}|_{W\backslash F_-}=i_{F_0,*}^-(\lambda^l\alpha_{F_0,i})$. Then, $\kappa_{X_+}(s_{l,i})=0$ by Lemma~\ref{lem:KirwanKernel} and $\kappa_{X_-}(s_{l,i})=j_{-,*}(h^l \psi_-^*(\alpha_{F_0,i}))$ by the diagram \eqref{Cohomology excess formula for iF0-}. The list \eqref{eqn:XMinusBasis} is rewritten as
$$
    \{\kappa_{X_-}(s_m)\}_{0\le m\le n_+-1} \cup \{\kappa_{X_-}(s_{l,i})\}_{0\le l\le \abs{c_{F_0}}-1, 0\le i\le n_0-1},
$$
which implies that
$$
    \{s_m\}_{0\le m\le n_+-1} \cup \{s_{l,i}\}_{0\le l\le \abs{c_{F_0}}-1, 0\le i\le n_0-1}
$$
is a collection of linearly independent homogeneous elements of $H_{\bbC^*}^*(W)$.

\color{black}

\subsubsection{Isomorphisms as $\tR$-modules}\label{sect:tRModIso}
We start by calculating the leading terms of the continuous and discrete Fourier transformations. In the rest of this section, by taking the limit $Q_W = 0$ on a series, we consider the series in the variables $Q_W$ and $S_{F_0}^{\mp\frac{1}{2c_{F_0}}}$.

\begin{lemma}
\label{lem:LTContFT}
Let $j = 0, \dots, \abs{c_{F_0}}-1$. Given $\alpha \in H_{\bbC^*}^*(W)$, and expanding $\alpha |_{F_0} \in H_{\bbC^*}^*(F_0) = H^*(F_0)[\lambda]$ as $\alpha |_{F_0} = \alpha_n \lambda^n + O(\lambda^{n-1})$ with $\alpha_n \neq 0$, we have
$$
    \FT_{F_0,j}(\alpha)|_{Q_W=0,\theta=0} = q_{F_0,j}\lambda_j^{n} \left(\alpha_n + \sum_{k>0}f_{j,k}(\alpha)S_{F_0}^{-\frac{k}{c_{F_0}}} \right)
$$
for some $f_{j,k}(\alpha)\in H^*(F_0)[z]$, $k>0$. Here, $\lambda_j$ and $q_{F_0,j}$ are defined in \eqref{eqn:Lambdaj} and \eqref{eqn:qFj} respectively.
\end{lemma}

\begin{proof}
By Proposition~\ref{prop:ContFTDmodule}\eqref{item:ContFTDmoduleCommute}, we have 
$$
    M_{F_0}(\sigma_j(\theta))\FT_{F_0,j}(\alpha) = S_{F_0}^{\frac{\rho_{F_0}}{zc_{F_0}}}\mathscr{F}_{F_0,j}(M_W(\theta)c).
$$
We take the restriction $Q_W=\theta=0$ on both sides. For the left-hand side, Proposition~\ref{prop:ContFTCone} implies that $M_{F_0}(\sigma_j(\theta))|_{Q_W=\theta=0} = e^{\frac{h_{F_0,j}}{z}}\Big(1+O\big(S_{F_0}^{-1/c_{F_0}}\big)\Big)$. For the right-hand side, by \eqref{eqn:ContFTLeading}, we have that
$$
    S_{F_0}^{\frac{\rho_{F_0}}{zc_{F_0}}}\mathscr{F}_{F_0,j}(M_W(\theta)\alpha)|_{Q_W=0,\theta=0} = S_{F_0}^{\frac{\rho_{F_0}}{zc_{F_0}}}\mathscr{F}_{F_0,j}(\alpha)|_{Q_W=0,\theta=0} = q_{F_0,j}e^{\frac{h_{F_0,j}}{z}}\lambda_j^n\Big(\alpha_n + O\big(S_{F_0}^{-1/c_{F_0}}\big)\Big).
$$
The lemma then follows.
\end{proof}

\begin{lemma}
\label{lem:LTDiscreteFT}
Given $\alpha \in H_{\bbC^*}^*(W)$, we have
$$
    \FT_{X_+}(\alpha)|_{Q_W=0,\theta=0}=\kappa_{X_+}(\alpha)+\sum_{k>0}f_{+,k}(\alpha)S_{F_0}^k,
$$
$$
    \FT_{X_-}(\alpha)|_{Q_W=0,\theta=0}=\kappa_{X_-}(\alpha)+\sum_{k>0}f_{-,k}(\alpha)S_{F_0}^{-k}
$$
for some $f_{\pm,k}(\alpha)\in H^*(X_{\pm})[z]$, $k > 0$.
\end{lemma}

\begin{proof}
We prove the lemma for $X_-$; the case $X_+$ is similar. By Proposition~\ref{prop:DiscreteFTDmodule}\eqref{item:DiscreteFTDmoduleCommute}, we have 
\begin{align*}
    M_{X_-}(\tau_{X_-}(\theta))\FT_{X_-}(\alpha) &= \sfF_{X_-}(M_W(\theta)\alpha) = \sum_{k\in\bbZ} S^{-k} \kappa_{X_-}(\cS^k M_W(\theta)\alpha) \\ &= \sum_{k\in\bbZ} S_{F_0}^{-k} Q_W^{-ka} \kappa_{X_-}(\cS^k M_W(\theta)\alpha).
\end{align*}
We have shown in Lemma~\ref{lem:DiscreteDefined} that this series is supported on $\bbC\formal{C_{X_-,\bbN}^{\vee}}$. It follows that, when $k<0$, $\kappa_{X_-}(\cS^k M_W(\theta)\alpha)$ vanishes in the limit $Q_W = 0$. On the other hand, when $k>0$, only the coefficient of $Q_W^{ka}$ in $\kappa_{X_-}(\cS^k M_W(\theta)\alpha)$ survives in the limit. We have
$$
    M_{X_-}(\tau_{X_-}(\theta))\FT_{X_-}(\alpha)|_{Q_W=\theta=0} = \kappa_{X_-}(\alpha)+\sum_{k>0}g_{-,k}(\alpha)S_{F_0}^{-k}
$$
for some $g_{-,k}(\alpha)\in H^*(X_-)\laurent{z^{-1}}$. In addition, by Corollary~\ref{cor:MirrorMapDiscrete}, we have $\tau_{X_-}(\theta)|_{Q_W=\theta=0} = O(S_{F_0}^{-1})$ and thus $M_{X_-}(\tau_{X_-}(\theta))|_{Q_W=0,\theta=0} = \Id + O(S_{F_0}^{-1})$. The statement in the lemma then follows. Note that the definition of $\FT_{X_-}$ ensures that each $f_{-,k}(\alpha)$ does not contain negative powers of $z$.
\end{proof}

\begin{proposition}\label{prop:tRModIso}
The $\tR$-module homomorphisms $\FT_{X_-}$ and $\Psi_+ = \FT_{X_+}\oplus \bigoplus_{j=0}^{\abs{c_{F_0}}-1}\FT_{F_0,j}$ in Theorem~\ref{thm:QDMDecompExt} are isomorphisms.
\end{proposition}

\begin{proof}
Consider the bases constructed in Section~\ref{sect:RefBases}.
By Proposition~\ref{prop:CompletionFree}, the collection of elements $\{s_m, s_{l,i}\}$ is an $\tR$-basis of $\QDM_{\bbC^*}(W)_{X_-}^{\wedge,\La}$. By Lemma~\ref{lem:LTDiscreteFT}, we have for each $m$ that $\FT_{X_-}(s_m) |_{Q_W = S_{F_0}^{-1} = 0, \theta = 0} = \kappa_{X_-}(s_m)$, and similarly for each $s_{l,i}$. It follows that the image of this basis under $\FT_{X_-}$ is an $\tR$-basis of $\tau_{X_-}^*\QDM(X_-)^{\La}$.

We now consider the image of $\{s_m, s_{l,i}\}$ under $\Psi_+$, by relating it to the basis 
$$
    \{\alpha_{X_+,m}\}_{0\le m\le n_+-1} \cup \{\alpha_{F_0,i,j}\}_{0\le i\le n_0-1, 0\le j\le \abs{c_{F_0}}-1}.
$$
We consider the leading term $\Psi_+|_{Q_W=\theta=0}$.
By Lemmas~\ref{lem:LTContFT} and \ref{lem:LTDiscreteFT}, and that 
$$
    s_{l,i}|_{F_0} = e_{\lambda}(N_{F_0,+})\lambda^{l}\alpha_{F_0,i} = \lambda^{l+r_{F_0,+}}(\alpha_{F_0,i}+O(\lambda^{-1}))
$$ 
for all $l, i$, we have the relation
$$
\left[
    \begin{array}{c}
    \begin{matrix}
        \Psi_+|_{Q_W=\theta=0}(s_m)
    \end{matrix}
    \\ \hdashline
    \begin{matrix}
        \Psi_+|_{Q_W=\theta=0}(s_{l,i})
    \end{matrix}
    \end{array}
\right]
= \left[
    \begin{array}{c:c}
    \begin{matrix}
        \Id_{X_+}+O(S_{F_0})
    \end{matrix}&
    S_{F_0}^{-\frac{r_{F_0}-1}{2c_{F_0}}+\frac{r_{F_0,+}-1}{c_{F_0}}}(*+O\big(S_{F_0}^{-1/c_{F_0}}\big))
    \\ \hdashline
    \begin{matrix}
       O\big(S_{F_0}^{-1/c_{F_0}}\big)
    \end{matrix}&
    \left(
    \begin{matrix}
        q_{F_0,j}\lambda_j^{l+r_{F_0,+}}(\Id_{F_0}+O\big(S_{F_0}^{-1/c_{F_0}}\big))
    \end{matrix}
    \right)_{l,j}
    \end{array}
\right]
\left[
    \begin{array}{c}
    \begin{matrix}
        \alpha_{X_+,m}
    \end{matrix}
    \\ \hdashline
    \begin{matrix}
        \alpha_{F_0, i, j}
    \end{matrix}
    \end{array}
\right].
$$
Here, in the bottom block of the vector on the left (resp.\ on the right), the elements are sorted first according to $l$ (resp.\ $j$), and then according to $i$. Let $C$ denote the block matrix in the middle. The ``$*$'' in the top-right block of $C$ denotes some matrix that is independent of $S_{F_0}$. The bottom-right block of $C$ is divided into $\abs{c_{F_0}} \times \abs{c_{F_0}}$ minors each of size $t \times t$, and the notation $(-)_{l, j}$ stands for $(l,j)$-th minor.
Moreover, we use $\Id_{X_+}$ (resp.\ $\Id_{F_0}$) to denote the $k \times k$ (resp.\ $t \times t$) identity matrix on 
$H^*(X)$ (resp.\ $H^*(F_0)$). 

Now we show that $C$ is invertible, as this would imply that the image of $\{s_m, s_{l,i}\}$ under $\Psi_+$ is an $\tR$-basis of $\tau_{X_+}^*\QDM(X_+)^{\La} \oplus \bigoplus_{j=0}^{\abs{c_{F_0}}-1} \zeta_j^*\QDM(F_0)^{\La}$.
We write $C = A + B$, where
$$
    A = \left[
        \begin{array}{c:c}
        \begin{matrix}
            \Id_{X_+}
        \end{matrix}&
        S_{F_0}^{-\frac{r_{F_0}-1}{2c_{F_0}}+\frac{r_{F_0,+}-1}{c_{F_0}}}(*)
        \\ \hdashline
        \begin{matrix}
           0
        \end{matrix}&
        \left(
        \begin{matrix}
            q_{F_0,j}\lambda_j^{l+r_{F_0,+}}(\Id_{F_0})
        \end{matrix}
        \right)_{l,j}
        \end{array}
    \right]
    =
    \left[
        \begin{array}{c:c}
        \begin{matrix}
            \Id_{X_+}
        \end{matrix}&
        S_{F_0}^{-\frac{r_{F_0}-1}{2c_{F_0}}+\frac{r_{F_0,+}-1}{c_{F_0}}}(*)
        \\ \hdashline
        \begin{matrix}
           0
        \end{matrix}&
        \left(
        \begin{matrix}
            q_{F_0,j}\lambda_0^{l+r_{F_0,+}}\zeta^{j(l+r_{F_0,+})}(\Id_{F_0})
        \end{matrix}
        \right)_{l,j}
        \end{array}
    \right],
$$
and 
$$
    B=\left[
        \begin{array}{c:c}
        \begin{matrix}
            O(S_{F_0})
        \end{matrix}&
        S_{F_0}^{-\frac{r_{F_0}-1}{2c_{F_0}}+\frac{r_{F_0,+}-1}{c_{F_0}}}O\big(S_{F_0}^{-1/c_{F_0}}\big)
        \\ \hdashline
        \begin{matrix}
           O\big(S_{F_0}^{-1/c_{F_0}}\big)
        \end{matrix}&
        \left(
        \begin{matrix}
            q_{F_0,j}\lambda_j^{l+r_{F_0,+}}O\big(S_{F_0}^{-1/c_{F_0}}\big)
        \end{matrix}
        \right)_{l,j}
        \end{array}
    \right],
$$
with $\zeta \coloneqq e^{\frac{2\pi i}{c_{F_0}}}$. Note that $A$ is invertible as a matrix with entries in $\bbC[z]\Biglaurent{S_{F_0}^{-\frac{1}{2c_{F_0}}}}$, because the matrix $(\zeta^{j(l+r_{F_0,+})})_{l,j}$ is invertible. Using that $(\zeta^{j(l+r_{F_0,+})})_{l,j}^{-1}=(\zeta^{-l(j+r_{F_0,+})})_{l,j}$, we may compute that
$$
    A^{-1} = \left[
        \begin{array}{c:c}
        \begin{matrix}
            \Id_{X_+}
        \end{matrix}&
        -S_{F_0}^{-\frac{r_{F_0}-1}{2c_{F_0}}+\frac{r_{F_0,+}-1}{c_{F_0}}}(*)
        \cdot (\lambda_0^{-j-r_{F_0,+}}(q_{F_0,l}\abs{c_{F_0}})^{-1}\zeta^{-l(j+r_{F_0,+})}\Id_{F_0})_{l,j}
        \\ \hdashline
        \begin{matrix}
           0
        \end{matrix}&
        \left(
        \begin{matrix}
            \lambda_0^{-j-r_{F_0,+}}(q_{F_0,l}\abs{c_{F_0}})^{-1}\zeta^{-l(j+r_{F_0,+})}(\Id_{F_0})
        \end{matrix}
        \right)_{l,j}
        \end{array}
    \right]. 
$$ 
Moreover, recall that $\lambda_0\in \bbC S_{F_0}^{\frac{1}{c_{F_0}}}$ and $q_{F_0,j}\in\bbC S_{F_0}^{-\frac{r_{F_0}-1}{2c_{F_0}}}$. We see that $BA^{-1}$ only contains positive powers of $S_{F_0}^{-\frac{1}{2c_{F_0}}}$, because the negative powers of $S_{F_0}^{-\frac{1}{2c_{F_0}}}$ in the bottom-right block of $A^{-1}$ are eliminated by the higher positive powers of $S_{F_0}^{-\frac{1}{2c_{F_0}}}$ in the bottom-right and top-right blocks of $B$, and the other blocks of $A^{-1}$ do not contain negative powers of $S_{F_0}^{-\frac{1}{2c_{F_0}}}$. Hence, the sequence $A^{-1}+A^{-1}BA^{-1}+A^{-1}BA^{-1}BA^{-1+}\cdots$ converges to a matrix with entries in $\bbC[z]\Biglaurent{S_{F_0}^{-\frac{1}{2c_{F_0}}}}$. Therefore, the inverse matrix of $A+B = C$ exists.
\end{proof}

\subsubsection{Compatibility with pairings}\label{sect:IntertwinePairings}
We now study the interaction of $\Phi$ and the pairings.
We will use the following more precise calculation of the leading term of $\FT_{X_-}$. 

\begin{lemma}\label{lem:SpecialLTPairing}
Given $\alpha \in H_{\bbC^*}^*(W)$ such that $\alpha|_{F_0}$ as a polynomial in $\lambda$ has degree $\deg_{\lambda}(\alpha|_{F_0})\le r_{F_0,-}-1$, we have 
$$
    \FT_{X_-}(\alpha)|_{Q_W=\theta=0} = \kappa_{X_-}(\alpha).
$$
\end{lemma}

\begin{proof}
We proceed as in the proof of Lemma~\ref{lem:LTDiscreteFT} and consider
$$
    M_{X_-}(\tau_{X_-}(\theta))\FT_{X_-}(\alpha)|_{Q_W=\theta=0} = \kappa_{X_-}(\alpha) + \sum_{k>0} S_{F_0}^{-k} Q_W^{-ka} \kappa_{X_-}(\cS^k M_W(\theta)\alpha) |_{Q_W=\theta=0}.
$$
As observed there, for $k>0$, only the coefficient of $Q_W^{ka}$ in $\cS^k M_W(\theta)\alpha$ survives in the limit $Q_W = 0$.
By \eqref{eqn:ShiftOpLocal}, we have
$$
    \cS^kM_W(\theta)\alpha|_{F_+} = Q_W^{k(a+b)}\prod_{c=1}^ke_{\lambda+cz}(N_{F_+/W})^{-1}e^{-kz\partial_{\lambda}}(M_W(\theta)\alpha)|_{F_+},
$$
$$
    \cS^kM_W(\theta)\alpha|_{F_0} = Q_W^{ka} \frac{\prod_{c=1-k}^{0}e_{\lambda+cz}(N_{F_0,+})}{\prod_{c=1}^{k}e_{\lambda+cz}(N_{F_0,-})} e^{-kz\partial_{\lambda}}(M_W(\theta)\alpha)|_{F_0}.
$$
Therefore, if we denote the coefficient of $Q_W^{ka}$ in $\cS^kM_W(\theta)\alpha|_{Q_W=\theta=0}$ by $w_k \in H_{\bbC^*}^*(W)\laurent{z^{-1}}$, then we have $w_k|_{F_+}=0$ and $w_k|_{F_0} = \frac{\prod_{c=1-k}^{0}e_{\lambda+cz}(N_{F_0,+})}{\prod_{c=1}^{k}e_{\lambda+cz}(N_{F_0,-})} e^{-kz\partial_{\lambda}}c|_{F_0}$, which means that 
$$
    w_k|_{W\backslash F_-}=i_{F_0,*}^+ \left(\frac{\prod_{c=1-k}^{-1}e_{\lambda+cz}(N_{F_0,+})}{\prod_{c=1}^{k}e_{\lambda+cz}(N_{F_0,-})}e^{-kz\partial_{\lambda}}\alpha|_{F_0} \right). 
$$
In this expression, the degree of $z$ is bounded above by
$$
    (k-1)r_{F_0,+}-kr_{F_0,-}+r_{F_0,-}-1=(k-1)c_{F_0}-1\le -1.
$$
It follows that 
$$
    M_{X_-}(\tau_{X_-}(\theta))\FT_{X_-}(\alpha)|_{Q_W=\theta=0} = \kappa_{X_-}(\alpha) + O(z^{-1}).
$$
Since $M_{W}(\tau_{X_-}(\theta))|_{Q_W=0,\theta=0} = \Id +O(z^{-1})$, we have $\FT_{X_-}(\alpha)|_{Q_W=\theta=0} = \kappa_{X_-}(\alpha) + O(z^{-1})$, which implies that $\FT_{X_-}(\alpha)|_{Q_W=\theta=0} = \kappa_{X_-}(\alpha)$ since $\FT_{X_-}$ does not involve negative powers of $z$.
\end{proof}

\begin{proposition}\label{prop:IntertwinePairings}
The $\tR$-module isomorphism $\Phi$ in Theorem~\ref{thm:QDMDecompExt} intertwines the pairing $P_{X_-}$ with $P_{X_+}\oplus \bigoplus_{j=0}^{\abs{c_{F_0}}-1}P_{F_0}$.  
\end{proposition}

\begin{proof}
We first consider $\Phi$ and the pairings in the limit $Q_W=\theta=0$.
Consider the bases constructed in Section~\ref{sect:RefBases}.
For each $l, i$, since $s_{l,i}|_{F_0}=e_{\lambda}(N_{F_0,+})\lambda^l\alpha_{F_0,i}$, we have $\deg_{\lambda}(s_{l,i}|_{F_0}) = l + r_{F_0,+} \le r_{F_0,-}-1$.
Combining Lemma~\ref{lem:SpecialLTPairing} and the computation of $(A+B)^{-1}$ above , we have 
$$
    \Phi^{-1}|_{Q_W=\theta=0}(\alpha_{X_+,m}) \in \phi_2(\alpha_{X_+,m})+O\Big(S_{F_0}^{-\frac{1}{2c_{F_0}}}\Big)\Im (\phi_2) + \sum_{l=0}^{\abs{c_{F_0}}-1} O\Big(S_{F_0}^{-\frac{l+1}{c_{F_0}}}\Big) j_{-, *} (h_-^lH^*(F_0))),
$$
\begin{align*}
    \Phi^{-1}|_{Q_W=\theta=0}(\alpha_{F_0,i,j}) \in & O\Big(S_{F_0}^{-\frac{1}{2c_{F_0}}}\Big)\Im (\phi_2) + \sum_{l=0}^{\abs{c_{F_0}}-1} \lambda_0^{-l-r_{F_0,+}}(q_{F_0,j}\abs{c_{F_0}})^{-1}\zeta^{-l(j+r_{F_0,+})}
    \\& j_{-, *} (h_-^l(\alpha_{F_0,i} + O\Big(S_{F_0}^{-\frac{1}{2c_{F_0}}}\Big)H^*(F_0))).
\end{align*}
Then, by Lemma~\ref{lem:DecompClassicalPairing}, we have
\begin{align*}
    & P_{X_-}(\Phi^{-1}|_{Q_W=\theta=0}(\alpha_{X_+,m_1}),\Phi^{-1}|_{Q_W=\theta=0}(\alpha_{X_+,m_2})) \\
    & = \int_{X_-}\phi_2(\alpha_{X_+,m_1})\cup\phi_2(\alpha_{X_+,m_2}) + O\Big(S_{F_0}^{-\frac{1}{2c_{F_0}}}\Big)
    = \int_{X_+}\alpha_{X_+,m_1}\cup\alpha_{X_+,m_2}+O\Big(S_{F_0}^{-\frac{1}{2c_{F_0}}}\Big),\\
    & P_{X_-}(\Phi^{-1}|_{Q_W=\theta=0}(\alpha_{X_+,m}), \Phi^{-1}|_{Q_W=\theta=0}(\alpha_{F_0,i,j})) = O\Big(S_{F_0}^{-\frac{1}{2c_{F_0}}}\Big), \\
    &P_{X_-}(\Phi^{-1}|_{Q_W=\theta=0}(\alpha_{F_0,i_1,j_1}),\Phi^{-1}|_{Q_W=\theta=0}(\alpha_{F_0,i_2,j_2})) \\
    & = \sum_{l+l'=\abs{c_{F_0}}-1}\lambda_0^{-\abs{c_{F_0}}-1-2r_{F_0,+}}q_{F_0,j_1}^{-1}q_{F_0,j_2}^{-1}c_{F_0}^{-2} \zeta^{-l(j_1+r_{F_0,+})-l'(j_2+r_{F_0,+})}(-1)^{r_{F_0,+}}\int_{F_0}\alpha_{F_0,i_1}\cup \alpha_{F_0,i_2} \\&+O\Big(S_{F_0}^{-\frac{1}{2c_{F_0}}}\Big) 
    = \delta_{j_1,j_2}\int_{F_0}\alpha_{F_0,i_1}\cup \alpha_{F_0,i_2} + O\Big(S_{F_0}^{-\frac{1}{2c_{F_0}}}\Big).
\end{align*}

Here, we explain the derivation of the second equation above; the derivation of the third is similar. In view of \eqref{poincare pairing phi2 cup phi2} and \eqref{Poincare pairing phi2 cup proj}, it suffices to show that for any $\gamma_l, \gamma'_l \in H^*(F_0)$, $l = 0, \dots, \abs{c_{F_0}}-1$, we have
$$
    \sum_{l_1, l_2 = 0}^{\abs{c_{F_0}}-1} \int_{X_-} O\Big(S_{F_0}^{-\frac{l_1+1}{c_{F_0}}}\Big) j_{-, *} (h_-^{l_1} \gamma_{l_1}) \cup \lambda_0^{-l_2 - r_{F_0,+}} q_{F_0,j}^{-1} j_{-, *} (h_-^{l_2}(\alpha_{F_0,i} + O\Big(S_{F_0}^{-\frac{1}{2c_{F_0}}}\Big)\gamma'_{l_2})) = O\Big(S_{F_0}^{-\frac{1}{2c_{F_0}}}\Big).
$$
By \eqref{Poincare pairing proj cup proj}, there is no contribution to the summation from $l_1, l_2$ with $l_1+l_2<\abs{c_{F_0}}-1$. On the other hand, when $l_1+l_2\ge \abs{c_{F_0}}-1$, recalling that $\lambda_0\in \bbC S_{F_0}^{\frac{1}{c_{F_0}}}$ by \eqref{eqn:Lambdaj} and $q_{F_0,j}\in \bbC S_{F_0}^{-\frac{r_{F_0}-1}{2c_{F_0}}}$ by \eqref{eqn:qFj}, we see that the pairing lies in
$$
    O\bigg(S_{F_0}^{-\frac{l_1+1+l_2+r_{F_0,+}}{c_{F_0}} + \frac{r_{F_0}-1}{2c_{F_0}}}\bigg) = O\bigg(S_{F_0}^{-\frac{r_{F_0,-}-r_{F_0,+}+1}{2c_{F_0}}}\bigg) = O\bigg(S_{F_0}^{-\frac{1}{2c_{F_0}}}\bigg).
$$

The asymptotic calculations above implies that $\Phi$ preserves the pairings in the limit $Q_W=\theta=S_{F_0}^{-\frac{1}{2c_{F_0}}}=0$. We may then proceed as in \cite[Section 5.6.2]{iritani2023quantum} to conclude that $\Phi$ preserves the pairings.
\end{proof}

\subsection{Reduction of coordinates}\label{sect:RedCoord}

We now work towards restricting the base ring of the isomorphism $\Phi$ in Theorem~\ref{thm:QDMDecompExt} from $\tR$ to
$$
    R = \bbC[z]\Biglaurent{Q_{X_-}^{-\frac{a}{2c_{F_0}}}}\formal{Q_{X_-},\tau_{X_-}}
$$
by eliminating redundant coordinates and Novikov variables. In this subsection, we first show that the isomorphism can be established over
$$
    \bbC[z]\Biglaurent{S_{F_0}^{-\frac{1}{2c_{F_0}}}}\formal{Q_W,\tau_{X_-}}
$$
where the infinite collection $\theta$ of coordinates is reduced to the finite collection $\tau_{X_-}$.

Consider the finite-dimensional subspace 
$$
    H \coloneqq \bbC\{s_m, s_{l,i}\} \subset H_{\bbC^*}^*(W)
$$ 
spanned by the elements $s_m$ and $s_{l,i}$ constructed in Section~\ref{sect:RefBases}. We introduce the corresponding coordinates $\vartheta = (\vartheta^m, \vartheta^{l,i})$ on $H$. Note that the Kirwan map $\kappa_{X_-}$ restricts to an isomorphism between $H$ and $H^*(X_-)$.
The inclusion $H \subset H_{\bbC^*}^*(W)$ induces a homogeneous ring homomorphism 
$$
    \bbC[z]\Biglaurent{S_{F_0}^{-\frac{1}{2c_{F_0}}}}\formal{Q_W,\theta} \longrightarrow \bbC[z]\Biglaurent{S_{F_0}^{-\frac{1}{2c_{F_0}}}}\formal{Q_W,\vartheta}.
$$
Under this base change, the isomorphism $\Phi$ in Theorem~\ref{thm:QDMDecompExt} restricts to a $\bbC[z]\Biglaurent{S_{F_0}^{-\frac{1}{2c_{F_0}}}}\formal{Q_W,\vartheta}$-module isomorphism
\begin{equation}\label{eqn:PhiRestricted}
    \Phi |_H\colon \tau_{X_-,H}^*\QDM(X_-)^{\La}\longrightarrow\tau_{X_+,H}^*\QDM(X_+)^{\La}\oplus \bigoplus_{j=0}^{\abs{c_{F_0}}-1}\zeta_{j,H}^*\QDM(F_0)^{\La}.
\end{equation}
Here, the map $\tau_{X_{\pm},H}$ (resp.\ $\zeta_{j,H}$, $j = 0, \dots, \abs{c_{F_0}}-1$) denotes the restriction of $\tau_{X_{\pm}}$ (resp.\ $\zeta_j$) to $H$. The modules $\tau_{X_-,H}^*\QDM(X_-)^{\La}$, $\tau_{X_+,H}^*\QDM(X_+)^{\La}$, $\zeta_{j,H}^*\QDM(F_0)^{\La}$ are the pullbacks of the quantum $D$-modules $\tau_{X_-}^*\QDM(X_-)^{\La}$, $\tau_{X_+}^*\QDM(X_+)^{\La}$, $\zeta_j^*\QDM(F_0)^{\La}$ respectively. The restriction $\Phi |_H$ preserves the quantum connections and the pairings.

We have the following computation of the linear terms of the restricted maps $\tau_{X_{\pm},H}$ and $\zeta_{j,H}$.

\begin{lemma}\label{lem:ResMirrorMapLT}
We have
$$
    \tau_{X_-,H}(\vartheta)|_{Q_W=0}=\kappa_{X_-}(\vartheta)+O(\vartheta^2), \quad 
    \tau_{X_+,H}(\vartheta)|_{Q_W=0}=\kappa_{X_+}(\vartheta)+O\Big(S_{F_0}^{-\frac{1}{2c_{F_0}}},\vartheta^2\Big).
$$ 
For any $0\le l,j\le\abs{c_{F_0}}-1$, $0 \le i \le n_0-1$, we have
$$
    \zeta_{j,H}(\vartheta^{l,i}s_{l,i})|_{Q_W=0}=h_{F_0,j}+c_{F_0}\lambda_j+q_{F_0,j}\lambda_j^{l+\abs{c_{F_0}}}\vartheta^{l,i}\Big(\alpha_{F_0,i}+O\Big(S_{F_0}^{-\frac{1}{2c_{F_0}}}\Big)\Big)+O(\vartheta^2).
$$
\end{lemma}

\begin{proof}
We first consider the statements for $X_\pm$. Let $s_\nu$ denote $s_m$ or $s_{l, i}$, and $\vartheta^\nu$ denote the corresponding coordinate $\vartheta^m$ or $\vartheta^{l, i}$. By Proposition~\ref{prop:DiscreteFTDmodule}\eqref{item:DiscreteFTDmoduleConnection}, we have 
$$    
    \FT_{X_{\pm}}(s_\nu)=\FT_{X_{\pm}}(z\nabla_{\vartheta^\nu}1)=(z\partial_{\vartheta^\nu}+\partial_{\vartheta^\nu}\tau_{X_{\pm},H}(\vartheta)\star_{\tau_{X_{\pm},H}(\vartheta)})\FT_{X_{\pm}}(1). 
$$
For $X_-$, restricting the two sides to $Q_W=\vartheta=z=0$ and applying Lemma~\ref{lem:SpecialLTPairing}, we obtain
$$
    \kappa_{X_-}(s_\nu)=\partial_{\vartheta^\nu}\tau_{X_-,H}(\vartheta)|_{Q_W=\vartheta=0},
$$
where for the right-hand side, we use the fact $\tau_{X_-,H}(\vartheta)|_{Q_W=\vartheta=0}=0$ and $\FT_{X_-}(1)|_{Q_W=\vartheta=0}=1$. This implies the statement in the lemma. 
Similarly, for $X_+$, we may conclude by the leading term computation in Lemma~\ref{lem:LTDiscreteFT} and the observation that $\tau_{X_+,H}(\vartheta)$, $\FT_{X_+}(1)$, and $\FT_{X_+}(s_\nu)$ only involve the nonnegative powers of $S_{F_0}$.

The proof of the statement for $F_0$ is analogous. By Propositions~\ref{prop:ContFTDmodule}\eqref{item:ContFTDmoduleConnection}, we have
$$
    \FT_{F_0,j}(s_{l,i})=\FT_{F_0,j}(z\nabla_{\vartheta^{l,i}}1)=(z\partial_{\vartheta^{l,i}}+\partial_{\vartheta^{l,i}}\zeta_{j,H}(\vartheta)\star_{\zeta_{j,H}(\vartheta)})\FT_{F_0,j}(1).
$$
We may then conclude by the leading term computation in Lemma~\ref{lem:LTContFT} and the observation that $\zeta_{j}(\vartheta)|_{Q_W=0}$ only involves the nonnegative power of $S_{F_0}$.
\end{proof}
It follows from Lemma~\ref{lem:ResMirrorMapLT} that the map $\tau_{X_-,H}$ is formally invertible over the coefficient ring $\bbC\Biglaurent{S_{F_0}^{-\frac{1}{2c_{F_0}}}}\formal{Q_W}$. 
We now consider the base change along the isomorphism
$$
    \tau_{X_-,H}^{-1}\colon \bbC[z]\Biglaurent{S_{F_0}^{-\frac{1}{2c_{F_0}}}}\formal{Q_W,\vartheta} \longrightarrow \bbC[z]\Biglaurent{S_{F_0}^{-\frac{1}{2c_{F_0}}}}\formal{Q_W,\tau_{X_-}}.
$$
Under this base change, $\tau_{X_-,H}^*\QDM(X_-)^{\La}=H^*(X_-)[z]\Biglaurent{S_{F_0}^{-\frac{1}{2c_{F_0}}}}\formal{Q_W,\vartheta}$ is isomorphic to $\QDM(X_-)^{\La}=H^*(X_-)[z]\Biglaurent{S_{F_0}^{-\frac{1}{2c_{F_0}}}}\formal{Q_W,\tau_{X_-}}$ as quantum $D$-modules. 
On the other hand, we define the refined maps 
\begin{align*} 
    &\ttau_{X_+} \coloneqq \tau_{X_-,H}^{-1}\circ \tau_{X_+,H}\in H^*(X_+)\Biglaurent{S_{F_0}^{-\frac{1}{2c_{F_0}}}}\formal{Q_W,\tau_{X_-}}, \\ 
    &\tzeta_j \coloneqq\tau_{X_-,H}^{-1}\circ  \zeta_{j,H}\in H^*(F_0) \Biglaurent{S_{F_0}^{-\frac{1}{2c_{F_0}}}}\formal{Q_W,\tau_{X_-}}.
\end{align*}
The pullback of $\tau_{X_+,H}^*\QDM(X_+)^{\La}$ and $\zeta_{j,H}^*\QDM(F_0)^{\La}$ along $\tau_{X_-,H}^{-1}$ are given by $\ttau_{X_+}^*\QDM(X_+)^{\La}$ and $\tzeta_j^*\QDM(F_0)^{\La}$ respectively. 
Observe that, for the basis $\big\{\phi_2(\alpha_{X_+,m}), j_{-,*}(h_-^l  \psi_-^*(\alpha_{F_0,i}))\big\}$ of $H^*(X_-)$ (see \eqref{eqn:XMinusBasis}) and the corresponding coordinates $\{\tau_{X_-}^m, \tau_{X_-}^{l,i}\}$, we have 
\begin{equation}\label{equ:asymptotic of mirror map}
    \begin{split}
    &\ttau_{X_+}(\tau_{X_-})|_{Q_W=0} = \sum_{m} \tau_{X_-}^m \alpha_{X_+,m}+O\Big((\tau_{X_-})^2, S_{F_0}^{-\frac{1}{2c_{F_0}}}\Big),
    \\
    &\tzeta_j(\tau_{X_-})|_{Q_W=0} = c_{F_0}\lambda_j+h_{F_0,j}+\sum_{l,i}\tau_{X_-}^{l,i}(-1)^l \lambda_j^{l+\abs{c_{F_0}}}\Big(\alpha_{F_0,i}+O\Big(S_{F_0}^{-\frac{1}{2c_{F_0}}}\Big)\Big) \\
    &\hspace{21em} + O\Big(\tau_{X_-}^m,(\tau_{X_-}^{l,i})^2, S_{F_0}^{-\frac{1}{2c_{F_0}}}\Big).
    \end{split}
\end{equation}
In particular, by Remark~\ref{rem:ContPullbackDefined}, the pullbacks $\widetilde{\zeta_j}^*\QDM(F_0)^{\mathrm{La}}$ are well-defined.

Our discussion in this subsection is summarized into the following statement.


\begin{proposition}\label{prop:QDMDecompRedCoord}
The pullback of $\Phi |_H$ \eqref{eqn:PhiRestricted} under $\tau_{X_-,H}^{-1}$ gives a $\bbC[z]\Biglaurent{S_{F_0}^{-\frac{1}{2c_{F_0}}}}\formal{Q_W,\tau_{X_-}}$-module isomorphism 
$$
    \Phi'\colon \QDM(X_-)^{\La}\longrightarrow \ttau^*_{X_+}\QDM(X_+)^{\La}\oplus \bigoplus_{j=0}^{\abs{c_{F_0}}-1}\tzeta_j^*\QDM(F_0)^{\La}.
$$
Moreover, the isomorphism $\Phi'$ intertwines the quantum connections and intertwines the pairings.
\end{proposition}

\subsection{QDM decomposition over reduced base}\label{sect:RedNovikov}
In this subsection, we eliminate the redundant Novikov variables in the isomorphism $\Phi'$ in Proposition~\ref{prop:QDMDecompRedCoord} so that the decomposition eventually holds over $R = \bbC[z]\Biglaurent{Q_{X_-}^{-\frac{a}{2c_{F_0}}}}\formal{Q_{X_-},\tau_{X_-}}$. For convenience, we make the following assumption in this subsection.


\begin{assumption}\label{assump:FixedDivisor}
We assume that $X_\pm = F_\pm$ are $\bbC^*$-fixed divisors of $W$.
We also assume that $2\le r_{F_0,+}<r_{F_0,-}$.
\end{assumption}

We remark that the assumption that $X_\pm = F_\pm$ be made \emph{without loss of generality} since we may replace $W$ by its blowup along $F_+ \cup F_-$ with the induced $\bbC^*$-action if necessary and, as we will show in Theorem~\ref{thm:QDMDecompRed} below, the final base ring $R$ over which the decomposition holds depends only on $X_-$ but not on $W$.
Moreover, the statement we obtain in Theorem~\ref{thm:QDMDecompRed} in the case $r_{F_0,+} \ge 2$ will turn out to be consistent with the statement in the case $r_{F_0,+} = 1$ (where $X_- = \Bl_{F_0}X_+$) obtained by \cite[Theorem 5.18]{iritani2023quantum}; see Remark~\ref{rem:ConsistentWithBlowup}. 


Under Assumption~\ref{assump:FixedDivisor}, in view of \eqref{eqn:KirwanFactorize}, the maps $\kappa_{X_\pm}$ defined in Definition~\ref{def:KirwanFixedComp} are indeed the Kirwan maps. The dual maps $\kappa_{X_{\pm}}^*$ and $\kappa_{F_0}^*$ are given by
\begin{equation}\label{eqn:DualKirwan3Comp}
    \begin{aligned}
        &\kappa_{X_+}^*(\beta_+) 
        = i_{X_+,*}(\beta_+) + (\beta_+, c_1(N_{X_+/W}))(\lambda^*-a-b),\\
        &\kappa_{X_-}^*(\beta_-) 
        = i_{X_-,*}(\beta_-) - (\beta_-, c_1(N_{X_-/W}))\lambda^*,\\
        & \kappa_{F_0}^*(\beta_0) 
        = i_{F_0,*}(\beta_0) - \frac{(\beta_0, c_1(N_{F_0/W}))}{c_{F_0}}(\lambda^*-a)
    \end{aligned}
\end{equation}
respectively. In particular, $\kappa_{X_\pm}^*$ defines a monoid homomorphism from $\NE_{\bbN}(X_\pm)$ to $\NE_{\bbN}(W) + \bbZ (\lambda^*-a)$, and $\kappa_{F_0}^*$ defines a monoid homomorphism from $\NE_{\bbN}(F_0)$ to $\NE_{\bbN}(W) + \frac{1}{2c_{F_0}}  \bbZ (\lambda^*-a)$. 
Moreover, recall from Proposition~\ref{prop:VGITGeometry} that we have inclusions $\bbP(N_{F_0,\pm}) \hookrightarrow X_\pm$. Since $2 \le r_{F_0,+} < r_{F_0,-}$, $\bbP(N_{F_0,\pm})$ admits a nontrivial fiber curve class. Under the pushforward to $W$ via the inclusion, the fiber curve class in $\bbP(N_{F_0,-}) \subset X_-$ (resp.\ $\bbP(N_{F_0,+}) \subset X_+$) can be deformed to the class $a$ (resp.\ $b$) in $W$. We thus use $a$ and $b$ to denote the fiber curve classes in $X_\pm$ by an abuse of notation, and use the associated Novikov variables
$$
    Q_{X_-}^a, \quad Q_{X_+}^b.
$$
Notice that $(a,c_1(N_{X_-/W})) = 1$ because $N_{X_-/W}|_{\bbP(N_{F_0,-})} \cong N_{\bbP(N_{F_0,-})/\bbP(N_{F_0,-}\oplus \cO)} = \psi_-^*(N_{F_0,-})(1)$, where $\psi_- \colon \bbP(N_{F_0,-})\to F_0$ is the natural projection. Similarly, $(b, c_1(N_{X_+/W})) = 1$. It follows that
\begin{equation}\label{eqn:KappaDualAB}
    \kappa_{X_-}^*(a) = a - \lambda^* = - \kappa_{X_+}^*(b).
\end{equation}
In particular, $\kappa_{X_-}^*$ extends to an injective monoid homomorphism 
\begin{equation}\label{eqn:KappaDualXMinus}
    \kappa_{X_-}^* \colon \NE_{\bbN}(X_-) + \bbZ a \to \NE_{\bbN}(W) + \bbZ (\lambda^*-a).
\end{equation}
We may further replace $\bbZ a$ above by $\frac{1}{2c_{F_0}} \bbZ a$ and $\bbZ (\lambda^*-a)$ by $\frac{1}{2c_{F_0}} \bbZ (\lambda^*-a)$.

We first establish the following lemmas in preparation for the reduction of curve classes in this subsection.

\begin{lemma}\label{lem:KappaDualPlusFactor}
Under Assumption~\ref{assump:FixedDivisor}, the monoid homomorphism $\kappa_{X_+}^* \colon \NE_{\bbN}(X_+) \to \NE_{\bbN}(W) + \bbZ (\lambda^*-a)$ factors uniquely through $\kappa_{X_-}^*$. In fact, there is a unique monoid homomorphism $\varphi \colon \NE_{\bbN}(X_+) \to \NE_{\bbN}(X_-) + \bbZ a$ that makes the following diagram commute:
$$
    \begin{tikzcd}
        \NE_{\bbN}(X_+) \arrow[r, "\varphi"] \arrow[rd, "\kappa_{X_+}^*", swap] & \NE_{\bbN}(X_-)+\bbZ a \arrow[d, "\kappa_{X_-}^*"] \\
     & \NE_{\bbN}(W) + \bbZ (\lambda^*-a).
    \end{tikzcd}
$$
\end{lemma}

\begin{proof}
Recall from Section~\ref{sect:VGIT} that $X_\pm$ admit a common blowup $\pi_{\pm}\colon \tX = X_+ \times_{X_0} X_-\to X_{\pm}$ with the same exceptional divisor $E = \bbP(N_{F_0,+}) \times_{F_0} \bbP(N_{F_0,-})$. Let $l_\pm$ denote the fiber curve classes of the $\bbP(N_{F_0,+})$- and the $\bbP(N_{F_0,-})$-factors respectively, which are also viewed as curve classes in $\tX$. Then $(l_{\pm},[E]) = -1$. There are surjective maps $\pi_{\pm,*}\colon \NE_{\bbN}(\tX) \to \NE_{\bbN}(X_{\pm})$ which annihilate $l_\mp$ respectively.

Let $\beta_+ \in \NE_{\bbN}(X_+)$. We take $\tbeta \in \NE_{\bbN}(\tX)+\bbZ l_-$ such that $\pi_{+,*}(\tbeta) = \beta_+$. Replacing $\tbeta$ by $\tbeta + (\tbeta, [E])l_-$ if necessary, we may assume that $(\tbeta, [E]) = 0$. Setting $\beta_- = \pi_{-,*}(\tbeta)$, where $\beta_-\in \mathrm{NE}_{\bbN}(X_-)+\bbZ a$, we show that
$$
    \kappa_{X_-}^*(\beta_-) = \kappa_{X_+}^*(\beta_+).
$$
It suffices to prove that for any $\omega\in H^2_{\bbC^*}(W)$, we have $(\kappa_{X_-}^*(\beta_-),\omega) = (\kappa_{X_+}^*(\beta_+),\omega)$, or equivalently, 
$$
    (\beta_-, \kappa_{X_-}(\omega)) = (\beta_+,\kappa_{X_+}(\omega)).
$$
By Lemma~\ref{lem:KirwanIdentifyCoh}, we have $\kappa_{X_+}(\omega)|_{U_+} = \kappa_{X_-}(\omega)|_{U_-}$. Moreover, the restriction of the Gysin sequence 
$$    
    \cdots \to H^{*-2r_{F_0,+}}(\bbP(N_{F_0,-})) \to H^*(X_-)\to H^*(U_-) \to \cdots
$$
to $*=2$ together with the condition $r_{F_0,+} \ge 2$ shows that $i_{U_-}^*\colon H^2(X_-)\to H^2(U)$ is injective. It follows from Lemma~\ref{lem:KirwanPhi2} that $\phi_2(\kappa_{X_+}(\omega))=\kappa_{X_-}(\omega)$, where $\phi_2 = \pi_{-,*} \circ \pi_+^*$ is defined by the correspondence. Since $(\tbeta, [E]) = 0$, we have
$$
    (\beta_-, \kappa_{X_-}(\omega)) = (\pi_{-,*}(\tbeta), \pi_{-,*}(\pi_+^*(\kappa_{X_+}(\omega)))) = (\tbeta, \pi_+^*(\kappa_{X_+}(\omega))) = (\beta_+,\kappa_{X_+}(\omega))
$$
as desired.
Finally, since $\kappa_{X_-}^*$ is injective, the element $\beta_-$ is uniquely determined by the property $\kappa_{X_-}^*(\beta_-) = \kappa_{X_+}^*(\beta_+)$. We may then define the desired map $\varphi$ by setting $\varphi(\beta_+) \coloneqq \beta_-$.
\end{proof}

\begin{remark}\label{rem:CurveCorrespondence}
Lemma \ref{lem:KappaDualPlusFactor} does not require that $c_{F_0} \neq 0$. From the proof, we see that $\varphi$ can be extended to a monoid homomorphism
$$
    \varphi \colon \NE_{\bbN}(X_+) + \bbZ b \longrightarrow \NE_{\bbN}(X_-) + \bbZ a
$$
which is indeed an isomorphism. Moreover, Equation \eqref{eqn:KappaDualAB} implies that $\varphi(b) = -a$.
\end{remark}

\begin{lemma}\label{lem:KappaDual0Factor}
Under Assumption~\ref{assump:FixedDivisor}, the monoid homomorphism $\kappa_{F_0}^* \colon \NE_{\bbN}(F_0) \to \NE_{\bbN}(W) + \frac{1}{2c_{F_0}} \bbZ (\lambda^*-a)$ factors uniquely through $\kappa_{X_-}^*$. In other words, there is a unique monoid homomorphism $\varphi_0 \colon \NE_{\bbN}(F_0) \to \NE_{\bbN}(X_-) + \frac{1}{2c_{F_0}} \bbZ a$ that makes the following diagram commute:
$$
\begin{tikzcd}
    \NE_{\bbN}(F_0) \arrow[r, "\varphi_0"] \arrow[rd, "\kappa_{X_+}^*", swap] & \NE_{\bbN}(X_-) + \frac{1}{2c_{F_0}} \bbZ a \arrow[d, "\kappa_{X_-}^*"] \\
 & \NE_{\bbN}(W) + \frac{1}{2c_{F_0}} \bbZ (\lambda^*-a).
\end{tikzcd}
$$
\end{lemma}

\color{black}

\begin{proof}
We consider the Zariski closure $\overline{W}_{F_0,-}$ of the Bia{\l}ynicki-Birula cell $W_{F_0,-}$. Because $X_-$ is a $\bbC^*$-fixed divisor of $W$, we have $\overline{W}_{F_0,-}\cong \bbP(N_{F_0,-}\oplus \cO)$ with the zero section $F_0$ and the infinite divisor $\bbP(N_{F_0,-})$. Let $i_0\colon F_0\to \overline{W}_{F_0,-}$ and $i_{\infty}\colon \bbP(N_{F_0,-})\to \overline{W}_{F_0,-}$ denote the inclusions.
The projection $\psi_- \colon \bbP(N_{F_0,-})\to F_0$ induces a surjective homomorphism $\psi_{-,*} \colon N_1(\bbP(N_{F_0,-})) \to N_1(F_0)$ whose kernel is $\bbZ a$. It also induces a surjective homomorphism $\psi_{-,*}\colon \NE_{\bbN}(\bbP(N_{F_0,-}))\to \NE_{\bbN}(F_0)$.


Let $\beta_0 \in \NE_{\bbN}(F_0)$. There exists uniquely $\beta_\infty \in N_1(\bbP(N_{F_0,-}))$ such that $i_{\infty,*}(\beta_\infty) = i_{0,*}(\beta_0)$. Such $\beta_\infty$ is characterized by the properties $\psi_{-,*}(\beta_{\infty}) = \beta_0$ and 
$$
    (\beta_{\infty}, c_1(N_{\bbP(N_{F_0,-})/\bbP(N_{F_0,-}\oplus \cO)})) = (\beta_{\infty},\psi_-^*(c_1(N_{F_0,-})) + h_-) = 0,
$$
where $h_- = c_1(\cO_{\bbP(N_{F_0,-})}(1))$. This is due to the observations that $(i_{0,*}(\beta_0), [\bbP(N_{F_0,-})]) = 0$ and $N_{\bbP(N_{F_0,-})/\bbP(N_{F_0,-}\oplus \cO)} \cong \psi_-^*(N_{F_0,-}) \otimes \cO_{\bbP(N_{F_0,-})}(1)$. Then, the surjectivity of $\phi_{-,*}$ on $\NE_{\bbN}$ implies that $\beta_\infty \in \NE_{\bbN}(\bbP(N_{F_0,-})) + \bbZ a$.

Now consider the commutative diagram
$$
    \begin{tikzcd}
        \bbP(N_{F_0,-}) \arrow[r, "j_-"] \arrow[d, "i_{\infty}"] & X_- \arrow[d, "i_{X_-}"] \\
        \overline{W}_{F_0,-} \arrow[r, "i_{\overline{W}_{F_0,-}}"] & W,
    \end{tikzcd}
$$
Recall our abuse of notation that $a$ denotes the fiber class of $\bbP(N_{F_0,-})$ as well as its pushforwards under $j_-$ and $i_{X_-} \circ j_-$. For any $c \in \frac{1}{2c_{F_0}} \bbZ$, we have $j_{-,*}(\beta_\infty) + ca \in \NE_{\bbN}(X_-) + \frac{1}{2c_{F_0}} \bbZ a$, and we compute that
\begin{align*}
    \kappa_{X_-}^*(j_{-,*}(\beta_\infty) + ca) &= i_{X_-,*}(j_{-,*}(\beta_\infty) + ca) - (j_{-,*}(\beta_\infty) + ca, c_1(N_{X_-/W}))\lambda^*\\
    &= i_{F_0,*}(\beta_0) - c(\lambda^*-a)
\end{align*}
where we use that $(a,c_1(N_{X_-/W})) = 1$ and
$$
    (j_{-,*}(\beta_\infty), c_1(N_{X_-/W})) = (\beta_\infty,j_-^*(c_1(N_{X_-/W}))) = (\beta_\infty, c_1(N_{\bbP(N_{F_0,-})/\bbP(N_{F_0,-}\oplus \cO)})) = 0.
$$
We may then define
$$
    \varphi_{0}(\beta_0) \coloneqq j_{-,*}(\beta_\infty) + \frac{1}{c_{F_0}}(\beta_0, c_1(N_{F_0/W}))a,
$$
so that $\kappa_{X_-}^*(\varphi_{0}(\beta_0)) = \kappa_{F_0}^*(\beta_0)$. Since $\kappa_{X_-}^*$ is injective, the map $\varphi_0$ is uniquely determined.
\end{proof}

Now we start to reduce the base of the decomposition of quantum $D$-modules. In the base ring $\bbC[z]\Biglaurent{S_{F_0}^{-\frac{1}{2c_{F_0}}}}\formal{Q_W,\tau_{X_-}}$ of the decomposition in Proposition~\ref{prop:QDMDecompRedCoord}, all equivariant Novikov variables involved have form $\hS^\beta = Q_W^{\beta_W}S_{F_0}^{\frac{k}{2c_{F_0}}}$ with
$$
    \beta = \beta_W + \frac{k}{2c_{F_0}}(\lambda^*-a) \in \NE_{\bbN}(W) + \frac{1}{2c_{F_0}} \bbZ (\lambda^*-a).
$$
We first give a sufficient condition for an equivariant Novikov variable $\hS^\beta$ to split off from the decomposition.



\begin{lemma}\label{lem:CurvesWithoutContribution}
Assume Assumption~\ref{assump:FixedDivisor}.
Let $\beta \in \NE_{\bbN}(W) + \frac{1}{2c_{F_0}} \bbZ (\lambda^*-a)$ such that $(\beta,[X_-])\ne 0$ or $(\beta,[X_+])\ne 0$, where $[X_{\pm}] \in H^2_{\bbC^*}(W)$ is the equivariant divisor class of $X_{\pm}$ in $W$. Then, the coefficients of $\hS^{\beta}$ in $\ttau_{X_+}(\tau_{X_-})$ and $\tzeta_j(\tau_{X_-})$ are zero. Furthermore, in the matrix representation of the isomorphism $\Phi'$ in Proposition~\ref{prop:QDMDecompRedCoord} (under any bases), the coefficient of $\hS^\beta$ in any entry is zero.
\end{lemma}

\begin{proof}
We consider the quantum connections along the $[X_{\pm}]$-direction. 
On $\QDM(X_-)^{\La}$, where we denote the connection by $\nabla$ for simplicity, we have
$$
    z\nabla_{[X_{\pm}] \hS \partial_{\hS}}
    = z [X_{\pm}] \hS \partial_{\hS} + \left( \kappa_{X_-}([X_{\pm}]) \star_{\tau_{X_-}} \right).
$$
On the pullbacks $\ttau^*_{X_+}\QDM(X_+)^{\La}$ and $\tzeta_j^*\QDM(F_0)^{\La}$, we have
$$
    z \ttau^*_{X_+}\nabla_{[X_{\pm}] \hS \partial_{\hS}}
    = z [X_{\pm}] \hS \partial_{\hS} + \left( \kappa_{X_+}([X_{\pm}]) \star_{\ttau_{X_+}(\tau_{X_-})} \right) + \big([X_{\pm}] \hS \partial_{\hS}\ttau_{X_+}(\tau_{X_-})\big) \star_{\ttau_{X_+}(\tau_{X_-})},
$$
$$
    z \tzeta_j^*\nabla_{[X_{\pm}] \hS \partial_{\hS}}
    = z [X_{\pm}] \hS \partial_{\hS} + \left( \kappa_{F_0}([X_{\pm}]) \star_{\tzeta_j(\tau_{X_-})} \right) + \big([X_{\pm}] \hS \partial_{\hS}\tzeta_j(\tau_{X_-})\big) \star_{\tzeta_j(\tau_{X_-})}
$$
respectively.
By Lemma~\ref{lem:KirwanKernel}, we have $[X_{\pm}]\in \Ker(\kappa_{X_{\pm}})$. In addition, we have $[X_{\pm}]\in\Ker(\kappa_{F_0})$.
Because $\Phi'$ intertwines $z\nabla_{[X_{\pm}] \hS \partial_{\hS}}$ with $z\ttau_{X_+}^*\nabla_{[X_{\pm}] \hS \partial_{\hS}}\oplus \bigoplus_{j=0}^{\abs{c_{F_0}}-1}z\tzeta_j^*\nabla_{[X_{\pm}] \hS \partial_{\hS}}$, by restricting the above to $z=0$, we see that 
\begin{equation}\label{eqn:MirrorMapPartialVanish}
    [X_{\pm}] \hS \partial_{\hS}\tzeta_j(\tau_{X_-})=[X_{\pm}] \hS \partial_{\hS}\ttau_{X_+}(\tau_{X_-})=0.
\end{equation}
Since $[X_{\pm}]\hS\partial_{\hS}\hS^\beta = (\beta,[X_{\pm}])\hS^\beta$, it follows from the hypothesis on $\beta$ that the coefficients of $\hS^{\beta}$ in $\ttau_{X_+}(\tau_{X_-})$ and $\tzeta_j(\tau_{X_-})$ are zero. Furthermore, Equation \eqref{eqn:MirrorMapPartialVanish} also implies that $\Phi'$ commutes with the operator $[X_{\pm}]\hS\partial_{\hS}$, which implies that the variable $\hS^\beta$ is not involved in $\Phi'$.
\end{proof}

Now we show that an equivariant curve class $\beta$ that does not meet the condition of Lemma~\ref{lem:CurvesWithoutContribution} always results from an extension from $X_-$.

\begin{lemma}\label{lem:BaseReductionXMinus}
Under Assumption~\ref{assump:FixedDivisor}, if $\beta \in \NE_{\bbN}(W) + \frac{1}{2c_{F_0}} \bbZ (\lambda^*-a)$ satisfies that
$(\beta,[X_{\pm}])=0$, it lies in the image of $\kappa_{X_-}^*$ in \eqref{eqn:KappaDualXMinus}.
\end{lemma}

\begin{proof}
We write $\beta = \beta_W + \frac{k}{2c_{F_0}}(\lambda^*-a)$ where $d_W\in \NE_{\bbN}(W)$ and $k \in \bbZ$.
Since every curve of $W$ can be deformed to a $\bbC^*$-invariant curve, we have 
\begin{align*}
    \NE_{\bbN}(W) &= i_{X_+,*}(\NE_{\bbN}(X_+)) + i_{X_-,*}(\NE_{\bbN}(X_-)) + i_{F_0,*}(\NE_{\bbN}(F_0)) + \bbN\{a,b\}\\
    &= i_{X_+,*}(\NE_{\bbN}(X_+)) + i_{X_-,*}(\NE_{\bbN}(X_-)) + i_{F_0,*}(\NE_{\bbN}(F_0))
\end{align*}
where the second equality follows from that the class $a$ (resp.\ $b$) can be deformed into $X_-$ (resp.\ $X_+$). We may then write
$$
    \beta_W = i_{X_+,*}(\beta_+) + i_{X_-,*}(\beta_-) + i_{F_0,*}(\beta_0)
$$
for some $\beta_{\pm}\in \NE_{\bbN}(X_\pm)$ and $\beta_0\in \NE_{\bbN}(F_0)$.

Recall that $\lambda^*-a$ is induced by the equivariant pushforward from $F_0$ and thus $(\lambda^*-a,[X_{\pm}]) = 0$. We have
$$
    0 = (\beta,[X_{\pm}]) = (i_{X_+,*}(\beta_+) + i_{X_-,*}(\beta_-), [X_\pm]) = 0. 
$$
This implies that
$$
    0 = (\beta_-, c_1(N_{X_-/W})) = (\beta_+, c_1(N_{X_+/W})).
$$
It follows from \eqref{eqn:DualKirwan3Comp} that $i_{X_-,*}(\beta_-) = \kappa_{X_-}^*(\beta_-)$, and that $i_{X_+,*}(\beta_+) = \kappa_{X_+}^*(\beta_+)$ which together with Lemma~\ref{lem:KappaDualPlusFactor} implies that $i_{X_+,*}(\beta_+)$ lies in the image of $\kappa_{X_-}^*$.

It remains to consider the term $i_{F_0,*}(\beta_0) + \frac{k}{2c_{F_0}}(\lambda^*-a)$ in $\beta$. By \eqref{eqn:DualKirwan3Comp}, this can be rewritten as
$$
    i_{F_0,*}(\beta_0) + \frac{k}{2c_{F_0}}(\lambda^*-a) = \kappa_{F_0}^*(\beta_0) + \frac{k'}{2c_{F_0}}(\lambda^*-a) = \kappa_{F_0}^*(\beta_0) - \frac{k'}{2c_{F_0}}\kappa_{X_-}^*(a)
$$
for some $k' \in \bbZ$. We may then conclude by Lemma~\ref{lem:KappaDual0Factor}.
\end{proof}

Using the ring homomorphisms
$$ 
    \bbC[z]\formal{Q_{X_+},\tau_{X_+}} \to \bbC[z]\Biglaurent{Q_{X_-}^{-\frac{a}{2c_{F_0}}}}\formal{Q_{X_-},\tau_{X_+}}, \quad
    \bbC[z]\formal{Q_{F_0},\tau_{F_0}} \to \bbC[z]\Biglaurent{Q_{X_-}^{-\frac{a}{2c_{F_0}}}}\formal{Q_{X_-},\tau_{F_0}}
$$
induced by $\varphi$ in Lemma~\ref{lem:KappaDualPlusFactor} and $\varphi_0$ in Lemma~\ref{lem:KappaDual0Factor} respectively, we introduce the following version of quantum $D$-modules:
\begin{align*}
    \QDM(X_-)^{\red, \La} &\coloneqq \QDM(X_-) \otimes_{\bbC[z]\formal{Q_{X_-},\tau_{X_-}}} \bbC[z]\Biglaurent{Q_{X_-}^{-\frac{a}{2c_{F_0}}}}\formal{Q_{X_-},\tau_{X_-}} \\
    &=H^*(X_-)[z]\Biglaurent{Q_{X_-}^{-\frac{a}{2c_{F_0}}}}\formal{Q_{X_-},\tau_{X_-}},\\
    \QDM(X_+)^{\red, \La} &\coloneqq \QDM(X_+) \otimes_{\bbC[z]\formal{Q_{X_+},\tau_{X_+}}} \bbC[z]\Biglaurent{Q_{X_-}^{-\frac{a}{2c_{F_0}}}}\formal{Q_{X_-},\tau_{X_+}}\\ 
    &=H^*(X_+)[z]\Biglaurent{Q_{X_-}^{-\frac{a}{2c_{F_0}}}}\formal{Q_{X_-},\tau_{X_+}},\\
    \QDM(F_0)^{\red, \La} &\coloneqq \QDM(F_0) \otimes_{\bbC[z]\formal{Q_{F_0},\tau_{F_0}}} \bbC[z]\Biglaurent{Q_{X_-}^{-\frac{a}{2c_{F_0}}}}\formal{Q_{X_-},\tau_{F_0}} \\
    &=H^*(F_0)[z]\Biglaurent{Q_{X_-}^{-\frac{a}{2c_{F_0}}}}\formal{Q_{X_-},\tau_{F_0}}.
\end{align*}
They are respectively equipped with the quantum connections 
\begin{align*}
    & \nabla^{X_-,\red, \La}_{\tau_{X_-}^i} = \partial_{\tau_{X_-}^i} + z^{-1} \left( \phi_{X_-,i} \star_{\tau_{X_-}} \right), 
    && \nabla^{X_+,\red, \La}_{\tau_{X_+}^i} = \partial_{\tau_{X_+}^i} + z^{-1} \left( \phi_{X_+,i} \star_{\tau_{X_+}} \right), \\
    & \nabla_{z \partial_z}^{X_-,\red, \La} = z \partial_z - z^{-1} \left( E_{X_-} \star_{\tau_{X_-}} \right) + \mu_{X_-},
    && \nabla_{z \partial_z}^{X_+,\red, \La} = z \partial_z - z^{-1} \left( E_{X_+} \star_{\tau_{X_+}} \right) + \mu_{X_+}, \\
    & \nabla_{\xi Q\partial_Q}^{X_-,\red, \La} = \xi Q\partial_{Q} + z^{-1} \left( \xi \star_{\tau_{X_-}} \right),
    && \nabla_{\xi Q\partial_Q}^{X_+,\red, \La} = \xi Q\partial_{Q} + z^{-1} \left( \varphi^*(\xi) \star_{\tau_{X_+}} \right),
\end{align*} 
\begin{align*}
    & \nabla^{F_0,\red, \La}_{\tau_{F_0}^i} = \partial_{\tau_{F_0}^i} + z^{-1} \left( \phi_{F_0,i} \star_{\tau_{F_0}} \right), \\
    & \nabla_{z \partial_z}^{F_0,\red, \La} = z \partial_z - z^{-1} \left( E_{F_0} \star_{\tau_{F_0}} \right) + \mu_{F_0}, \\
    & \nabla_{\xi Q\partial_Q}^{F_0,\red, \La} = \xi Q\partial_{Q} + z^{-1} \left( \varphi_0^*(\xi) \star_{\tau_{F_0}} \right),
\end{align*}
where $\varphi^*: H^2(X_+) \to H^2(X_-)$ (resp.\ $\varphi_{F_0}^*\colon H^2(X_-)\to H^2(F_0)$) is the dual map of $\varphi$ (resp.\ $\varphi_0$).
They are also equipped with the pairings $P_{X_-}$, $P_{X_+}$ and $P_{F_0}$ respectively.

By Lemmas~\ref{lem:CurvesWithoutContribution} and~\ref{lem:BaseReductionXMinus}, together with the injectivity of $\kappa_{X_-}^*$, the refined maps $\ttau_{X_+}(\tau_{X_-})$ and $\tzeta_j(\tau_{X_-})$, $j = 0, \dots, \abs{c_{F_0}}-1$, factor uniquely through $\kappa_{X_-}^*$ and determine the maps
\begin{align*}
    & \tau_{X_+}^{\red}\colon \bbC[z]\Biglaurent{Q_{X_-}^{-\frac{a}{2c_{F_0}}}}\formal{Q_{X_-},\tau_{X_+}} \to \bbC[z]\Biglaurent{Q_{X_-}^{-\frac{a}{2c_{F_0}}}}\formal{Q_{X_-},\tau_{X_-}}, \\
    & \zeta_j^{\red} \colon \bbC[z]\Biglaurent{Q_{X_-}^{-\frac{a}{2c_{F_0}}}}\formal{Q_{X_-},\tau_{F_0}} \to \bbC[z]\Biglaurent{Q_{X_-}^{-\frac{a}{2c_{F_0}}}}\formal{Q_{X_-},\tau_{X_-}}.
\end{align*}
Let $(\tau_{X_+}^{\red})^*\QDM(X_+)^{\red, \La}$ and $(\zeta_j^{\red})^*\QDM(F_0)^{\red, \La}$ denote pullbacks of $\QDM(X_+)^{\red, \La}$ and $\QDM(F_0)^{\red, \La}$ under $\tau_{X_+}^{\red}$ and $\zeta_j^{\red}$ respectively. Notice that the two pullbacks are modules over $R = \bbC[z]\Biglaurent{Q_{X_-}^{-\frac{a}{2c_{F_0}}}}\formal{Q_{X_-},\tau_{X_-}}$. By Lemmas~\ref{lem:CurvesWithoutContribution} and~\ref{lem:BaseReductionXMinus}, the isomorphism $\Phi'$ in Proposition~\ref{prop:QDMDecompRedCoord} induces an isomorphism as $R$-modules. We arrive at the following conclusion.




\begin{theorem}[QDM decomposition]\label{thm:QDMDecompRed}
There is an $R$-module isomorphism
$$
    \Phi^{\red}\colon \QDM(X_-)^{\red, \La} \longrightarrow (\tau_{X_+}^{\red})^*\QDM(X_+)^{\red, \La} \oplus \bigoplus_{j=0}^{\abs{c_{F_0}}-1}(\zeta_j^{\red})^*\QDM(F_0)^{\red, \La}
$$ 
that intertwines the quantum connections and intertwines the pairing $P_{X_-}$ with $P_{X_+}\oplus \bigoplus_{j=0}^{\abs{c_{F_0}}-1}P_{F_0}$.
\end{theorem}

\begin{remark}\label{rem:ConsistentWithBlowup}
As discussed earlier, Theorem~\ref{thm:QDMDecompRed} does not require Assumption~\ref{assump:FixedDivisor}.
In the case $r_{F_0,+}=1$, where $X_-\cong \mathrm{Bl}_{F_0}X_+$ by Proposition~\ref{prop:VGITGeometry}, Iritani \cite[Theorem 5.18]{iritani2023quantum} established a decomposition of quantum $D$-modules over the base ring $\bbC[z]\Biglaurent{S_{F_0}^{-\frac{1}{2c_{F_0}}}}\formal{Q_{X_+},\tau_{X_-}}$. 
Under the map
$$
    \NE_{\bbN}(X_-) + \frac{1}{2c_{F_0}}\bbZ a \to \NE_{\bbN}(X_+) + \frac{1}{2c_{F_0}} \bbZ (\lambda^*-a), \quad \beta_- + \frac{k}{2c_{F_0}} a \mapsto q_{-,*}(\beta_-) + \frac{k- (\beta_-, [E_-])}{2c_{F_0}} (\lambda^*-a)
$$
where $q_-: X_- \to X_+ (=X_0)$ is the projection and $E_- (=P_-)$ is the exceptional divisor in $X_-$, the base ring is indeed isomorphic to $R=\bbC[z]\Biglaurent{Q_{X_-}^{-\frac{a}{2c_{F_0}}}}\formal{Q_{X_-},\tau_{X_-}}$.
It follows that \cite[Theorem 5.18]{iritani2023quantum} is consistent with our Theorem~\ref{thm:QDMDecompRed}.
\end{remark}

\section{Decomposition of quantum $D$-modules for general VGIT wall-crossings and flips}\label{sect:General}

In this section, we establish the decomposition of quantum $D$-modules for a general simple VGIT wall-crossing. For a wall-crossing that is a type-$(r_+-1, r_--1)$ flip with $r_+ < r_-$, the key step is a reduction to the 3-component $\bbC^*$-case (considered in Section~\ref{sect:Decomposition}) via the master space construction of Thaddeus \cite{thaddeus1996geometric}, which we explain in Section~\ref{sect:MasterSpace}. In the case of flops ($r_+ = r_-$), we appeal to the previous works \cite{lee2010flops, lee2016ordinaryflops1, lee2016ordinaryflops2, lee2016ordinaryflops3} and reformulate their main result in terms of quantum $D$-modules in Section~\ref{sect:GeneralVGITFlops}. Furthermore, we discuss the application to decompositions of quantum $D$-modules of general flips (Conjecture~\ref{conj:IntroGeneralFlip}).

\subsection{Master space construction for VGIT wall-crossings}\label{sect:MasterSpace}
We first review the master space construction of \cite[Construction 3.1]{thaddeus1996geometric}; see also \cite[Section 3]{liu2025invariance} for a summary. We take the setup of Section~\ref{sect:Setup} where $W$ is a smooth, quasi-projective variety that is projective over an affine variety and admits a linear action of a reductive algebraic group $G$, and $L_\pm$ are two ample linearizations that give a simple VGIT wall-crossing. Under condition~\eqref{cond:WtGcd} in Definition~\ref{def:SimpleWall}, fixing the isomorphism $G_x \cong \bbC^*$ such that $v_+>0$ and replacing $L_{\pm}$ by some positive powers if necessary, we assume that $G_x$ acts on $(L_+\otimes L_-^{\vee})|_{x}$ by scaling with weight $1$ for any $x\in W^0$.

Consider the projective bundle
$$
    Y \coloneqq \bbP(L_+ \oplus L_-)
$$
which is equipped with the natural $G$-action and the action of $\bbC^* =\colon T$ scaling the $L_+$-factor. The ample line bundle $\cO(1)$ over $Y$ admits a natural linearization of the $G$-action. Define
$$
    M \coloneqq Y \gitquot_{\cO(1)} G
$$
which is a quasi-projective variety equipped with the residue $T$-action. In addition, there is a family of $T$-linearizations on $\cO(1)$ over $Y$ parameterized by $s \in [-1, 1]$. More precisely, if $(u_+, u_-)$ are fiberwise coordinates on $L_+ \oplus L_-$, then the linearization corresponding to $s$ is specified by the action $t \cdot (u_+, u_-) = \big(t^{(1+s)/2} u_+, t^{(1-s)/2} u_-\big)$ for $t \in T$.
This family descends to a family of $T$-linearizations on $M$. Let $Y_\pm \subset Y$ denote the $0$- and $\infty$-sections of the projective bundle respectively. Then the GIT quotient of $Y_\pm$ under $G$ are respectively isomorphic to $X_\pm$, which are the two $T$-fixed divisors of $M$. In particular, they are the highest/lowest GIT quotients of $M$ by $T$ induced by the $T$-linearizations corresponding to $s = \pm 1$ respectively.


Conditions~\eqref{cond:SimpleWall} and \eqref{cond:WtGcd} in Definition~\ref{def:SimpleWall} imply that $M$ is also smooth. Let $q\colon Y \to W$ denote the projection, and set
$$
    Y^0 \coloneqq \pi^{-1}(W^0) \setminus (Y_+ \cup Y_-)
$$
which is a principal $T$-bundle over $W^0$. Then $G$ acts freely on $Y^0$. Denote this principal $G$-bundle by $p\colon Y^0 \to Y^0 \gitquot G \coloneqq M^0$. Then
$$
    M^T = X_+ \sqcup X_- \sqcup M^0.
$$
For any point $x \in W^0$, the stabilizer $G_x \cong \bbC^*$ is identified with $T$, and the $G_x$-action on $(L_+\otimes L_-^{\vee})|_{x}$ is identified with the $T$-action. There is an identification $q^*(N_{W^0/W}) = p^*(N_{M^0/M})$ under which the $G_x$ weights are identified with the $T$-weights. The condition~\eqref{cond:WtPm1} in Definition~\ref{def:SimpleWall} then implies that the $T$-weights on $N_{M^0/M}$ can only be $\pm 1$. Moreover, the connectedness of $W^0$ implies that $M^0$ is connected. In fact, we have
$$
    W^0\gitquot G(0) = M^0.
$$
The projectivity of $X_\pm$ implies that
$$
    H^0(M, \cO_M) = H^0(W, \cO_W)^G = H^0(X_\pm, \cO_{X_\pm}) = \bbC
$$
which implies that $M$ is projective. Therefore, $M$ is an instance of a 3-component $\bbC^*$-VGIT wall-crossing (in the sense of Assumption~\ref{3-component assumption}) between the quotients $X_\pm$.

\begin{example}[$\bbC^*$-VGIT wall-crossings]
Suppose $G = \bbC^*$. In this case, $W^0$ is a connected component of $W^{\bbC^*}$. Moreover, for any $x \in W^0$, the stabilizer group $G_x$ is equal to $G$. By the geometric relation between $W$ and $M$ discussed above, on the set $Y^0$ the $G$-action is identified with the $T$-action and thus $M^0 = W^0$.
\end{example}

\subsection{Decomposition of quantum $D$-modules for VGIT wall-crossing flips}\label{sect:GeneralVGITFlips}
Let $X_\pm$ be smooth projective varieties that are related by a simple VGIT wall-crossing as above. Suppose 
$X_- \dashrightarrow X_+$ is a type-$(r_+-1, r_--1)$ flip with $r_+ \le r_-$. If $r_+ < r_-$, we may apply Theorem~\ref{thm:QDMDecompRed} to the master space $M$ to obtain the following decomposition theorem for the quantum $D$-modules. In particular, the quantity $c_{M^0}$ defined for the $\bbC^*$-fixed component $M^0$ in $M$ is equal to $r_+ - r_-$.

\begin{theorem}\label{thm:GeneralQDMDecomposition}    
Suppose $r_+ < r_-$. There exist maps 
$$
    \tau_{X_+}^{\red}(\tau_{X_-}) \in  H^*(X_+)\Biglaurent{Q_{X_-}^{\frac{a}{2(r_- - r_+)}}}\formal{Q_{X_-},\tau_{X_-}},
$$ 
$$
    \zeta_j^{\red}(\tau_{X_-}) \in H^*(M_0)\Biglaurent{Q_{X_-}^{\frac{a}{2(r_- - r_+)}}}\formal{Q_{X_-},\tau_{X_-}}, \quad j = 0, \dots, r_- - r_+ -1,
$$
and a $\bbC[z]\Biglaurent{Q_{X_-}^{\frac{a}{2(r_- - r_+)}}}\formal{Q_{X_-},\tau_{X_-}}$-module isomorphism 
$$
    \Psi \colon \QDM(X_-)^{\La, \red} \longrightarrow (\tau_{X_+}^{\red})^*\QDM(X_+)^{\La, \red} \oplus \bigoplus_{j=0}^{r_- - r_+ - 1}(\zeta_j^{\red})^*\QDM(M^0)^{\La, \red}
$$ 
that intertwines the quantum connections and intertwines the pairing $P_{X_-}$ with $P_{X_+} \oplus \bigoplus_{j=0}^{r_- - r_+ -1}P_{M^0}$.
\end{theorem}

Following the same argument as \cite[Section 5.9]{iritani2023quantum}, we can directly deduce decompositions of quantum cohomology rings for  VGIT wall-crossing flips. The statement is provided in Corollary~\ref{cor:decom of qh}. As noted in Remark~\ref{rem:Eigenvalues}, the decomposition of quantum cohomology rings induces a decomposition of the eigenvalues of quantum multiplications by Euler vector fields along the restriction $Q_{X_-}^{\beta} = 0$ for all $\beta\notin \frac{1}{2(r_--r_+)}\bbZ a$ and $\tau_{X_-}=0$. More precisely, along this locus, 
$E_{X_+}\star_{\tau_{X_+}^{\red}(\tau_{X_-})}$ and $E_{M^0}\star_{\zeta_j^{\red}(\tau_{X_+})}$ deform to families of operators 
$$
    E_{X_+}^q\colon H^*(X_+)\to H^*(X_+), \quad E_{M^0,j}^q\colon H^*(M^0)\to H^*(M^0)
$$
respectively, where $q=Q_{X_-}^{-\frac{a}{2(r_--r_+)}}$ is the parameter for this family. The evaluation $Q_{X_-}^{-\frac{a}{2(r_--r_+)}}=q\in \bbC$ is well-defined because $H^*(X_+)\Biglaurent{Q_{X_-}^{\frac{a}{2(r_--r_+)}}}=H^*(X_+)\Big[Q_{X_-}^{\pm\frac{a}{2(r_--r_+)}}\Big]$ and there is no convergence problem. For any $q\in \bbC$, we can compute the eigenvalues of $E_{X_+}^q$ and $E_{M^0,j}^q$ as follows.

\begin{proposition}\label{prop:eigenvalue of Euler vector field}
For any $q\in \bbC$, all the eigenvalues of $E_{X_+}^q$ are $0$ and those of $E_{M^0,j}^q$ are $(r_--r_+)q^2e^{\frac{\pi\sqrt{-1}(r_+-2j)}{r_--r_+}}$.
\end{proposition}

In the proof of the proposition, we abuse the notation $\lambda_j$ both for its original definition in \eqref{eqn:Lambdaj} and its preimage under the base change of Novikov variables induced by $\kappa_{X_-}^*$. In other words, we identify $Q_{X_-}^{-\frac{a}{(r_--r_+)}}$ with $S_{M^0}^{\frac{1}{c_{M^0}}}$ through $\kappa_{X_-}^*$.

\begin{proof}
Along the locus $Q_{X_-}^{\beta}=0$ for $\beta\notin \frac{1}{2(r_--r_+)}\bbZ a$ and $\tau_{X_-}=0$, by 
the asymptotics \eqref{equ:asymptotic of mirror map}, we have 
$$
    \tau_{X_+}^{\red}(0)|_{Q_{X_-}^{\beta}=0}\in H^*(X_+)\Big[Q_{X_-}^{\frac{a}{2(r_--r_+)}}\Big], \quad  \zeta_j^{\red}(0)|_{Q_{X_-}^{\beta}=0}\in \bbC Q_{X_-}^{-\frac{a}{2(r_--r_+)}}+ H^*(M^0)\Big[Q_{X_-}^{\frac{a}{2(r_--r_+)}}\Big],
$$
where the term $\bbC Q_{X_-}^{-\frac{a}{2(r_--r_+)}}$ does not contribute to the quantum product. Because the evaluation $Q_{X_-}^{-\frac{a}{2(r_--r_+)}}=q$ strictly increases the degree of $Q_{X_-}^{\frac{a}{2(r_--r_+)}}$ (recall that $\mathrm{deg} \Big(Q_{X_-}^{\frac{a}{2(r_--r_+)}}\Big)=-1$), the eigenvalues of $E_{X_+}^q$ and $E_{M^0,j}^q$ are given by their $H^0$-components. After making the restriction $Q_{X_-}^{\beta}=0$ for $\beta\notin \frac{1}{2(r_--r_+)}\bbZ a$ and $\tau_{X_-}=0$, the $H^0$-components of $\tau_{X_+}^{\red}$ and $\zeta_j^{\red}$ are computed to be $0$ and $c_{M^0}\lambda_j=(r_--r_+)q^2e^{\frac{\pi\sqrt{-1}(r_+-2j)}{r_--r_+}}$ respectively. The proposition then follows from the formula of Euler vector fields \eqref{Euler vector field}.
\end{proof}

\subsection{Isomorphism of quantum $D$-modules for VGIT wall-crossing flops}
\label{sect:GeneralVGITFlops}
Now we consider a VGIT wall-crossing between $X_\pm$ as in Section~\ref{sect:GeneralVGITFlips} with $r_+ = r_-$. In this case, $X_-\dasharrow X_+$ is an ordinary flop and the varieties $X_\pm$ are $K$-equivalent (see Remark~\ref{rem:KDomination}). 
Works of Lee, Lin, Qu, and Wang \cite{lee2010flops, lee2016ordinaryflops1, lee2016ordinaryflops2, lee2016ordinaryflops3} show that quantum cohomology is invariant under ordinary flops after analytic continuation. As mentioned in Remark~\ref{rem:definition of standard flip}, the result remains valid without assuming that the contractions $q_\pm$ are log-extremal.
In this subsection, we summarize this result and reformulate it as an isomorphism of quantum $D$-modules, which complements Theorem~\ref{thm:GeneralQDMDecomposition}.

Suppose $r_+ = r_- \ge 2$ (since otherwise $X_\pm$ are isomorphic). Recall that we denote the fiber curve class of $P_\pm \subset X_{\pm}$ in $\NE_{\bbN}(X_{\pm})$ by $b$ and $a$ respectively. 
We define the quotient monoids
$$
        \NE_{\bbN}(X_+)_{/b} = \NE_{\bbN}(X_+)/\sim_+, \quad \NE_{\bbN}(X_-)_{/a}=\NE_{\bbN}(X_-)/\sim_- 
$$
where the equivalence relation $\sim_+$ is defined by $\beta_1 \sim_+ \beta_2$ if and only if $\beta_1 - \beta_2 \in \bbZ b$, and $\sim_-$ is defined by $\beta_1 \sim_- \beta_2$ if and only if $\beta_1 - \beta_2 \in \bbZ a$. 
Given an equivalence class $[\beta_+] \in \NE_{\bbN}(X_+)_{/b}$, we may find a unique representative $\beta_{+, \min} \in \NE_{\bbN}(X_+)$ in the same equivalence class such that $\beta_{+, \min} - a \notin \NE_{\bbN}(X_+)$. Similarly, for any $[\beta_-] \in \NE_{\bbN}(X_-)_{/a}$, we may define the representative $\beta_{-, \min} \in \NE_{\bbN}(X_-)$. In this notation,
a general element of $\bbC\formal{Q_{X_+}}$ has the form 
$$
    \sum_{\beta_+ \in \NE_{\bbN}(X_+)}c_{\beta_+} Q_{X_+}^{\beta_+} = \sum_{[\beta_+] \in \NE_{\bbN}(X_+)_{/b}} g_{+,[\beta_+]} Q_{X_+}^{\beta_{+,\min}}
$$
where $c_{\beta_+} \in \bbC$ and $g_{+,[\beta_+]} \in \bbC\formal{Q_{X_+}^b}$, and a general element of $\bbC\formal{Q_{X_-}}$ has the form 
$$
    \sum_{\beta_- \in \NE_{\bbN}(X_-)}c_{\beta_-} Q_{X_-}^{\beta_-} = \sum_{[\beta_-] \in \NE_{\bbN}(X_-)_{/a}} g_{-,[\beta_-]} Q_{X_-}^{\beta_{-,\min}}
$$
where $c_{\beta_-} \in \bbC$ and $g_{-,[\beta_-]} \in \bbC\formal{Q_{X_-}^a}$.

Writing $r = r_+ = r_-$, we define the basic rational function
$$
    h(x) \coloneqq \frac{x}{1-(-1)^r x}
$$
which is defined for $x \in (\bbC \cup \{\infty\}) \setminus \{(-1)^r\}$ and satisfies that $h(x) + h(x^{-1}) = (-1)^{r+1}$.
We define the following subrings of $\bbC\formal{Q_{X_+}}$ and $\bbC\formal{Q_{X_-}}$ respectively: 
\begin{align*}
    & \bbC\formal{Q_{X_+}}^h \coloneqq \left\{ \sum_{[\beta_+] \in \NE_{\bbN}(X_+)_{/b}} g_{+,[\beta_+]} Q_{X_+}^{\beta_{+,\min}} \ \bigg|\  g_{+,[\beta_+]} \in \bbC[Q_{X_+}^b, h(Q_{X_+}^b)] \right\}, \\
    & \bbC\formal{Q_{X_-}}^h \coloneqq \left\{ \sum_{[\beta_-] \in \NE_{\bbN}(X_-)_{/a}} g_{-,[\beta_-]} Q_{X_-}^{\beta_{-,\min}} \ \bigg|\  g_{-,[\beta_-]} \in \bbC[Q_{X_-}^a, h(Q_{X_-}^a)]\right\}.
\end{align*}
In particular, the inclusions $\bbC\formal{Q_{X_{\pm}}}^h\subset \bbC\formal{Q_{X_{\pm}}}$ are induced by the Taylor series expansion of $h(x)$ at $x=0$.
Moreover, we invert the variables $Q_{X_+}^b$ and $Q_{X_-}^a$ in the above and introduce the following version of Laurent extensions:
\begin{align*}
    & R_+ \coloneqq \left\{ \sum_{[\beta_+] \in \NE_{\bbN}(X_+)_{/b}} g_{+,[\beta_+]} Q_{X_+}^{\beta_{+,\min}} \ \bigg|\  g_{+,[\beta_+]} \in \bbC[Q_{X_+}^{\pm b}, h(Q_{X_+}^b)] \right\}, \\
    & R_- \coloneqq \left\{ \sum_{[\beta_-] \in \NE_{\bbN}(X_-)_{/a}} g_{-,[\beta_-]} Q_{X_-}^{\beta_{-,\min}} \ \bigg|\  g_{-,[\beta_-]} \in \bbC[Q_{X_-}^{\pm a}, h(Q_{X_-}^a)]\right\}.
\end{align*}

For the flop $f\colon X_- \dashrightarrow X_+$, the correspondence given by the graph closure of $f$ induces an isomorphism 
$$
    \phi\colon H^*(X_+) \longrightarrow H^*(X_-)
$$ 
on cohomology groups that preserves the Poincar\'e pairings \cite[Theorem 0.1]{lee2010flops} (see Proposition~\ref{prop:Chow motive decom of flips}, Corollary~\ref{cor:deRhamIsoVGIT}). Moreover, the correspondence induces a monoid isomorphism
$$
    \varphi \colon \NE_{\bbN}(X_+)+\bbZ b \longrightarrow \NE_{\bbN}(X_-)+\bbZ a
$$
such that $\varphi(b) = -a$, which coincides with the isomorphism $\varphi$ in Lemma~\ref{lem:KappaDualPlusFactor} (see Remark~\ref{rem:CurveCorrespondence}). 
\color{black}
The isomorphism $\varphi$ induces an isomorphism between $R_+$ and $R_-$ that identifies $Q_{X_+}^{\beta_+}$ with $Q_{X_-}^{\varphi(\beta_-)}$. In particular, $Q_{X_+}^b$ is identified with $Q_{X_-}^{-a}$, under which we have $h(Q_{X_+}^b) + h(Q_{X_-}^{-a}) = (-1)^{r+1}$. Therefore, the correspondence induces isomorphisms
$$
    R_+\formal{\tau_{X_+}} \cong R_-\formal{\tau_{X_-}}, \quad
    H^*(X_+)\otimes_{\bbC}R_+\formal{\tau_{X_+}} \cong H^*(X_-)\otimes_{\bbC}R_-\formal{\tau_{X_-}}.
$$

The main result of Lee, Lin, Qu, and Wang \cite{lee2010flops, lee2016ordinaryflops1, lee2016ordinaryflops2, lee2016ordinaryflops3} may be summarized as follows.

\begin{proposition}[\cite{lee2010flops, lee2016ordinaryflops1, lee2016ordinaryflops2, lee2016ordinaryflops3}]\label{Prop: big quantum ring identify in flop case}
The quantum products $\star_{\tau_{X_{\pm}}}$ of $X_\pm$ are well-defined on
$$
    H^*(X_{\pm})\formal{Q_{X_{\pm}}}^h\formal{\tau_{X_{\pm}}} = H^*(X_{\pm})\otimes_{\bbC}\bbC\formal{Q_{X_{\pm}}}^h\formal{\tau_{X_{\pm}}}
$$
respectively. In addition, after extending linearly over $R_+\formal{\tau_{X_+}} \cong R_-\formal{\tau_{X_-}}$, the quantum products $\star_{\tau_{X_+}}$ and $\star_{\tau_{X_-}}$ are identified on $H^*(X_+)\otimes_{\bbC}R_+\formal{\tau_{X_+}} \cong H^*(X_-)\otimes_{\bbC}R_-\formal{\tau_{X_-}}$.
\end{proposition}

We now reformulate Proposition~\ref{Prop: big quantum ring identify in flop case} in terms of quantum $D$-modules. Define 
$$
    \QDM(X_{\pm})^h \coloneqq H^*(X_{\pm})[z]\formal{Q_{X_{\pm}}}^h\formal{\tau_{X_{\pm}}} 
$$
and the Laurent extended version
$$
    \QDM(X_{\pm})^{h,\La} \coloneqq H^*(X_{\pm})[z]\otimes_{\bbC} R_{\pm}\formal{\tau_{X_{\pm}}}.
$$
Note that both $\QDM(X_+)^{h,\La}$ and $\QDM(X_-)^{h,\La}$ are defined over the common base ring
\begin{equation}\label{eqn:R0}
    R^0 \coloneqq \bbC[z]\otimes_{\bbC}R_-\formal{\tau_{X_-}}\cong \bbC[z]\otimes_{\bbC}R_+\formal{\tau_{X_-}}.
\end{equation}
By Proposition~\ref{Prop: big quantum ring identify in flop case}, the quantum connection $\nabla$ is well-defined on $\QDM(X_{\pm})^h$, which is naturally extended to $\QDM(X_{\pm})^{h,\La}$. We have the following statement which is the counterpart of Theorem~\ref{thm:GeneralQDMDecomposition}.


\begin{theorem}
\label{thm:QDMDecompCrep}
Consider the setup of Section~\ref{sect:GeneralVGITFlips} and suppose $r_+ = r_-$. The correspondence induces a map $\tau_{X_+}(\tau_{X_-}) \in H^*(X_+)$ and an $R^0$-module isomorphism 
$$
    \Phi^{0} \colon\QDM(X_-)^{h,\La}\longrightarrow \tau_{X_+}^*\QDM(X_+)^{h,\La}.
$$
that intertwines the quantum connections and the Poincar\'e pairings.
\end{theorem}

\begin{remark}
In the case $r_+ = r_-$, one may attempt to establish the isomorphism of quantum $D$-modules following the proof strategy of Theorem~\ref{thm:GeneralQDMDecomposition}, that is, via a reduction to the 3-component $\bbC^*$-case and the framework of Section~\ref{sect:Decomposition}. The main issue with this approach is that, 
when $c_{F_0} = r_+ - r_- = 0$ (in the notation of Section~\ref{sect:Decomposition}), the degree of the variable $S_{F_0} = Q_W^{-a}S$ is zero. 
We are thus unable to use the approach of Section~\ref{sect:Decomposition} to unify the base of the quantum $D$-modules of $X_\pm$. For instance, the rings $\bbC\formal{C_{X_{\pm},\bbN}^{\vee}}$ do not fit into a natural common base ring, and if we take the completion of $\QDM_{\bbC^*}(W)$ with respect to $X_-$, the Fourier transformation to $X_+$ cannot be extended to the completion (see \cite[Proposition 5.4]{iritani2023quantum}).
The advantage of the approach of \cite{lee2010flops, lee2016ordinaryflops1, lee2016ordinaryflops2, lee2016ordinaryflops3} in this case is that, the quantum products are shown to be analytic in the directions of $Q_{X_+}^b$ and $Q_{X_-}^a$ respectively. Moreover, the base change $\tau_{X_+}(\tau_{X_-})$ is shown to not involve any nontrivial Novikov variables.
\end{remark}

\subsection{Decomposition of quantum $D$-modules for standard flips: local model}\label{sect:FlipLocalModel} 
For the rest of Section~\ref{sect:General}, we discuss the application of Theorem~\ref{thm:GeneralQDMDecomposition} to the decomposition of quantum $D$-modules for standard flips (Conjecture~\ref{conj:IntroGeneralFlip}). In this subsection, we first consider the local model.

Let $V_\pm$ be vector bundles over a smooth projective variety $S$ with ranks $r_\pm$ respectively such that $r_+ < r_-$. Let $\psi_\pm \colon \bbP(V_\pm) \to S$ denote the projections. We define the (projective) local model
$$ 
    X_{\pm,\loc} = \bbP_{\bbP(V_{\pm})}(\psi_\pm^*(V_\mp) (-1) \oplus \cO).
$$

We realize $X_{-,\loc}\dashrightarrow X_{+,\loc}$ as a simple VGIT wall-crossing by a construction similar to \cite[Section 3.1]{shen2025quantum}.
Let
$$
    W \coloneqq V_+\oplus V_-\oplus \cO
$$
which is the total space of a vector bundle on $S$ with rank $r_+ + r_- + 1$. We introduce an action of a rank-$2$ torus 
$$
    G = \bbC^* \times \bbC^*
$$
on $W$ where the first $\bbC^*$-factor scales the fibers of the three direct summands with weights $(1,0,1)$ and the second $\bbC^*$-factor scales the fibers with weights $(0,-1,-1)$. 
Now, let $L_S$ be an ample line bundle on $S$ and consider its pullback $\pi^*L_S$ to $W$ under the projection  $\pi\colon W\to S$. The following two characters
$$
    \chi_{\pm} \colon \bbC^* \times \bbC^* \longrightarrow \bbC^*, \quad \chi_+(t_1, t_2) = t_1, \quad \chi_-(t_1,t_2) = t_2^{-1}
$$ 
of $G$ induce two linearizations on $\pi^*L$ respectively. More precisely, the actions of $G$ induced by $\chi_{\pm}$ are respectively defined by 
$$
    g\cdot (x,a_{\pi(x)})\longmapsto (g\cdot x, \chi_{\pm}(g)a_{\pi(x)})
$$
for any $x\in W,\; a_{\pi(x)}\in (\pi^*L)_x=L_{\pi(x)}$, and $g\in G$.. Moreover, the two $G$-ample linearizations induce a simple VGIT wall-crossing $W\gitquot G(-)\dashrightarrow W\gitquot G(+)$, with 
$$
    X_{\pm, \loc} = W \gitquot G(\pm).
$$
Hence, Theorem~\ref{thm:GeneralQDMDecomposition} can be applied to give a decomposition of quantum $D$-modules for $X_{\pm, \loc}$.

\begin{corollary}\label{cor:qdmdec for local model}
Conjecture~\ref{conj:IntroGeneralFlip} holds for the local model $X_{-,\loc}\dashrightarrow X_{+,\loc}$.
\end{corollary}

\subsection{Decomposition of quantum $D$-modules for standard flips: case of $D$-flips}\label{sect:GeneralFlips}
Let $f \colon X_- \dashrightarrow X_+$ be a standard flip of type $(r_+-1, r_--1)$ with $r_+ \le r_-$ and with wall $S$. Assume that $X_\pm$ and $S$ are smooth and projective.
In this subsection, we consider when $X_\pm$ can be realized directly as a simple VGIT wall-crossing of a smooth, quasi-projective variety $W$ equipped with a $G$-action. When $r_+ < r_-$, together with Theorem~\ref{thm:GeneralQDMDecomposition}, this would provide an alternative approach to Conjecture~\ref{conj:IntroGeneralFlip}. In the case $r_+ = 1$ of blowups, Example~\ref{ex:Blowup} gives a simple construction with $W$ smooth, projective and $G = \bbC^*$. We therefore focus on the case $r_+ \ge 2$, where the contractions $q_\pm$ are \emph{small} in the sense that the exceptional loci have codimension at least 2.

Following \cite{thaddeus1996geometric}, we assume that $f \colon X_- \dashrightarrow X_+$ is a \emph{$D$-flip} in the sense that $X_0$ is normal and there exists a $\bbQ$-Cartier divisor $D$ on $X_-$ that satisfies the following conditions:
\begin{enumerate}[label=(\roman*), wide]
    \item The line bundle $\cO(-D)$ is relatively ample over $X_0$ and induces the contraction $q_-\colon X_- \to X_0$.

    \item The pushforward divisor $f_*(D)$ on $X_+$ is $\bbQ$-Cartier.
    
    \item The line bundle $\cO(f_*(D))$ is relatively ample over $X_0$ and induces the contraction $q_+\colon X_+ \to X_0$.
\end{enumerate}
In this case, as indicated in \cite{thaddeus1996geometric} as a converse to Theorem 3.3 there, $X_\pm$ may be realized as an instance of $\bbC^*$-VGIT wall-crossing for some variety $W$, which may be constructed by applying the construction in the proof of \cite[Proposition 1.7]{thaddeus1996geometric} (due to Reid) affine locally on $X_0$ and gluing the local pieces. In more detail, taking any affine chart $U_0 = \Spec(R_0)$ of $X_0$ and denoting the preimage $q_{\pm}^{-1}(U_0)$ by $U_{\pm}$, we assign an affine chart $\Spec(R)$ of $W$ where 
$$
    R = \bigoplus_{n\in\bbZ} R_n, \quad R_n \coloneqq H^0(X_-,\cO(nD)|_{U_-}).
$$
The algebra $R$ is equipped with a natural $\bbZ$-grading by $n$.
Because $q_\pm$ are both small contractions, the flip $f\colon X_-\dashrightarrow X_+$ induces identifications 
$$
    H^0(U_-,\cO(nD)|_{U_-}) = H^0(U_0,\cO(nq_{-,*}(D))|_{U_0}) =  H^0(U_+,\cO(nf_*(D))|_{U_+})
$$
for any $n\in \bbZ$. In particular, the degree-0 part of $R$ is indeed $R_0$.
As $U_0$ varies across $X_0$, the algebra $R$ glues to a quasi-coherent sheaf of $\bbZ$-graded algebras, and $W$ is the relative $\Spec$. The $\bbZ$-grading induces a $\bbC^*$-action on $W$ under which the two nontrivial GIT quotients are $W\gitquot G(\pm) = X_\pm$. In addition, we have
$$
    W^{\bbC^*} = W^0 = W^0 \gitquot G(0) = S.
$$



In general, the variety $W$ hereby constructed is not necessarily quasi-projective or smooth. We remark that, in the construction of the master space in Section~\ref{sect:MasterSpace}, the quasi-projectivity condition is only used to guarantee that the GIT quotients $X_\pm = W \gitquot G(\pm)$ are good quotients. This indeed holds in our present situation, and thus the master construction is still applicable.

It would be interesting to understand when the variety $W$ constructed above is smooth, which would imply the conditions \eqref{cond:t0}--\eqref{cond:WtGcd} in Definition~\ref{def:SimpleWall}. By \cite[Theorem 1.11]{yeung2019homological}, $W$ has at worst Gorenstein singularities. Below, we give a sufficient condition for the smoothness of $W$ and the condition \eqref{cond:WtPm1} in Definition~\ref{def:SimpleWall}.


\begin{lemma}\label{lem:DFlipSmoothVGIT}
Let $f \colon X_- \dashrightarrow X_+$ be a $D$-flip as above and assume without loss of generality that $D$ and $f_*(D)$ are Cartier. Suppose the following conditions are met on any affine chart $U_0 = \Spec(R_0)$ of $X_0$:
\begin{enumerate}[label=(\roman*), wide]
    \item \label{ass:Deg1Generation} 
    The graded algebras
    $$
        R_{\ge 0} \coloneqq \bigoplus_{n \ge 0} R_n, \quad R_{\le 0} \coloneqq \bigoplus_{n \le 0} R_n
    $$
    are respectively generated by $R_1$ and $R_{-1}$ over $R_0$.
    
    \item \label{ass:IdealPrime} The ideal $I=\langle R_{\ne 0}\rangle$ in $R$ generated by all homogeneous elements with nonzero degree is prime.
    
    
    \item\label{ass:RestrictToO1} The restriction of the line bundle $\cO(-D)$ to $U_- \cap P_- \cong (U_0 \cap S) \times \bbP^{r_--1}$ is isomorphic to the pullback of $\cO_{\bbP^{r_--1}}(1)$ to the product. (Notice that $U_0\cap S$ is affine.)
\end{enumerate}
Then, the variety $W$ constructed above is smooth and the condition~\eqref{cond:WtPm1} in Definition~\ref{def:SimpleWall} is satisfied for the wall $S$.
\end{lemma}

\begin{proof}
We prove the lemma affine locally. The $\bbC^*$-fixed locus of the chart $U_0$ is $U_0 \cap S$ and is defined by the prime ideal $I$. Note that condition~\ref{ass:Deg1Generation} implies that $I$ is generated by $R_{\pm 1}$. We define $I_0 \coloneqq I \cap R_0$, which is an ideal in $R_0$. 
Because $I$ is prime, we have $\Spec(R_0/I_0) = \Spec(R/I) = U_0 \cap S$ which is smooth and
$$
    U_+ = \Proj (R_{\ge 0}), \quad U_+ \cap P_+ = \Proj (R_{\ge 0}/I_0R_{\ge 0}), 
$$
$$
    U_- = \Proj (R_{\le 0}),  \quad U_- \cap P_- = \Proj (R_{\le 0}/I_0R_{\le 0}).
$$

We first show that $R$ is regular at every prime ideal $\fp \notin V(I)$. In this case, we can find $s \in R_1$ or $s \in R_{-1}$ such that $\fp \in D(s)=\Spec(R_s)$. Without loss of generality, we assume that $s \in R_1$. Because $R$ is generated by $R_{\pm 1}$ as an $R_0$-algebra, we have 
$$
    R_s =  R_{(s)}[s,s^{-1}] = (R_{\ge 0})_{(s)}[s,s^{-1}]
$$
where $R_{(s)}$ denotes the subalgebra of degree-$0$ elements in $R_s$, and similarly for $(R_{\ge 0})_{(s)}$.
However, $\Spec((R_{\ge 0})_{(s)})$ is an affine chart of $U_+$ and is thus smooth. Therefore, $R_s$ is regular, which implies that $R_{\fp}$ is also regular.

It remains to show that $R$ is regular at any $\fp \in V(I)$, or equivalently, $\Spec(R/I)\to \Spec (R)$ is a regular immersion. In the rest of the proof, we show that
$$
    \bigoplus_{k\ge 0} I^k/I^{k+1} \cong \Sym^\bullet (I/I^2)
$$
in several steps of computation.

To begin with, because $\cO(-D)|_{U_-}  = \widetilde{R_{\le 0}[1]}$ on $U_-$, its restriction to $U_- \cap P_-$ is $\widetilde{ R_{\le 0}/I_0R_{\le 0} [1]}$. By condition~\ref{ass:RestrictToO1}, there exist $y_1, \dots, y_{r_-} \in R_{-1}$ such that the set of global sections of $\cO(-D)|_{U_- \cap P_-}$ is the free $R_0/I_0$-module
$$
    R_{-1}/I_0R_{-1} \cong R_0/I_0\{y_1,\cdots,y_{r_-}\},
$$
where we abusively use $y_i$ to also denote its image in the quotient. In other words, $y_1, \dots, y_{r_-}$ provide linear coordinates on the $\bbP^{r_--1}$-factor of $U_- \cap P_- \cong (U_0 \cap S) \times \bbP^{r_--1}$.
Again, by condition~\ref{ass:Deg1Generation}, $R_{-1}/I_0R_{-1}$ generates $R_{\le 0}/I_0R_{\le 0}$ as an $R_0/I_0$-algebra. By dimension considerations, there are no algebraic relations among the generators $y_1,\cdots,y_{r_-}$, and
$$
    R_{\le 0}/I_0R_{\le 0} \cong R_0/I_0[y_1,\cdots,y_{r_-}].
$$
Similarly, there exist $x_1, \dots, x_{r_+} \in R_{1}$ such that
$$
    R_{\ge 0}/I_0R_{\ge 0} \cong R_0/I_0[x_1,\cdots,x_{r_+}].
$$


Next, we compute $\bigoplus_{k\ge 0} I_0^k R /I_0^{k+1} R$.
Recall that $P_-\hookrightarrow X_-$ is a regular immersion with normal bundle $\psi_-^*(N_+)(-1)$. Thus the conormal sheaf of $U_-\cap P_-$ in $U_-$, which is $\widetilde{I_0R_{\le0}/I_0^2R_{\le0}}$, is locally free and isomorphic to the pullback of $\cO_{\bbP^{r_--1}}(1)^{r_+}$ from the $\bbP^{r_--1}$-factor. By taking global sections, we see that $I_0/I_0^2$ is a free $R_0/I_0$-module of rank $r_+ \cdot r_-$. On the other hand, condition~\ref{ass:Deg1Generation} implies that the $r_+ \cdot r_-$ elements $\{x_iy_j\}_{1 \le i \le r_+, 1 \le j \le r_-}$ generate $I_0/I_0^2$ as an $R_0/I_0$-module. Since $R_0/I_0$ is a Noetherian ring, we have
\begin{equation}\label{eqn:I0I0Square}
    I_0/I_0^2 \cong R_0/I_0 \{x_iy_j \ |\ 1\le i \le r_+, 1\le j\le r_-\}.
\end{equation}
Now we compute the module $I_0R_{\le0}/I_0^2R_{\le0}$, and more generally $I_0^kR_{\le0}/I_0^{k+1}R_{\le0}$ for any $k \ge 0$, by considering the localizations with respect to the linear coordinates $y_1, \dots, y_{r_-}$ (for the $\bbP^{r_--1}$-factor).
For each $j = 1, \dots, r_-$, the module $(I_0R_{\le0}/I_0^2R_{\le0})_{(y_j)}$ is a free $(R_{\le0}/I_0R_{\le0})_{(y_j)}$-module of rank $r_+$.
Moreover, Equation \eqref{eqn:I0I0Square} implies that is generated by the $r_+$ elements $\{x_iy_j\}_{1 \le i \le r_+}$. It follows that
$$
    (I_0R_{\le0}/I_0^2R_{\le0})_{(y_j)} = (R_{\le0}/I_0R_{\le0})_{(y_j)} \{x_iy_j \ \big|\ 1 \le i \le r_+\}.
$$
Writing out the $R_0/I_0$-algebra structure of $(R_{\le0}/I_0R_{\le0})_{(y_j)}$ explicitly, we have
$$
    (I_0R_{\le0}/I_0^2R_{\le0})_{(y_j)} \cong R_0/I_0\big[y_ly_j^{-1} \ \big|\ l \ne j\big] \big\{x_iy_j \ \big|\ 1 \le i \le r_+\big\}.
$$
Since the immersion of $U_- \cap P_- = \Proj (R_{\le 0}/I_0R_{\le 0})$ into $U_- = \Proj (R_{\le 0})$ is regular, on the affine charts we have
\begin{align*}
    \bigoplus_{k\ge 0} (I_0^k R_{\le0}/I_0^{k+1} R_{\le0})_{(y_j)} & \cong \Sym^\bullet_{(R_{\le0}/I_0R_{\le0})_{(y_j)}} (I_0R_{\le0}/I_0^2R_{\le0})_{(y_j)}\\ 
    &= R_0/I_0 \big[y_ly_j^{-1}, x_iy_j \ \big|\ l \ne j, 1 \le i \le r_+\big].
\end{align*} 
Combining the computation on the affine charts, we have that for any $k \ge 0$ and $n \le 0$,
$$
    I_0^kR_n/I_0^{k+1}R_n \cong R_0/I_0 \Big\{x_1^{a_1} \dots x_{r_+}^{a_{r_+}}y_1^{b_1}\dots y_{r_-}^{b_{r_-}} \ \Big|\ a_i, b_j \ge 0, a_1+\dots+a_{r_+}=k, b_1+\dots +b_{r_-} = k-n \Big\}.
$$
On the other hand, for $n \ge 0$, a similar computation gives
$$
    I_0^kR_n/I_0^{k+1}R_n \cong R_0/I_0 \Big\{x_1^{a_1} \dots x_{r_+}^{a_{r_+}}y_1^{b_1}\dots y_{r_-}^{b_{r_-}} \ \Big|\ a_i, b_j \ge 0, a_1+\dots+a_{r_+}=k+n, b_1+\dots +b_{r_-} = k \Big\}.
$$
Combining the two directions, we have 

$$
    \bigoplus_{k\ge 0}I_0^kR/I_0^{k+1}R\cong R_0/I_0[x_1,\dots,x_{r_+}, y_1,\dots, y_{r_-}].
$$

Finally, we compute $\bigoplus_{k \ge 0} I^k/I^{k+1}$. 
We use the shorthand notation $(I^k)_n \coloneqq I^k \cap R_n$. 
We claim that for any $k \ge 0$ and $n \in \bbZ$, we have
\begin{equation}\label{eqn:IkComputation}
    (I^k)_n = \begin{cases}
        (I^{k+1})_n & \text{if $k < \abs{n}$ or $k=\abs{n}+2t+1$ for some $t \in \bbZ_{\ge 0}$,}\\
        I_0^t R_n & \text{if $k=\abs{n}+2t$ for some $t \in \bbZ_{\ge 0}$.}
    \end{cases}
\end{equation}
The case $k < \abs{n}$ directly follows from condition~\ref{ass:Deg1Generation}. When $k=\abs{n}+2t+1$, 
\color{black}
for any $z=z_1\cdots z_k \in (I^k)_n$ such that $z_i \in I$ is homogeneous of degree $a_i$, since $k$ and $n$ have different parities, at least one $a_i$ is even. This means that $z_i\in (I)_{a_i}\subset I^2$ and $z\in (I^{k+1})_n$.
Now consider the remaining case $k=\abs{n}+2t$. We prove the claim for $n \ge 0$ and the case $n \le 0$ is similar.
Notice that $(I^{n+2t})_n/(I^{n+2t+1})_n$ can be generated by 
$$
    \Big\{x_1^{a_1}x_2^{a_2}\dots x_{r_+}^{a_{r_+}}y_1^{b_1}\dots y_{r_-}^{b_{r_-}} \ \Big|\ a_1+\dots+a_{r_+}=n+t, b_1+\dots+b_{r_-}=t \Big\},
$$
which means that 
$$
    (I^{n+2t})_n = I_0^t R_n + (I^{n+2t+1})_n = I_0^t R_n+(I^{n+2t+2})_n.
$$
It then suffices to show that $(I^{n+2t+2})_n \subseteq I_0^tR_n$. By condition~\ref{ass:Deg1Generation}, any element in $(I^{n+2t+2})_n$ can be written as the sum of elements of form $z=z_1\cdots z_{k+2t+w}\in (I^{k+2t+2})_k$ with $w \ge 2$ and $z_i\in I_0 \cup R_{- 1} \cup R_1$. Let $a_\pm$ and $a_0$ denote the number of $z_i$'s that have degree $\pm 1$ and $0$ respectively. Then $a_++a_-+a_0=k+2t+2w$ and $a_+-a_-=k$ imply that $2a_-+a_0=2t+2w$ and hence, $2a_-+2a_0\ge 2t+2w \ge2t+4$. It follows that $a_- + a_0 \ge t+2$ and thus $z \in I_0^{t+2}R_k \subset I_0^tR_k$.

By \eqref{eqn:IkComputation}, for any $k \ge 0$ and $n \in \bbZ$, we have
$$
    (I^k/I^{k+1})_n \coloneqq (I^k)_n/(I^{k+1})_n = \begin{cases}
        0 & \text{if $k < \abs{n}$ or $k=\abs{n}+2t+1$ for some $t \in \bbZ_{\ge 0}$,}\\
        I_0^t R_n/I_0^{t+1}R_n & \text{if $k=\abs{n}+2t$ for some $t \in \bbZ_{\ge 0}$.}
    \end{cases}
$$
In particular, for $k = 1$, we see that
$$
    I/I^2 = R_1/I_0R_1 \oplus R_{-1}/I_0R_{-1} \cong R/I\{x_1, \dots, x_{r_+}, y_1,\cdots,y_{r_-}\}
$$
where recall $R_0/I_0 \cong R/I$. Therefore, we have
$$
    \bigoplus_{k\ge 0} I^k/I^{k+1} \cong \bigoplus_{k\ge 0} I_0^kR/I_0^{k+1}R \cong R/I[x_1,\dots,x_{r_+}, y_1,\dots,y_{r_-}] \cong \Sym^\bullet (I/I^2)
$$
as desired. Finally, note that the condition~\eqref{cond:WtPm1} of Definition~\ref{def:SimpleWall} is also a direct consequence of condition~\ref{ass:Deg1Generation}.
\end{proof}



\begin{corollary}\label{cor:FlipQDMDecomposition}    
For a standard flip $f \colon X_-\dashrightarrow X_+$ of type $(r_+-1,r_--1)$ with $r_+< r_-$, if $f$ is a $D$-flip that satisfies the conditions in Lemma~\ref{lem:DFlipSmoothVGIT}, Conjecture~\ref{conj:IntroGeneralFlip} holds.
\end{corollary}

\bibliographystyle{plain}  
\bibliography{dahema} 

\end{document}